\documentclass[reqno,11pt]{amsart}

\usepackage{amssymb}
\usepackage[utf8]{inputenc}
\usepackage[T1]{fontenc}
\usepackage[letterpaper,hmarginratio=1:1]{geometry}
\usepackage{tikz-cd}

\usepackage{amsmath,amsthm,mathrsfs} 
\usepackage{marvosym}
\usepackage{xcolor}
\usepackage{colonequals}
\usepackage{enumitem}
\usepackage{mathdots}
\usepackage{framed}
\usepackage{url}
\usepackage{xypic}
\usepackage{multicol}
\usepackage{fancyhdr}
\usepackage{nccmath}
\usepackage{color}
\usepackage{xparse}
\usepackage{bbm}
\usepackage[unicode,colorlinks]{hyperref}
\usepackage{microtype}

     \definecolor{dark-red}{rgb}{0.54,0,0}
     \definecolor{dark-green}{rgb}{0,0.54,0}
     \definecolor{dark-magenta}{rgb}{0.54,0,0.54}
     \definecolor{dark-cyan}{rgb}{0,0.54,0.54}
     \hypersetup{%
         unicode=true,
         breaklinks=true,
         pdfcreator=,
         pdfproducer=,
         pdfstartview =, 
         pdfpagelayout = ,
         pdfhighlight = /O,
         linkcolor=dark-red, 
         linkbordercolor=dark-red, 
         anchorcolor=red, 
         citecolor=dark-green, 
         citebordercolor=dark-green, 
         filecolor=dark-cyan, 
         filebordercolor=dark-cyan, 
         menucolor=dark-red, 
         menubordercolor=dark-red, 
         runcolor=dark-cyan, 
         runbordercolor=dark-cyan, 
         urlcolor=dark-cyan, 
         urlbordercolor=dark-cyan, 
}

\newcommand{\Affine}{\mathbb{A}}

\newcommand\FF{\protect\mathbb{F}}

\newcommand\NN{\protect\mathbb{N}}
\newcommand\QQ{\protect\mathbb{Q}}

\newcommand\ZZ{\protect\mathbb{Z}}

\newcommand\GG{\protect\mathbb{G}}

\newcommand\sF{\mathscr{F}}

\newcommand\cA{\mathcal{A}}

\newcommand\cG{\mathcal{G}}

\newcommand\cO{\mathcal{O}}

\newcommand\cW{\mathcal{W}}

\DeclareMathOperator\GL{GL}
\DeclareMathOperator\Hom{Hom}
\DeclareMathOperator\Tr{Tr}

\DeclareMathOperator\Gal{Gal}

\DeclareMathOperator\Ind{Ind}

\newcommand\id{\text{id}}

\newcommand{\from}{\colon}

\theoremstyle{theorem} \newtheorem{proposition}{Proposition}[section]
\theoremstyle{definition} \newtheorem{definition}[proposition]{Definition}
\theoremstyle{theorem} \newtheorem{lemma}[proposition]{Lemma}
\theoremstyle{remark} \newtheorem{remark}[proposition]{Remark}
\theoremstyle{remark} 
\theoremstyle{remark} 
\theoremstyle{definition} 
\theoremstyle{definition} 
\theoremstyle{theorem} \newtheorem*{displaytheorem}{Theorem}
\theoremstyle{theorem} \newtheorem{theorem}[proposition]{Theorem}
\theoremstyle{theorem} \newtheorem{corollary}[proposition]{Corollary}
\theoremstyle{definition} 
\theoremstyle{theorem} 
\theoremstyle{remark} 
\theoremstyle{definition} 
\theoremstyle{definition} 
\theoremstyle{definition} 
\theoremstyle{definition} 
\theoremstyle{remark} \newtheorem*{claim*}{Claim}
\theoremstyle{remark} 
\theoremstyle{theorem} 
\theoremstyle{theorem} \newtheorem{conjecture}[proposition]{Conjecture}
\theoremstyle{definition} 
\theoremstyle{definition} 
\theoremstyle{theorem} \newtheorem{mytheorem}{Theorem}
\theoremstyle{remark} \newtheorem{importantremark}[proposition]{Important Remark}

\theoremstyle{remark}

\newcommand\ur{\text{nr}}

\newcommand\OH{\mathcal{O}}

\newcommand\Khat{\widehat{K}^{\ur}}
\newcommand\TT{\mathbb{T}}
\newcommand\UU{\mathbb{U}}
\newcommand\BB{\mathbb{B}}

\DeclareMathOperator\Fr{Fr}

\newcommand\EFF{\mathscr{F}}

\DeclareMathOperator\pr{pr}

\newcommand\Loc{\mathcal{L}}

\newcommand\F{\FF_{q^2}}
\newcommand\Unip{U_h^{2,q}}
\newcommand\UnipF{\Unip(\F)}

\newcommand\ONE{\mathbbm{1}}

\title[Deligne-Lusztig Constructions for Division Algebras]{Deligne-Lusztig Constructions for Division Algebras and the Local Langlands Correspondence}
\author{Charlotte Chan}
\thanks{This work was partially supported by NSF grants DMS-0943832 and DMS-1160720.}

\begin{document}

\begin{abstract}
Let $K$ be a local non-Archimedean field of positive characteristic and let $L$ be the degree-$n$ unramified extension of $K$. Let $\theta$ be a smooth character of $L^\times$ such that for each nontrivial $\gamma \in \Gal(L/K)$, $\theta$ and $\theta/\theta^\gamma$ have the same level. Via the local Langlands and Jacquet-Langlands correspondences, $\theta$ corresponds to an irreducible representation $\rho_\theta$ of $D^\times$, where $D$ is the central division algebra over $K$ with invariant $1/n$.

In 1979, Lusztig proposed a cohomological construction of supercuspidal representations of reductive $p$-adic groups analogous to Deligne-Lusztig theory for finite reductive groups. In this paper we prove that when $n = 2$, the $p$-adic Deligne-Lusztig (ind-)scheme $X$ induces a correspondence $\theta \mapsto H_\bullet(X)[\theta]$ between smooth one-dimensional representations of $L^\times$ and representations of $D^\times$ that matches the correspondence given by the LLC and JLC.
\end{abstract}

\maketitle

\tableofcontents

\section{Introduction}\label{s:introduction}

Deligne-Lusztig theory gives a geometric description of the irreducible representations of finite groups $G$ of Lie type. In \cite{DL76}, Deligne and Lusztig introduce certain locally closed subvarieties $X$ in flag varieties over finite fields and prove that the irreducible representations of $G$ occur with multiplicity one in the $\ell$-adic \'etale cohomology groups $H_c^i(X, \overline \QQ_\ell)$. The variety $X$ has an action by $G \times T$, where $T$ is a maximal torus in $G$, and the $G$-representations can (more or less) be picked off by considering the subspaces $H_c^i(X, \overline \QQ_\ell)[\theta]$ of the cohomology groups where $T$ acts by some character $\theta$. 

In \cite{L79}, Lusztig suggests an analogue of Deligne-Lusztig theory for $p$-adic groups $G$. He introduces a certain infinite-dimensional variety $X$ which has a natural action of $G \times T$, and defines $\ell$-adic homology groups $H_i(X)$ respecting this action. One can then study the correspondence $\theta \mapsto H_i(X)[\theta]$ between characters $\theta$ of $T$ and representations of $G$ arising from the subspace $H_i(X)[\theta]$ of $H_i(X)$ on which $T$ acts by some character $\theta$.

Consider the following set-up. Let $K$ be a local non-Archimedean field of positive odd characteristic and let $L \supset K$ be the unramified extension of degree $n$. In the situation that $T = L^\times$ and $G = D_{1/n}^\times$, where $D_{1/n}$ is the central division algebra over $K$ with invariant $1/n$, the local Langlands and Jacquet-Langlands correspondences (LLC and JLC) give rise to a correspondence between characters of $T$ and representations of $T$. Indeed: To a smooth character $\theta \from L^\times \to \overline \QQ_\ell^\times$, one can associate a smooth irreducible $n$-dimensional representation $\sigma_\theta$ of the Weil group $\cW_K$ of $K$, which corresponds to an irreducible supercuspidal representation $\pi_\theta$ of $\GL_n(K)$ (via LLC), which finally corresponds to an irreducible representation $\rho_\theta$ of $D_{1/n}^\times$ (via JLC).

The main theorem of this paper is:

\begin{mytheorem}[Rough Formulation]\label{t:roughLLC}
Let $n = 2$. For a broad class of characters $\theta \from L^\times \to \overline \QQ_\ell^\times$, there exists an $r$ (dependent on $\theta$) such that
\begin{equation*}
H_i(X)[\theta] = \begin{cases}
\rho_\theta & \text{if $i = r$,} \\
0 & \text{otherwise.}
\end{cases}
\end{equation*}
\end{mytheorem}

In pictorial form, we have
\begin{equation*}
\begin{tikzcd}[column sep=tiny,row sep=small]
{\theta} \arrow[mapsto]{ddd}[left]{\text{$p$-adic Deligne-Lusztig}} & & {\theta} \arrow[mapsto]{d} & & {\mathfrak X} \arrow{d} \\
{} & & {\sigma_\theta} \arrow[mapsto]{d} & & {\cG_K(n)} \arrow{d}{\text{ Local Langlands}} \\
{} & & {\pi_\theta} \arrow[mapsto]{d} & & {\cA_K(n)} \arrow{d}{\text{ Jacquet-Langlands}} \\ 
H_\bullet(X)[\theta] & \cong & {\rho_\theta} & & {\cA'{}_K(n)}
\end{tikzcd}
\end{equation*}
where
\begin{align*}
\mathfrak X &\colonequals \{\text{characters $L^\times \to \overline \QQ_\ell^\times$ with trivial $\Gal(L/K)$-stabilizer}\} \\
\cG_K(n) &\colonequals \{\text{smooth irreducible dimension-$n$ representations of the Weil group $\cW_K$}\} \\
\cA_K(n) &\colonequals \{\text{supercuspidal irreducible representations of $\GL_n(K)$}\} \\
\cA'{}_K(n) &\colonequals \{\text{smooth irreducible representations of $D_{1/n}^\times$}\}
\end{align*}

\subsection{What is Known}

In \cite{B12}, Boyarchenko presents a method for explicitly calculating the representations $H_i(X)[\theta]$ and does so for a special class of characters $\theta$ in the case when $G$ is the multiplicative group of the central division algebra with Hasse invariant $1/n$ over a local field $K$, and $T = L^\times$, where $L$ is the unramified degree-$n$ extension of $K$. The approach is to reduce the computation to a problem involving certain finite unipotent groups and then develop a ``Deligne-Lusztig theory'' for these groups.

Before we continue, we must introduce some terminology. Let $D_{1/n}$ denote the central division algebra with Hasse invariant $1/n$ over $K = \FF_q((\pi))$ for $q$ a $p$-power and let $L = \FF_{q^n}((\pi))$. The \textit{level} of a smooth character $\theta \from L^\times \to \overline \QQ_\ell^\times$ is the smallest integer $h$ such that $\theta$ is trivial on $U_L^h \colonequals 1 + \pi^h \cO_L \subset \cO_L^\times$, where $\cO_L$ is the ring of integers of $L$. The set of characters of $L^\times$ has a natural action by $\Gal(L/K)$. We say that $\theta$ is \textit{primitive} if for any $\gamma \in \Gal(L/K)$, both $\theta$ and $\theta/\theta^\gamma$ have the same level. (Equivalently, $\theta$ is primitive if its restriction to $U_L^{h-1}$ has trivial $\Gal(L/K)$-stabilizer.)

In Section \ref{s:introduction} we recall the unipotent situation established by Boyarchenko in \cite{B12}. We describe a unipotent group scheme $U_h^{n,q}$ over $\FF_p$ together with a subscheme $X_h \subset U_h^{n,q}$ that comes with a left action by $U_L^1/U_L^h$ and a right action by $U_h^{n,q}(\FF_{q^n})$. This unipotent group depends on three parameters that are determined by the set-up in the following way. If $\theta$ is a character of level $h$, then the computation of the eigenspaces $H_i(X)[\theta]$ for $D_{1/n}^\times$ over $K = \FF_q((\pi))$ will reduce to a computation of $H_c^i(X_h, \overline \QQ_\ell)[\chi]$ for $U_h^{n,q}(\FF_{q^n})$, where $\chi$ is the character of $U_L^1/U_L^h$ induced by $\theta$. To be completely clear, the three parameters $n,$ $q,$ and $h$ correspond respectively to the Hasse invariant of the division algebra, the size of the residue field of $K$, and the level of $\theta$.

In \cite{BW11}, Boyarchenko and Weinstein give a complete description of the $U_2^{n,q}(\FF_{q^n})$-representations $H_c^i(X_2, \overline \QQ_\ell)[\chi]$. They prove the following

\begin{displaytheorem}[Boyarchenko and Weinstein]
Given a character $\chi \from U_L^1/U_L^2 \to \overline \QQ_\ell^\times$, there exists a unique $r$ such that $H_c^i(X_2, \overline \QQ_\ell)[\chi]$ vanishes when $i \neq r$ and is an irreducible $U_2^{n,q}(\FF_{q^n})$-representation when $i = r$. Furthermore, every irreducible representation of $U_2^{n,q}(\FF_{q^n})$ occurs with multiplicity one in $\bigoplus_{i \in \ZZ} H_c^i(X, \overline \QQ_\ell)$.
\end{displaytheorem}

It turns out that the scheme $X_2$ is very closely related to certain open affinoid of the Lubin-Tate tower (see \cite{BW11}), and in \textit{op.\ cit.}\ Boyarchenko and Weinstein use the above theorem to give a purely local proof of the local Langlands and Jacquet-Langlands correspondences for a broad class of supercuspidals (those whose Weil parameters are induced from a primitive character of an unramified degree-$n$ extension). In \cite{BW13}, Boyarchenko and Weinstein use this result to give a geometric realization of the local Langlands and Jacquet-Langlands correspondences in this class of supercuspidals.

The analogue of the above theorem of Boyarchenko and Weinstein for $U_h^{n,q}(\FF_{q^n})$ when $h > 2$, however, was almost completely unknown. Indeed, the only higher level situation known was the case $h = 3$ and $n = 2$, which Boyarchenko computed in \cite{B12} (see Theorem 5.20 of \cite{B12}). The following conjecture appears in \cite{B12}.


\begin{conjecture}[Boyarchenko]\label{c:5.18}
Given a character $\chi \from U_L^1/U_L^h \to \overline \QQ_\ell^\times$, there exists $r \geq 0$ such that $H_c^i(X_h, \overline \QQ_\ell)[\chi]$ vanishes when $i \neq r$ and is an irreducible $U_h^{n,q}(\FF_{q^n})$-representation when $i = r$.
\end{conjecture}

Assuming Conjecture \ref{c:5.18} holds, Boyarchenko gives a complete description of the $D_{1/n}^\times$-representations $H_i(X)[\theta]$ based on the $U_h^{n,q}(\FF_{q^n})$-representations $H_c^i(X_h)[\chi]$ (see Proposition 5.19 of \cite{B12}). Thus, the problem of determining the representations $H_i(X)[\theta]$ arising from Lusztig's $p$-adic analogue of Deligne-Lusztig varieties, depends only on the Deligne-Lusztig theory of the finite unipotent group $U_h^{n,q}(\FF_{q^n})$.

\begin{remark}
The varieties constructed in \cite{L79} are not the affine Deligne-Lusztig varieties. In \cite{I13}, Ivanov shows that there are no nontrivial morphisms from the cohomology of the affine Deligne-Lusztig to $D_{1/2}^\times$-representations of level $> 1$. This is not true for the varieties in \cite{L79} associated to division algebras (see Theorem \ref{t:introLLC}).
\end{remark}

\subsection{Outline of this Paper}

In this paper, we prove these two conjectures when $n = 2$ and $\chi \from U_L^1/U_L^h \to \overline \QQ_\ell^\times$ has the property that its restriction $\psi \colonequals \chi|_{U_L^{h-1}/U_L^h}$ has trivial $\Gal(\F/\FF_q)$-stabilizer. (In this situation, we say that $\psi$ has conductor $q^2$.) Using Proposition 5.19 of \cite{B12}, we can then describe the representations $H_i(X)[\theta]$ for primitive $\theta$ of arbitrary level. 

Let $\cA_\psi$ denote the set of such $\chi$ and let $\cG_\psi$ denote the set of irreducible representations of $\UnipF$ restricting to a multiple of $\psi$. In Section \ref{s:repthy}, we prove

\begin{mytheorem}\label{t:introbij}
There exists a bijection
\begin{align*}
\cA_\psi \longleftrightarrow \cG_\psi, \qquad \chi \mapsto \rho_\chi.
\end{align*}
\end{mytheorem}

Using an explicit description of this bijection, we prove a certain character formula in Section \ref{s:morerepthy} that plays a crucial role in Section \ref{s:cohomreps}.

In Section \ref{s:homspace}, we prove that there are no nontrivial morphisms from $\rho_\chi$ to $H_c^i(X_h, \overline \QQ_\ell)$ if $i \neq h-1$. This allows us to apply a variant of a Deligne-Lusztig fixed point formula (see Lemma 2.13 of \cite{B12}) in order to compute subspaces of intertwines in $H_c^i(X_h, \overline \QQ_\ell)$. These computations, done in Section \ref{s:cohomreps}, allow us to prove

\begin{mytheorem}\label{t:introdesc}
The cohomology groups $H_c^i(X_h, \overline \QQ_\ell)[\chi]$ vanish when $i \neq h-1$ and
\begin{equation*}
H_c^{h-1}(X_h, \overline \QQ_\ell)[\chi] \cong \rho_\chi.
\end{equation*}
\end{mytheorem}

In Section \ref{s:examples}, we show how to carry out the arguments of Section \ref{s:morerepthy}, \ref{s:homspace}, and \ref{s:cohomreps} in the special case $h = 3$. This allows us to illustrate the structure and flavor of the proofs in a simpler setting. It also gives a different proof of Theorem 5.20 of \cite{B12}.

It is worth noting here that Theorem \ref{t:introdesc} is stronger than Conjecture 2; it requires considerably more work to prove $H_c^{h-1}(X_h, \overline \QQ_\ell) \cong \rho_\chi$ than to prove its irreducibility (compare the proofs of Theorems \ref{t:irrep} and \ref{t:interdim}). Because we have an explicit description of the $\UnipF$-representations $H_c^i(X_h, \overline \QQ_\ell)[\chi]$, we can use Proposition 5.19 of \cite{B12} to explicitly describe the $D_{1/2}^\times$-representations $H_i(X, \overline \QQ_\ell)[\theta]$. The final theorem in this paper, whose rough formulation was stated in the main introduction (see Theorem \ref{t:roughLLC}), compares the correspondence
\begin{equation*}
\theta \mapsto H_i(X, \overline \QQ_\ell)[\theta]
\end{equation*}
to known correspondences between characters of $L^\times$ and representations of division algebras.

We can now formulate Theorem \ref{t:roughLLC} more precisely.

\begin{mytheorem}\label{t:introLLC}
Let $\theta \from L^\times \to \overline \QQ_\ell^\times$ be a primitive character of level $h$ and let $\rho_\theta$ be the $D_{1/2}^\times$ representation corresponding to $\theta$ under the local Langlands and Jacquet-Langlands correspondences. Then $H_i(X, \overline \QQ_\ell)[\theta] = 0$ for $i \neq h-1$, and
\begin{equation*}
H_{h-1}(X, \overline \QQ_\ell)[\theta] \cong \rho_\theta.
\end{equation*}
\end{mytheorem}

\subsection{Acknowledgements}

I am deeply grateful to Mitya Boyarchenko for introducing me to this area of research, and also for his support and advice. His work on Deligne-Lusztig constructions for $p$-adic groups inspired most of the ideas of this paper. I am also extremely thankful for his careful comments, corrections, and suggestions on early drafts of this paper.

\subsection{Notation and Set-Up}\label{s:notation}

Let $K = \FF_q((\pi))$ and let $L$ be a degree-2 unramified extension of $K$. We will work with the algebraic group $\Unip$, the higher unipotent group described in \cite{B12} and \cite{BW11}. This group has a natural filtration
\begin{equation*}
\{1\} \subset H_{2(h-1)} \subset H_{2(h-1) - 1} \subset \cdots \subset H_2 \subset H_1 = \Unip,
\end{equation*}
where $H_k \colonequals \{1 + \sum a_i \tau^i : i \geq k\}$. We will also make use of the subgroup
\begin{equation*}
H \colonequals \{1 + \sum a_i \tau^i : \text{$i$ is even}\}.
\end{equation*}

We will work with a subscheme $X_h \subset \Unip$ that is defined in \cite{B12}. We restate this here. For an $\FF_p$-algebra $A$ and $a_1, \ldots, a_{2(h-1)} \in A$, we will associate a matrix
\begin{equation*}
\iota_h(1 + \sum a_i \tau^i) \colonequals \left(\begin{matrix}
1 + a_2 \pi + a_4 \pi^2 + \cdots & a_1 + a_3 \pi + a_5 \pi^2 + \cdots \\
a_1^q \pi + a_3^q \pi^2 + \cdots & 1 + a_2^q \pi + a_4^q \pi^2 + \cdots
\end{matrix}\right),
\end{equation*}
which determines a map $\iota_h$ from $\Unip(A)$ to the set of matrices over $A[[\pi]]/(\pi^h)$. The $p$-adic Deligne-Lusztig construction $X$ described in \cite{L79} can be identified with a certain set $\widetilde X$ described in \cite{B12}, which has an ind-scheme structure given by
\begin{equation*}
\widetilde X = \bigsqcup_{m \in \ZZ} \varprojlim_h \widetilde X_h^{(0)}.
\end{equation*}
By construction (see \textit{op.\ cit.}\ for details), $\iota_h(1 + \sum a_i \tau^i)$ is in the $A$-points $\widetilde X_h^{(0)}(A)$ of $\widetilde X_h^{(0)}$ if and only if its determinant is fixed by $\Fr_q$. We define $X_h \subset \Unip$ to be $\iota_h^{-1}(\widetilde X_h^{(0)})$.

The map $\iota_h$ has the following property, which we will refer to as Property $\ddagger$. If $A$ is an $\F$-algebra, then $\iota_h(xy) = \iota_h(x) \iota_h(y)$ for all $x \in \Unip(A)$ and all $y \in \UnipF$. Moreover, for $y \in \UnipF$, we have $\det(y) \in \FF_q[\pi]/(\pi^h)$. It therefore follows that $X_h$ is stable under right-multiplication by $\UnipF$. We denote by $x \cdot g$ the action of $g \in \UnipF$ on $x \in X_h$.

We now describe a left action of $H(\F)$ on $X_h$. We can identify $H(\F)$ with the set $\iota_h(H(\F))$. Note that by Property $\ddagger$, the map $\iota_h$ actually preserves the group structure of $H(\F)$. Since $\iota_h$ is injective, then we in fact have a group isomorphism $H(\F) \cong \iota_h(H(\F))$. Explicitly, this isomorphism is given by
\begin{equation*}
1 + \sum a_{2i} \tau^{2i} \mapsto \left(\begin{matrix} 1 + \sum a_{2i} \pi^i & 0 \\ 0 & 1 + \sum a_{2i}^q \pi^i \end{matrix}\right).
\end{equation*}
We already observed that $\det \iota_h(H(\F)) \subset \FF_q[\pi]/(\pi^h)$. Thus for any $\F$-algebra $A$, the natural left-multiplication action of $\iota_h(H(\F))$ on the set of matrices over $A[[\pi]]/(\pi^h)$ stabilizes $X_h(A)$. This defines a left action of $H(\F)$ on $X_h$.\footnote{Warning: This is not the same as the left-multiplication action of $H(\F) \subset H(A)$ on $\Unip(A)$.} We denote by $g * x$ the action of $g \in H(\F)$ on $x \in X_h$.

An observation which will be frequently used throughout this paper is that we have canonical isomorphisms
\begin{align*}
U_L^1/U_L^h &\overset{\sim}{\to} H(\F), && 1 + \sum_{i=1}^{h-1} a_i \pi^i \mapsto 1 + \sum_{i=1}^{h-1} a_i \tau^{2i} \\
\F \overset{\sim}{\to} U_L^{h-1}/U_L^h &\overset{\sim}{\to} H_{2(h-1)}(\F), && a \mapsto 1 + a \pi^{h-1} \mapsto 1 + a \tau^{2(h-1)}.
\end{align*}
Thus the left action of $H(\F)$ can be interpreted as a left action of $U_L^1/U_L^h$. Because the study of $\UnipF$ and the cohomology groups $H_c^i(X_h, \overline \QQ_\ell)$ arose because of the interest in computing the representations arising from Deligne-Lusztig constructions for division algebras, we will often refer to the left $H(\F)$-action as the left $(U_L^1/U_L^h)$-action.

\section{The Representation Theory of \texorpdfstring{$\UnipF$}{U_h^{2q}(\text{F}_{q^2})}}\label{s:repthy}

In this section, we will describe a class of irreducible representations of $\UnipF$ that are in bijection with certain characters $\chi \from U_L^1/U_L^h \to \overline \QQ_\ell^\times$. Recall that we have canonical isomorphisms $U_L^1/U_L^h \cong H(\F)$ and $\F \cong U_L^{h-1}/U_L^h \cong H_{2(h-1)}(\F)$.

The additive characters of $\FF_{q^n}$ have a natural action by $\Gal(\FF_{q^n}/\FF_q)$. We say that a character $\psi \from \FF_{q^n} \to \overline \QQ_\ell^\times$ has \textit{conductor} $m$ if its stabilizer in $\Gal(\FF_{q^n}/\FF_q)$ is $\Gal(\FF_{q^n}/\FF_{q^m})$. In this paper we will only work with the case when $n = 2$ and only work with characters $\psi \from \F \to \overline \QQ_\ell^\times$ that have conductor $2$. In this case, this just means that there exists some $x \in \F$ such that $\psi(x^q) \neq \psi(x)$.

Let $\cA_\psi$ denote the set of all characters $\chi \from U_L^1/U_L^h \to \overline \QQ_\ell^\times$ whose restriction to $U_L^{h-1}/U_L^h$ is equal to $\psi$. Let $\cG_\psi$ denote the set of irreducible representations of $\UnipF$ wherein $H_{2(h-1)}(\F)$ acts via $\psi$. In this section, we will prove the following theorem (see Propositions \ref{p:bijodd} and \ref{p:bijeven}):

\begin{theorem}\label{t:bij}
If $\psi$ has conductor $q^2$, then there exists a bijection between $\cA_\psi$ and $\cG_\psi$. Furthermore, every rep of $\cG_\psi$ has dimension $q^{h-1}$.
\end{theorem}

The first subgroup of importance is the following:
\begin{equation*}
H_0' \colonequals \{1 + \sum a_i \tau^i : \text{$i = 2(h-1)$ OR $i > (h-1)$ is odd}\} \subset \Unip.
\end{equation*}
For an additive character $\psi$ of $\F$, define the character $\widetilde \psi$ of $H_0'(\F)$ as 
\begin{equation*}
\widetilde \psi \from H_0'(\F) \to \overline \QQ_\ell^\times, \qquad 1 + \sum a_i \tau^i \mapsto \psi(a_{2(h-1)}).
\end{equation*} 

\begin{lemma}\label{l:psitilde}
Let $\psi$ be an additive character of $\F$ with conductor $q^2$. If $\rho$ is an irreducible representation of $\UnipF$ where $H_{2(h-1)}(\F)$ acts by $\psi$, then the restriction of $\rho$ to $H_0'(\F)$ must contain $\widetilde \psi$.
\end{lemma}

\begin{proof}
We prove this inductively. Let
\begin{align*}
G_1 &\colonequals \{1 + a_{2(h-1) - 1} \tau^{2(h-1) - 1} + a_{2(h-1)} \tau^{2(h-1)}\}, \\
G_2 &\colonequals \{1 + a_{2(h-1) - 3} \tau^{2(h-1) - 3} + a_{2(h-1) - 1} \tau^{2(h-1) - 1} + a_{2(h-1)} \tau^{2(h-1)}\}, \\
&\; \vdots \\
G_{\lfloor (h-1)/2 \rfloor} &\colonequals H_0'.
\end{align*}

Since
\begin{equation*}
1 + a_{2(h-1) - 1} \tau^{2(h-1) - 1} + a_{2(h-1)} \tau^{2(h-1)} = (1 + a_{2(h-1) - 1} \tau^{2(h-1) - 1})(1 + a_{2(h-1)} \tau^{2(h-1)}),
\end{equation*}
then every extension of $\psi$ to $G_1(\F)$ is of the form
\begin{equation*}
1 + a_{2(h-1) - 1} \tau^{2(h-1) - 1} + a_{2(h-1)} \tau^{2(h-1)} \mapsto \nu(a_{2(h-1) - 1}) \psi(a_{2(h-1)})
\end{equation*}
for some additive character $\nu$ of $\F$. Let $\psi_1$ denote the extension of $\psi$ to $G_1(\F)$ given by 
\begin{equation*}
\psi_1(1 + a_{2(h-1)-1} \tau^{2(h-1)-1} + a_{2(h-1)} \tau^{2(h-1)}) \colonequals \psi(a_{2(h-1)}).
\end{equation*}
For $g_1 = 1 - b_1 \tau$ and $h = 1 + \sum a_i \tau^i \in G_1(\F)$, we have
\begin{equation*}
{}^{g_1} \psi_1(h) = \psi_1(g_1 h g_1^{-1}) = \psi_1(h) \psi(b_1 a_{2(h-1)-1}^q - b_1^q a_{2(h-1) - 1}).
\end{equation*}
Since $\psi$ has conductor $q^2$, every character of $\F$ is of the form $y \mapsto \psi(xy^q - x^q y)$ for some $x \in \F$. Thus for any additive character $\nu$ of $\F$, there exists an $g_1$ such that ${}^{g_1} \psi_1(h) = \psi_1(h) \nu(a_{2(h-1)-1}).$ We may therefore conclude that the restriction of $\rho$ to $G_1(\F)$ contains $\psi_1$.

We now work on extending $\psi_1$ to $G_2(\F)$. Since
\begin{equation*}
h = h_0(1 + a_{2(h-1) - 3} \tau^{2(h-1) - 3}),
\end{equation*}
where $h \in G_2(\F)$ and $h_0 \in G_1(\F)$, then every extension of $\psi_1$ to $G_2(\F)$ is of the form
\begin{equation*}
h \mapsto \nu(a_{2(h-1) - 3}) \psi_1(h_0),
\end{equation*}
where as before, $\nu$ is some additive character of $\F.$ Let $\psi_2$ denote the extension of $\psi_1$ to $G_2(\F)$ given by 
\begin{equation*}
\psi_2(1 + \sum a_i \tau^i) \colonequals \psi_1(1 + a_{2(h-1)-1} \tau^{2(h-1)-1} + a_{2(h-1)} \tau^{2(h-1)}) = \psi(a_{2(h-1)}).
\end{equation*}
For $g_2 = 1 - b_3 \tau^3$ and $h = 1 + \sum a_i \tau^i \in G_2(\F)$, we have
\begin{equation*}
{}^{g_2} \psi_2(h) = \psi_2(g_2 h g_2^{-1}) = \psi_2(h) \psi(b_3 a_{2(h-1)-3}^q - b_3^q a_{2(h-1) - 3}).
\end{equation*}
As before, this shows that the restriction of $\rho$ to $G_2(\F)$ contains $\psi_2$.

Continuing this for each $G_i$, we see that the conclusion of the Lemma holds.
\end{proof}

Now consider the subgroup $H'(\F) \subset \UnipF$ defined as follows:
\begin{equation*}
H'(\F) \colonequals \begin{cases}
\{1 + \sum a_i \tau^i : \text{$i$ is even OR $i > h-1$ is odd}\} & \text{if $h$ is odd}, \\
\{1 + \sum a_i \tau^i : \text{$i$ is even OR $i \geq h-1$ is odd; $a_{h-1} \in \FF_q$}\} & \text{if $h$ is even}.
\end{cases}
\end{equation*}
Recall that
\begin{equation*}
H_0'(\F) \colonequals \{1 + \sum a_i \tau^i : \text{$i = 2(h-1)$ OR $i > h-1$ is odd}\} \subset \UnipF
\end{equation*}
and in the case that $h$ is even, define
\begin{equation*}
H_1'(\F) \colonequals \{1 + \sum a_i \tau^i : \text{$i = 2(h-1)$ OR $i \geq h-1$ is odd; $a_{h-1} \in \FF_q$}\} \subset \UnipF.
\end{equation*}
The behavior of the representation theory of $\UnipF$ depends on the parity of $h$. At the core of this distinction is the following.

\begin{lemma}\label{l:centpsi}
Let $\psi$ be an additive character of $\F$ with conductor $q^2$. Then
\begin{enumerate}[label=(\alph*)]
\item
If $g \in H'(\F)$, then $gag^{-1} \in H_0'(\F)$ and $\widetilde \psi(gag^{-1}) = \widetilde \psi(a)$ for all $a \in H_0'(\F)$.

\item
Let $h$ be odd. If $g \notin H'(\F)$, then there exists an $a \in H_0'(\F)$ such that $gag^{-1} \in H_0'(\F)$ but $\widetilde \psi(gag^{-1}) \neq \widetilde \psi(a)$.

\item
Let $h$ be even and let $\theta$ be any extension of $\widetilde \psi$ to $H_1'(\F)$. If $g \notin H'(\F)$, then there exists an $a \in H_1'(\F)$ such that $gag^{-1} \in H_1'(\F)$ but $\theta(gag^{-1}) \neq \theta(a).$
\end{enumerate}
\end{lemma}

\begin{proof}
We handle (a) first. Consider $a = 1 + \sum a_i \tau^i \in H_0'(\F)$ and $g = 1 + \sum b_i \tau^i \in H'(\F)$. Then write $gag^{-1} = 1 + \sum c_i \tau^i$. It is clear that if in fact $g \in H_0'(\F)$, then $\widetilde \psi(gag^{-1}) = \widetilde \psi(a).$ 

When $h$ is odd, we can write $g = g' g''$ where $g' \in H(\F)$ and $g'' \in H_0'(\F)$. Thus all that remains to show is that if $g \in H(\F)$, then $\widetilde \psi(gag^{-1}) = \widetilde \psi(a).$ Now, $g = 1 + \sum b_i \tau^i$ has the property that $b_i = 0$ for $i$ odd. Since $a_i = 0$ for all $i$ even, with the exception of when $i = 2(h-1)$, we see that the only contribution of $g$ to the product $gag^{-1}$ occurs in $c_i$ for $i$ odd and $i > h-1$. Thus $gag^{-1} \in H_0'(\F)$ and since the changes only occur in the odd coefficients, we have that $\widetilde \psi(gag^{-1}) = \widetilde \psi(a).$

When $h$ is even, we can write $g = g' g''$ where $g' \in H(\F)$ and $g'' \in H_1'(\F)$. If $g = 1 + b \tau^{h-1} \in H_1'(\F)$, then the coefficient of $\tau^{2(h-1)}$ in $gag^{-1}$ is equal to $ba_{h-1}^q - a_{h-1}b^q + a_{2(h-1)} = a_{2(h-1)}$ since $a_{h-1}, b \in \FF_q$. It follows that $\widetilde \psi$ is centralized by $g$. By the same argument as in the previous paragraph, we see that $H(\F)$ centralizes $\widetilde \psi$, and this completes the proof that $H'(\F)$ centralizes $\widetilde \psi$.

We now show (b) and (c). Suppose that $g = 1 + \sum b_i \tau^i \in \UnipF \smallsetminus H'(\F)$. Let $r$ be the smallest odd integer such that $b_r \neq 0$. By assumption, $r \leq h-1$. We may write $g = (1 + b_r \tau^r)g'$, where the coefficient of $\tau^r$ in $g'$ vanishes. 

First assume that $r < h-1$ and consider $a = 1 + a_{2(h-1)-r} \tau^{2(h-1)-r} \in H_0'(\F)$. Note that $gag^{-1} \in H_0'(\F)$. We have $\widetilde \psi(g' a (g')^{-1}) = \widetilde \psi(a)$, so
\begin{align*}
\widetilde \psi(gag^{-1}) 
&= \widetilde \psi((1 + b_r \tau^r)(1 + a_{2(h-1)-r} \tau^{2(h-1) - r})(1 - b_r \tau^r)) \\
&= \widetilde \psi(1 + a_{2(h-1)-r} \tau^{2(h-1)-r} + (b_r a_{2(h-1)-r}^q - b_r^q a_{2(h-1)-r}) \tau^{2(h-1)}) \\
&= \widetilde \psi(a) \psi(b_r a_{2(h-1)-r}^q - b_r^q a_{2(h-1)-r}),
\end{align*}
Since $\psi$ has conductor $q^2$, then it follows that we can find $a_{2(h-1)-r}$ such that the above $\neq \widetilde \psi(a)$. Thus we have shown the desired conclusion with the assumption that $r < h-1$.

If $h$ is odd, then $g \in \UnipF \smallsetminus H'(\F)$ forces $r < h-1$, and thus (b) follows from the above. Now let $h$ be even and let $\theta$ be an arbitrary extension of $\widetilde \psi$ to $H_1'(\F)$. By the above, we see that if $r < h-1$, then there exists an element $a \in H_1'(\F)$ such that $gag^{-1} \in H_1'(\F)$ but $\theta(gag^{-1}) \neq \theta(a).$ Therefore all that remains to show is the case when $r = h-1$. Consider $a = 1 + a_{h-1} \tau^{h-1} \in H_1'(\F)$. Then
\begin{equation*}
\theta(gag^{-1}) = \theta(a) \psi(b_{h-1} a_{h-1}^q - b_{h-1}^q a_{h-1}) = \theta(a) \psi(a_{h-1}(b_{h-1} - b_{h-1}^q)),
\end{equation*}
where the last equality holds since $a_{h-1} \in \FF_q$. Since $g \notin H'(\F)$, we have $b_{h-1} \notin \FF_q$, then we see that we arrange for $a_{2(h-1)-r} \in \FF_q$ to be such that the above $\neq \theta(a)$ (since $\psi$ has conductor $q^2$ and thus its restriction to $\ker \Tr_{\F/\FF_q}$ is nontrivial). This proves (c).
\end{proof}

\begin{corollary}\label{c:irrep}
If $\nu$ is any extension of $\widetilde \psi$ to $H'(\F)$, then the representation $\Ind_{H'(\F)}^{\UnipF}(\nu)$ is irreducible.
\end{corollary}

\begin{proof}
By Lemma \ref{l:centpsi}(b), we can apply Mackey's criterion and conclude irreducibility.
\end{proof}

In Section \ref{s:odd} and \ref{s:even}, we will describe all such extensions $\nu$. We will also analyze the following question: Given distinct extensions $\nu$ and $\nu'$ of $\widetilde \psi$ to $H'(\F)$, when are the representations $\Ind_{H'(\F)}^{\UnipF}(\nu)$ and $\Ind_{H'(\F)}^{\UnipF}(\nu')$ isomorphic? To begin answering this question, we will need the following lemma. In the next two subsections, we will give a complete answer.

\begin{lemma}\label{l:normH'}
The normalizer of $H'(\F)$ in $\UnipF$ is equal to the subgroup
\begin{equation*}
K \colonequals \{1 + \sum a_i \tau^i : \text{$i$ is even OR $i \geq h-2$ is odd}\} \subseteq \UnipF.
\end{equation*}
Note that when $h$ is odd, then $H'(\F)$ is an index-$q^2$ subgroup of $K$, and when $h$ is even, then $H'(\F)$ is an index-$q$ subgroup of $K$.
\end{lemma}

\begin{proof}
Let $K$ be as in the statement of the lemma.

Let $h$ be odd. To show that $K$ normalizes $H'(\F)$, we need only show that for $b = 1 + \sum b_i \tau^i \in H'(\F)$ and $g = 1 - a_{h-2} \tau^{h-2}$, we have $gbg^{-1} \in H'(\F)$. But this is clear since $g$ only contributes to the coefficients of $\tau^i$ for $i \geq (h-2)+2 > h-1$.

Let $h$ be even. To show that $K$ normalizes $H'(\F)$, we need only show that for $b = 1 + \sum b_i \tau^i \in H'(\F)$ and $g = 1 - a_{h-1} \tau^{h-1}$, we have $gbg^{-1} \in H'(\F)$. Again, this is clear since $g$ only contributes to the coefficients of $\tau^i$ for $i \geq (h-1)+2 > h-1$.

Thus all that remains to show is that no other elements of $\UnipF$ normalize $H'(\F)$. Consider $g = 1 + \sum a_i \tau^i \in \Unip \smallsetminus H'(\F)$. Let $r$ be the smallest odd integer such that $a_r \neq 0$. Then we may write $g = (1 + a_r \tau^r) g'$ where the coefficient of $\tau^r$ in $g'$ vanishes. Note that by assumption $r < h-2$. Let $s$ be the largest even integer such that $1 + a \tau^{r+s} \notin H'(\F)$ for $a \notin \FF_q$. (If $h$ is even, then $s = h-1 - r$, and if $h$ is odd, then $s = h - r$.) Then for any $b \in \F$, $x \colonequals (g')^{-1}(1 + b \tau^s)g' \in H'(\F)$ and 
\begin{equation*}
gxg^{-1} = (1+a_r \tau^r)(1 + b \tau^s)(1-a_r \tau^r + \cdots) = 1 + b \tau^s + (a_r b^q - a_r b) \tau^{s+r} + \cdots
\end{equation*}
In particular, we can pick $b \notin \FF_q$, and this implies $gxg^{-1} \notin H'(\F)$. Thus we have shown that $K$ is contains the normalizer of $H'(\F)$, and this completes the proof.
\end{proof}

\subsection{Case: $h$ odd}\label{s:odd}

Recall that we have
\begin{equation*}
H' \colonequals \{1 + \sum a_i \tau^i : \text{$i$ is even; or $i > h-1$ and $i$ is odd}\}.
\end{equation*}
For $\chi \in \cA_\psi$, consider the character on $H'(\F)$ defined as
\begin{equation*}
\chi^\sharp(1 + \sum a_i \tau^i) = \chi(1 + a_2 \pi + \cdots + a_{2(h-1)} \pi^{h-1}).
\end{equation*}

\begin{lemma}\label{l:extodd}
Let $\psi$ be an additive character of $\F$ of conductor $q^2$. If $\nu$ is an extension of $\widetilde \psi$ to $H'(\F)$, then $\nu = \chi^\sharp$ for some $\chi \in \cA_\psi$. Moreover, 
\begin{equation*}
\Ind_{H_0'(\F)}^{H'(\F)}(\widetilde \psi) \cong \bigoplus_{\chi \in \cA_\psi} \chi^\sharp.
\end{equation*}
\end{lemma}

\begin{proof}
The maximum number of extensions of $\widetilde \psi$ to $H'(\F)$ is equal to the index of $H_0'(\F)$ in $H'(\F)$. That is, the maximum number of extensions is $q^{3(h-1)}/q^{h+1}=q^{2(h-1)-2}$. On the other hand, it is clear that $\chi^\sharp$ is an extension of $\widetilde \psi$ and varying $\chi \in \cA_\psi$ gives $q^{2(h-2)}$ distinct extensions $\nu$. Therefore in fact every such $\nu$ is of the form $\chi^\sharp$.
\end{proof}

\begin{lemma}\label{l:irrepodd}
Let $\psi$ be an additive character of $\F$ with conductor $q^2$, and let $\chi \in \cA_\psi$. The representation 
\begin{equation*}
\rho_\chi \colonequals \Ind_{H'(\F)}^{\UnipF}(\chi^\sharp)
\end{equation*}
is irreducible and hence $\rho_\chi \in \cG_\psi$.
\end{lemma}

\begin{proof}
Since $\chi^\sharp$ is an extension of $\widetilde \psi$ to $H'(\F)$, this lemma is just a special case of Corollary \ref{c:irrep}.
\end{proof}

\begin{lemma}\label{l:orthodd}
For $\chi_1, \chi_2 \in \cA_\psi$, we have $\rho_{\chi_1} \cong \rho_{\chi_2}$ if and only if $\chi_1 = \chi_2$.
\end{lemma}

\begin{proof}
This follows from Corollary \ref{c:mult}. 
\end{proof}

\begin{corollary}\label{c:Vpsiodd}
Let $\psi$ be a character of $\F$ with conductor $q^2$. If $\rho \in \cG_\psi$, then $\rho$ occurs with multiplicity one in the representation $V_\psi \colonequals \Ind_{H_0'(\F)}^{\UnipF}(\widetilde \psi)$.
\end{corollary}

\begin{proof}
Let $\rho \in \cG_\psi$. It follows from Lemma \ref{l:psitilde} that the restriction of $\rho$ to $H_0'(\F)$ must contain $\widetilde \psi$. Therefore $\rho$ is a direct summand of $V_\psi$. Lemma \ref{l:extodd} implies that $\Ind_{H_0'(\F)}^{H'(\F)}(\widetilde \psi) \cong \bigoplus_{\chi \in \cA_\psi} \chi^\sharp$. Thus
\begin{equation*}
\Ind_{H_0'(\F)}^{\UnipF}(\widetilde \psi) \cong \Ind_{H'(\F)}^{\UnipF}(\bigoplus_{\chi \in \cA_\psi} \chi^\sharp) \cong \bigoplus_{\chi \in \cA_\psi} \rho_\chi.
\end{equation*}
By Lemma \ref{l:irrepodd} and Lemma \ref{l:orthodd}, this a direct sum of nonisomorphic irreducible representations, and this completes the proof. 
\end{proof}

We have now shown the following.

\begin{proposition}\label{p:bijodd}
Let $\psi$ be a character of $\F$ with conductor $q^2$. There is a bijective correspondence
\begin{equation*}
\cA_\psi \longleftrightarrow \cG_\psi
\end{equation*}
and this correspondence is given by
\begin{equation*}
\chi \longleftrightarrow \rho_\chi.
\end{equation*}
Furthermore, every representation in $\cG_\psi$ has dimension $q^{h-1}$.
\end{proposition}

\begin{proof}
Injectivity follows from Lemma \ref{l:orthodd}. Surjectivity follows from Lemma \ref{c:Vpsiodd}. Thus every representation of $\cG_\psi$ is of the form $\rho_\chi$, which has dimension equal to $|\UnipF|/|H'(\F)| = q^{4(h-1)}/q^{3(h-1)} = q^{h-1}$.
\end{proof}

\subsection{Case: $h$ even}\label{s:even}

We first recall some general facts about group representations. Suppose that $G$ is a group and $H, K, N \subset G$ are subgroups such that $H = K \cdot N$. Note that if $\chi$ is a character of $K$ and $\theta$ is a character of $N$ such that $\chi = \theta$ on the intersection $K \cap N$, then the function $f(k \cdot n) \colonequals \chi(k) \theta(n)$ is well-defined. Now let $\chi$ and $\theta$ be multiplicative. If $K$ normalizes $N$ and $K$ centralizes $\theta$, then in fact 
\begin{equation*}
f(k_1n_1k_2n_2) = f(k_1k_2(k_2^{-1}n_1 k_2 n_2)) = \chi(k_1 k_2) \theta(n_1 n_2) = f(k_1k_2n_1n_2),
\end{equation*}
so $f$ is multiplicative.

We now apply the above to the situation when $K = H(\F)$, $N = H_1'(\F)$ and $H = H'(\F)$. Recall that we have 
\begin{align*}
H(\F) &\colonequals \{1 + \sum a_i \tau^i : \text{$i$ is even}\} \\
H_1'(\F) &\colonequals \{1 + \sum a_i \tau^i : \text{$i = 2(h-1)$ OR $i \geq h-1$ is odd; $a_{h-1} \in \FF_q$}\} \\
H'(\F) &\colonequals \{1 + \sum a_i \tau^i : \text{$i$ is even OR $i \geq h-1$ is odd; $a_{h-1} \in \FF_q$}\}
\end{align*}
Note that $H_1'(\F)$ is an abelian subgroup of $\UnipF$ containing $H_0'(\F)$ as an index-$q$ subgroup. Thus there are $q$ extensions of $\widetilde \psi$ to $H_1'(\F)$. Given such an extension $\theta$ and given $\chi \in \cA_\psi$, we wish to construct a character $\widetilde \chi$ of $H'(\F)$ that extends both $\chi$ and $\widetilde \psi$. (This is the analogue of $\chi^\sharp$ in the case that $h$ is odd.)

We see that $H(\F) \cap H_1'(\F) = H_{2(h-1)}(\F)$ and that $\chi$ and $\theta$ agree on this intersection. Now define
\begin{equation*}
\widetilde \chi_\theta(kn) = \chi(k) \theta(n) \qquad \text{for $k \in H(\F)$ and $n \in H_1'(\F)$}.
\end{equation*}
This is well-defined.

\begin{lemma}
The map $\widetilde \chi_\theta \from H'(\F) \to \overline \QQ_\ell^\times$ is a group homomorphism. 
\end{lemma}

\begin{proof}
It is enough to show that $H(\F)$ normalizes $H_1'(\F)$ and that $H(\F)$ centralizes $\theta$. Write $k = 1 + \sum a_i \tau^i \in H(\F)$ and $n = 1 + \sum b_i \tau^i \in H_1'(\F)$. Since the only nonzero terms of $k$ are $a_{2i} \tau^{2i}$, then $n$ only differs by $knk^{-1}$ in the coefficients of $\tau^i$ for $i$ odd and $> h-1$. Thus $H(\F)$ normalizes $H_1'(\F)$. Moreover, the coefficient of $\tau^{2(h-1)}$ in $knk^{-1}n^{-1}$ is equal to $b_{2(h-1)}$. Since $\theta$ is an extension of $\widetilde \psi$, then it follows that $H(\F)$ centralizes $\theta$. This completes the proof.
\end{proof}

\begin{lemma}\label{l:exteven}
Let $\psi$ be an additive character of $\F$ of conductor $q^2$. If $\nu$ is an extension of $\widetilde \psi$ to $H'(\F)$, then there exists a $\theta$ and $\chi \in \cA_\psi$ such that $\nu = \widetilde \chi_\theta$. Moreover,
\begin{equation*}
\Ind_{H_0'(\F)}^{H'(\F)}(\widetilde \psi) \cong \bigoplus_\theta \bigoplus_{\chi \in \cA_\psi} \widetilde \chi_\theta.
\end{equation*}
\end{lemma}

\begin{proof}
The maximum number of extensions of $\widetilde \psi$ to $H'(\F)$ is $[H'(\F) : H_0'(\F)] = q^{3(h-1)+1}/q^{h+1} = q^{2(h-1)-1}$. On the other hand, it is clear that $\widetilde \chi_\theta$ is an extension of $\widetilde \psi$ and varying $\chi$ and $\theta$ give rise to $q^{2(h-2)+1}$ distinct extension $\nu$. Therefore in fact every such $\nu$ is of the form $\widetilde \chi_\theta$.
\end{proof}

\begin{lemma} \label{l:compateven}
Let $\theta_1$ and $\theta_2$ be extensions of $\widetilde \psi$ to $H_1'(\F)$. Let $\widetilde \chi_i \colonequals \widetilde \chi_{\theta_i}$ for $i = 1,2$. Then
\begin{equation*}
\Ind_{H'(\F)}^{\UnipF}(\widetilde \chi_1) \cong \Ind_{H'(\F)}^{\UnipF}(\widetilde \chi_2).
\end{equation*}
\end{lemma}

\begin{proof}
Suppose that $\theta_1$ and $\theta_2$ are any extensions of $\widetilde \psi$ to $H_1'(\F)$. Recall that the corresponding characters $\widetilde \chi_1$ and $\widetilde \chi_2$ of $H'(\F)$ are defined as
\begin{equation*}
\widetilde \chi_i(kn) = \chi(k) \theta_i(n),
\end{equation*}
where $k \in H(\F)$ and $n \in H_1'(\F)$. Note that for any $g \in \UnipF$, we have
\begin{equation*}
gkng^{-1} = (gkg^{-1})(gng^{-1}).
\end{equation*}
Now consider the element $g = 1 - a\tau^{h-1} \in \UnipF$. Since $h$ is even, $h-1$ is odd, and $\widetilde \chi_1(gkg^{-1}) = \chi(k)$. Therefore
\begin{equation*}
{}^g \widetilde \chi_1(kn) = \chi(k) \cdot {}^g \theta_1(n).
\end{equation*}
We thus see that to show that $\widetilde \chi_1$ and $\widetilde \chi_2$ are $\UnipF$-conjugate, it suffices to show that there exists a $g = 1 - a \tau^{h-1} \in \UnipF$ such that ${}^g \theta_1 = \theta_2$. 

Now, for any $n = 1 + \sum b_i \tau^i \in H_1'(\F)$,
\begin{align*}
gng^{-1} 
&= \Big(1-a\tau^{h-1}\Big)\Big(1 + \sum_{h-1 \leq i} b_i \tau^i\Big)\Big(1+a\tau^{h-1}+a^{q+1}\tau^{h-1}\Big) \\
&= 1 + \Big(\sum_{h-1 \leq i < 2(h-1)} b_i \tau^i\Big) + \Big(b_{2(h-1)} + b_{h-1} a^q - ab_{h-1}^q\Big)\tau^{2(h-1)}.
\end{align*}
Thus
\begin{equation*}
{}^g \theta_1(n) = \theta_1(n) \psi(b_{h-1} a^q - ab_{h-1}^q).
\end{equation*}
From here, we need only show that $\#\{{}^g\theta : g = 1 - a\tau^{h-1}\} = q$, where $\theta$ is any extension of $\psi$ to $G$.

Noting that $b_{h-1} \in \FF_q$ since $b \in G$, the above computation shows that
\begin{align*}
\theta_1(gng^{-1}) = \theta_1(n) \psi(b_{h-1} a^q - ab_{h-1}).
\end{align*}
Since $\psi$ has trivial $\Gal(\F/\FF_q)$-stabilizer, then in particular it is nontrivial on $\ker \Tr_{\F/\FF_q}$. It is not difficult to show that every additive character of $\FF_q$ can be written as $b \mapsto \psi(b(a^q - a))$ for some $a \in \FF_{q^2}$. This completes the proof.
\end{proof}

\begin{lemma}\label{l:irrepeven}
Let $\psi$ be an additive character of $\F$ with conductor $q^2$, and let $\chi \in \cA_\psi$. The representation 
\begin{equation*}
\rho_\chi \colonequals \Ind_{H'(\F)}^{\UnipF}(\widetilde \chi_\theta)
\end{equation*}
is irreducible and hence $\rho_\chi \in \cG_\psi$.
\end{lemma}

\begin{proof}
This is a special case of Corollary \ref{c:irrep}.
\end{proof}

\begin{remark}
Note that by Lemma \ref{l:compateven}, the chosen extension $\theta$ does not change the representation $\Ind_{H'(\F)}^{\UnipF}(\widetilde \chi_\theta)$. This justifies the suppression of $\theta$ in the notation $\rho_\chi$ introduced in Lemma \ref{l:irrepeven}.
\end{remark}

\begin{lemma}\label{l:ortheven}
For $\chi_1, \chi_2 \in \cA_\psi$, we have $\rho_{\chi_1} \cong \rho_{\chi_2}$ if and only if $\chi_1 = \chi_2$.
\end{lemma}

\begin{proof}
From Lemma \ref{l:normH'}, we know that there are at most $q$ characters $\nu$ of $H'(\F)$ such that $\Ind_{H'(\F)}^{\UnipF}(\nu) \cong \rho_\chi$. From Lemma \ref{l:compateven}, we have found $q$ such characters, namely $\widetilde \chi_\theta$. Therefore we have $\Ind_{H'(\F)}^{\Unip}(\nu) \cong \Ind_{H'(\F)}^{\Unip}(\nu')$ if and only if there exists $\chi \in \cA_\psi$ and extensions $\theta$ and $\theta'$ of $\widetilde \psi$ to $H_1'(\F)$ such that $\nu = \widetilde \chi_\theta$ and $\nu' = \widetilde \chi_\theta'$. The desired result follows.
\end{proof}

\begin{corollary}\label{c:Vpsieven}
Let $\psi$ be a character of $\F$ with conductor $q^2$. If $\rho \in \cG_\psi$, then $\rho$ occurs with multiplicity $q$ in the representation $V_\psi \colonequals \Ind_{H_0'(\F)}^{\UnipF}(\widetilde \psi)$.
\end{corollary}

\begin{proof}
This follows from Corollary \ref{c:mult}. 
\end{proof}

We have now shown the following.

\begin{proposition}\label{p:bijeven}
Let $\psi$ be a character of $\F$ with conductor $q^2$. There is a bijective correspondence
\begin{equation*}
\cA_\psi \longleftrightarrow \cG_\psi
\end{equation*}
and this correspondence is given by
\begin{equation*}
\chi \longleftrightarrow \rho_\chi.
\end{equation*}
Furthermore, every representation in $\cG_\psi$ has dimension $q^{h-1}$.
\end{proposition}

\begin{proof}
Injectivity follows from Lemma \ref{l:ortheven}. Surjectivity follows from \ref{c:Vpsieven}. Thus every representation of $\cG_\psi$ is of the form $\rho_\chi$, which has dimension equal to $|\UnipF|/|H'(\F)| = q^{4(h-1)}/q^{3(h-1)} = q^{h-1}$.
\end{proof}

\section{A Character Formula}\label{s:morerepthy}

In this section we establish a character formula for certain representations of $\UnipF$. The main consequence of this formula is that we will be able to decompose the irreducible representations $\rho_\chi$ of $\UnipF$ as representations of the subgroup $H(\F) \subset \UnipF$. Moreover, we will show that for an additive character $\psi \from \F \to \overline \QQ_\ell^\times$ of conductor $q^2$, the elements of $\cG_\psi$ are uniquely determined by their restrictions to $H(\F)$. The character formula (Theorem \ref{t:decomp}) and its consequences (Corollaries \ref{c:mult} and \ref{c:detbyH}) play a fundamental role in Section \ref{s:cohomreps}.

We establish some notation first. Recall the subgroup $H \subset \Unip$ and, abusing notation, define
\begin{equation*}
H \colonequals H(\F) = \{1 + \sum_{i=1}^{h-1} a_{2i} \tau^{2i}\} \subset \UnipF.
\end{equation*}
We will also need the subgroups
\begin{align*}
N_k &\colonequals \{1 + \sum a_i \tau^i : \text{$i$ even OR $i > k$}\} \subset \UnipF, \\
K &\colonequals N_{h-1} \subset \UnipF.
\end{align*}
Given a character $\chi \from H \to \overline \QQ_\ell^\times$, let $\chi^\sharp$ be the character of $K$ defined in Section \ref{s:repthy} (note that $K = H_0'(\F)$). Let $\rho_\chi$ be the irreducible $\UnipF$-representation constructed in Section \ref{s:repthy} and recall that
\begin{equation*}
\Ind_K^{\UnipF}(\chi^\sharp) \cong
\begin{cases}
\rho_\chi & \text{if $h$ is odd,} \\
q \cdot \rho_\chi & \text{if $h$ is even.}
\end{cases}
\end{equation*}

Define
\begin{equation*}
G_k \colonequals \{1 + \sum a_i \tau^i \in H : \text{$a_{2i} \in \FF_q$ for $1 \leq i \leq k$}\} \subseteq H.
\end{equation*}
We define $G_0 \colonequals H.$ Note that the center of $\UnipF$ is exactly $G_{h-2}$. We thus have a tower of subgroups
\begin{equation*}
Z(\UnipF) = G_{h-2} \subset G_{h-3} \subset \cdots \subset G_1 \subset G_0 = H.
\end{equation*}

In this section, we will often write $1 + \sum a_i \tau^i = \sum a_i \tau^i$, where it is understood that $\tau^0 = 1$ and $a_0 = 1$.

The main results of this section are the following theorem and its corollaries. All proofs are in Section \ref{s:proofdecomp}.

\begin{theorem}\label{t:decomp}
Let $\chi$ be a character of $H$ whose restriction to $H_{2(h-1)}$ has conductor $q^2$. Let $\rho_\chi$ denote the irreducible representation of $\UnipF$ constructed in Section \ref{s:repthy}. Then as elements of the Grothendieck group of $H$,
\begin{equation*}
\rho_\chi = (-1)^h\left(q \cdot \chi + \sum_{i=1}^{h-2} (-1)^i (q+1) \Ind_{G_i}^H(\chi)\right).
\end{equation*}
\end{theorem}

Since $H$ is abelian, Theorem \ref{t:decomp} allows us to easily read off the decomposition of $\rho_\chi$ as a representation of $H$.

\begin{corollary}\label{c:mult}
Let $\chi$ be as in Theorem \ref{t:decomp}. Let $\cA(\chi)$ be the collection of all characters $\theta \from H \to \overline \QQ_\ell^\times$ such that, for some even $k$, $\theta$ agrees with $\chi$ on $G_{h-2-k}$ but not on $G_{h-2-k-1}$. 
\begin{enumerate}[label=(\alph*)]
\item
If $h$ is odd, then the restriction of $\rho_\chi$ to $H$ comprises
\begin{equation*}
\begin{cases}
\text{$1$ copy of $\chi$,} \\
\text{$q+1$ copies of $\theta$, for $\theta \in \cA(\chi)$.}
\end{cases}
\end{equation*}

\item
If $h$ is even, then the restriction of $\rho_\chi$ to $H$ comprises
\begin{equation*}
\begin{cases}
\text{$q$ copies of $\chi$,} \\
\text{$q+1$ copies of $\theta$, for $\theta \in \cA(\chi)$.}
\end{cases}
\end{equation*}
\end{enumerate}
\end{corollary}

An immediate consequence of Corollary \ref{c:mult} is

\begin{corollary}\label{c:detbyH}
Let $\rho$ be an irreducible representation of $\UnipF$ wherein $H_{2(h-1)}(\F)$ acts via a character $\psi$ of conductor $q^2$. Then $\rho$ is uniquely determined by its restriction to $H(\F)$.
\end{corollary}

\subsection{Proof of Theorem \ref{t:decomp} and Corollary \ref{c:mult}}\label{s:proofdecomp}

The proof of Corollary \ref{c:mult} hinges upon Theorem \ref{t:decomp}, whose proof hinges upon the following proposition. 

\begin{proposition}\label{p:trace}
Let $s \in H$.
\begin{enumerate}[label=(\alph*)]
\item
If $s \in G_{h-2}$, then
\begin{equation*}
\Tr \rho_\chi(s) = q^{h-1} \chi(s).
\end{equation*}

\item
If $s \in G_{h-2-k} \smallsetminus G_{h-2-k+1}$ for some $1 \leq k \leq h-2$, then 
\begin{equation*}
\Tr \rho_\chi(s) = (-1)^k q^{h-1-k} \chi(s).
\end{equation*}
\end{enumerate}
\end{proposition}

The organization of this section is as follows: we will prove a sequence of lemmas (Lemmas \ref{l:form} to \ref{l:sumzero}), which will allow us to prove, in quick succession, Proposition \ref{p:trace}, Theorem \ref{t:decomp}, and Corollary \ref{c:mult}.

\begin{remark}\label{r:isotypic}
The representation $\Ind_K^{\UnipF}(\chi^\sharp)$ is a sum of copies of $\rho_\chi$; it consists of $1$ copy when $h$ is odd and $q$ copies when $h$ is even. Thus, to prove Proposition \ref{p:trace}, it suffices to compute the sum
\begin{equation*}
\Tr \Ind_K^{\UnipF}(\chi^\sharp)(s) = \frac{1}{|K|} \sum_{t \in \UnipF} \chi_\circ^\sharp(tst^{-1}) \qquad \text{for $s \in H$,}
\end{equation*}
where 
\begin{equation*}
\chi_\circ^\sharp(g) = \begin{cases} \chi^\sharp(g) & \text{if $g \in K$}, \\ 0 & \text{otherwise}. \end{cases}
\end{equation*}
Since $G_{h-2}$ is the center of $\UnipF$, Proposition \ref{p:trace}(a) is easy. Lemmas \ref{l:form} through \ref{l:sumzero} build up to the proof of Proposition \ref{p:trace}(b). 
\end{remark}

\begin{remark}\label{r:equivcoord}
Note that 
\begin{equation*}
s = 1 + \sum_{i=1}^{h-1} s_{2i} \tau^{2i} \in G_{h-2-k} \smallsetminus G_{h-2-k+1}
\end{equation*}
is equivalent to the conditions
\begin{equation*}
s_{2i} \in \FF_q \quad \text{if $i \leq h-2-k$,} \qquad \text{and} \qquad s_{2(h-2-k+1)} \notin \FF_q.
\end{equation*}
\end{remark}

\begin{lemma}\label{l:form}
Every element of $\UnipF$ can be written in the form
\begin{equation*}
(1 - a_1 \tau^1)(1 - a_3 \tau^3) \cdots (1 - a_{2(h-1)-1} \tau^{2(h-1)-1}) \cdot g \qquad \text{for some $g \in H$}.
\end{equation*}
\end{lemma}

\begin{proof}
We prove this inductively. It is clear that every element of $N_{2(h-1)-2}$ can be written as 
\begin{equation*}
(1 - a\tau^{2(h-1)-1}) \cdot  g \qquad \text{for some $g \in N_{2(h-1)-1} = H$}.
\end{equation*}
We now show that every element of $N_{k-1}$ for $k$ odd can be written as
\begin{equation*}
(1-a\tau^k)\cdot g \qquad \text{for some $g \in N_k$}.
\end{equation*}
If we write
\begin{align*}
(1 - a\tau^k) \cdot g = (1 - a\tau^k) \cdot (\sum g_i \tau^i) = \sum s_i \tau^i,
\end{align*}
then we have
\begin{equation*}
\begin{cases}
s_i = g_i & \text{if $i \leq k-1$,} \\
s_i = g_i - a g_{i-k}^q & \text{if $i \geq k$ and $i$ is odd,} \\
s_i = g_i & \text{if $i \geq k+1$ and $i$ is even.}
\end{cases}
\end{equation*}
Note that in this notation, we automatically have $g_0 = s_0 = 1$. From here, we see that if we pick any $\sum s_i \tau^i \in N_{k-1}$, then we can pick an $a$ and $g_i$'s satisfying the above so that $\sum g_i \tau^i \in N_k$. Explicitly, we can let
\begin{equation*}
\begin{cases}
g_i = s_i & \text{if $i \leq k-1$,} \\
a = -s_k, \\
g_i = s_i & \text{if $i \geq k+1$, $i$ even} \\
g_i = s_i + a g_{i-k}^q & \text{if $i \geq k+2$, $i$ odd.}
\end{cases}
\end{equation*}
Note that $g_{2j}$ and $g_{2j+1}$ are defined independently and that the $g_{2j+1}$'s are defined recursively. This completes the proof.
\end{proof}

\begin{lemma}\label{l:conjU}
Let $s = \sum s_i \tau^i \in \UnipF$ and let $r$ be an odd integer with $1 \leq r \leq 2(h-1)$. Then if we let $s' = \sum s_i' \tau^i = (1-a\tau^r)s(1-a\tau^r)^{-1}$ for some $a \in \F$, we have
\begin{align*}
s_n' &= s_n + \sum_{\substack{r(l+1)+2m=n \\ l \geq 0}} -a^{q^l+q^{l-1} + \cdots + q + 1}(s_{2m}^q - s_{2m}) \\
&\qquad \qquad \qquad \qquad+ \sum_{\substack{r(l+1)+2m+1=n \\ l \geq 0}} -a^{q^{l-1} + q^{l-2} + \cdots + q + 1}(a s_{2m+1}^q - a^q s_{2m+1}).
\end{align*}
\end{lemma}

\begin{proof}
Let $b = \sum b_i \tau^i = (1-a\tau^r)^{-1}$ for $a \in \F$. Then it is easy to see that
\begin{equation*}
b_i = \begin{cases}
1 & \text{if $i = 0$,} \\
a^{q^{l-1}+q^{l-2}+\cdots+q+1} & \text{if $i = lr > 0$,} \\
0 & \text{otherwise.}
\end{cases}
\end{equation*}
This implies that
\begin{align*}
s_n' 
&= s_n + \sum_{\substack{r(l+1)+2m=n \\ l \geq 0}} -as_{2m}^qb_{lr} + s_{2m}b_{(l+1)r} + \sum_{\substack{r(l+1)+2m+1=n \\ l \geq 0}} -as_{2m+1}^qb_{lr} + s_{2m+1} b_{(l+1)r}^q \\
&= s_n + \sum_{\substack{r(l+1)+2m=n \\ l \geq 0}} -a^{q^l+q^{l-1} + \cdots + q + 1}(s_{2m}^q - s_{2m}) \\
&\qquad \qquad \qquad \qquad+ \sum_{\substack{r(l+1)+2m+1=n \\ l \geq 0}} -a^{q^{l-1} + q^{l-2} + \cdots + q + 1}(a s_{2m+1}^q - a^q s_{2m+1}).
\end{align*}
In the last step, we used the fact that $a^{q^{l+1} + \cdots + q^2} = (a^{q^{l-1} + \cdots + 1})^{q^2} = a^{q^{l-1} + \cdots + 1}$ since $a \in \F$.
\end{proof}

For the next few lemmas, we need an auxiliary definition. 

\begin{definition}
If $s \in 1 + \sum a_i \tau^i \in \UnipF$ is such that
\begin{equation}\label{e:prop*}\tag{Property $\star$}
\begin{cases}
s_{2i} \in \FF_q & \text{if $i \leq h-2-k$}, \\
s_{2i} \notin \FF_q & \text{if $i = h-1-k$}, \\
s_i = 0 & \text{if $i$ is odd and $i \leq 2(h-1)-k$},
\end{cases}
\end{equation}
then we will say that $s$ \textit{satisfies Property $\star$ for $k$}. 
\end{definition}

\begin{remark}
It is implicit in the formulation of Property $\star$ that we must have $k \leq h-2$. Thus the second condition (regarding which odd coefficients vanish) implies that $s \in K$. Note further that if $s \in H$ satisfies Property $\star$, then $s \in G_{h-2-k} \smallsetminus G_{h-1-k}$.
\end{remark}

\begin{lemma}\label{l:semicentH}
Suppose that $s \in \UnipF$ satisfies Property $\star$ for $k$. Then for any $g \in H$, the element $gsg^{-1}$ also satisfies Property $\star$ for $k$. Furthermore, if we write $gsg^{-1} = s' = \sum s_i' \tau^i$, then
\begin{equation*}
s_{2i}' = s_{2i} \qquad \text{for all $i$.}
\end{equation*}
\end{lemma}

\begin{proof}
By Lemma \ref{l:form}, we can write $s$ in the form
\begin{equation*}
(1 - a_1 \tau^1)(1 - a_3 \tau^3) \cdots (1 - a_{2(h-1)-1} \tau^{2(h-1)-1}) \cdot g' \qquad \text{for some $g' \in H$}.
\end{equation*}
By assumption, we can take $a_i = 0$ for $i \leq 2(h-1)-k$. Since $H$ is abelian, then for any $g \in H$,
\begin{equation*}
gsg^{-1} = g(1 - a_1 \tau^1)(1 - a_3 \tau^3) \cdots (1 - a_{2(h-1)-1} \tau^{2(h-1)-1})g^{-1} \cdot g'.
\end{equation*}
Thus all we need to show is that for $r > 2(h-1)-k$ odd and $a \in \F$, $g(1-a\tau^r)g^{-1}$ satisfies Property $\star$ and that the coefficients of $\tau^{2i}$ in $g(1-a\tau^r)g^{-1}$ vanish. But this is clear since the only contribution of $g$ to this conjugated element appear in the coefficients of $\tau^{r+2i}$.
\end{proof}

\begin{lemma}\label{l:semicentr}
Suppose that $s \in \UnipF$ satisfies Property $\star$ for $k$. Then for any odd $r > k$ and $a \in \F$, the element $(1-a\tau^r)s(1-a\tau^r)^{-1}$ also satisfies Property $\star$ for $k$. Furthermore, if we write $(1-a\tau^r)s(1-a\tau^r)^{-1} = s' = \sum s_i' \tau^i$, then
\begin{equation*}
s_{2i}' = s_{2i} \qquad \text{for all $i$.}
\end{equation*}
\end{lemma}

\begin{proof}
Consider the element $1 - a\tau^r \in N_k$ with $r$ odd. Combining Property $\star$ and Lemma \ref{l:conjU}, we know that if we write $s' = (1-a\tau^r)s(1-a\tau^r)^{-1}$, we have
\begin{align*}
s_n' &= s_n + \sum_{\substack{r(l+1)+2m=n \\ l \geq 0 \\ m \geq h-1-k}} -a^{q^l+q^{l-1} + \cdots + q + 1}(s_{2m}^q - s_{2m}) \\
&\qquad \qquad \qquad \qquad+ \sum_{\substack{r(l'+1)+2m+1=n \\ l' \geq 0 \\ 2m+1 \geq 2(h-1)-k+1}} -a^{q^{l'-1} + q^{l'-2} + \cdots + q + 1}(a s_{2m+1}^q - a^q s_{2m+1}).
\end{align*}
Using that $n \equiv l + 1 \pmod 2$ and $n \equiv l' \pmod 2$, the above implies, in particular,
\begin{equation}\label{e:special}
\begin{cases}
s_n' = s_n & \text{if $n$ is odd and $n \leq \min\{r + 2(h-2-k), 2r+2(h-1)-k\}$}, \\
s_n' = s_n & \text{if $n$ is even and $n \leq \min\{2r + 2(h-2-k), r + 2(h-1)-k\}$}.
\end{cases}
\end{equation}
If $k$ is odd, then by assumption $r \geq k+2$ and
\begin{align*}
\min\{r + 2(h-&2-k), 2r+2(h-1)-k\} \\
&\geq \min\{k+2+2(h-2-k), 2(k+2)+2(h-1)-k\}\\
&=2(h-1)-k, \\
\min\{2r + 2(h-&2-k), r + 2(h-1)-k\} \\
&\geq \min\{2(k+2) + 2(h-2-k), k+2 + 2(h-1)-k\} \\
&= 2(h-1)+2.
\end{align*}
Thus Equation \eqref{e:special} implies
\begin{equation*}
\begin{cases}
s_n' = s_n & \text{if $n$ is odd and $n \leq 2(h-1)-k$}, \\
s_n' = s_n & \text{if $n$ is even and $n \leq 2(h-1)+2$}.
\end{cases}
\end{equation*}
If $k$ is even, then by assumption $r \geq k+1$ and
\begin{align*}
\min\{r + 2(h-&2-k), 2r+2(h-1)-k\} \\
&\geq \min\{k+1+2(h-2-k), 2(k+1)+2(h-1)-k\}\\
&=2(h-1)-k-1, \\
\min\{2r + 2(h-&2-k), r + 2(h-1)-k\} \\
&\geq \min\{2(k+1) + 2(h-2-k), k+1 + 2(h-1)-k\} \\
&= 2(h-1).
\end{align*}
Note that if $n$ is odd and $n \leq 2(h-1)-k$ for $k$ even, then in fact $n \leq 2(h-1)-k-1$. Thus Equation \eqref{e:special} implies
\begin{equation*}
\begin{cases}
s_n' = s_n & \text{if $n$ is odd and $n \leq 2(h-1)-k$}, \\
s_n' = s_n & \text{if $n$ is even and $n \leq 2(h-1)$},
\end{cases}
\end{equation*}

Therefore we have shown that for $s' = (1 - a \tau^r)s(1-a\tau^r)^{-1}$, we have
\begin{equation*}
\begin{cases}
s_n' = s_n = 0 & \text{if $n$ is odd and $n \leq 2(h-1)-k$,} \\
s_n' = s_n \in \FF_q & \text{if $n$ is even and $n \leq 2(h-2-k)$} \\
s_n' = s_n \notin \FF_q & \text{if $n$ is even and $n = 2(h-1-k)$} \\
s_n' = s_n & \text{if $n$ is even.}
\end{cases}
\end{equation*}
Thus $s'$ satisfies Property $\star$.
\end{proof}

\begin{lemma}\label{l:semicent}
Suppose $s = 1 + \sum s_i \tau^i \in \UnipF$ satisfies Property $\star$ for $k$. Then for any $t \in N_k$, the element $tst^{-1}$ also satisfies Property $\star$ for $k$. Furthermore, if $tst^{-1} = s' = \sum s_i' \tau^i$, then 
\begin{equation*}
s_{2i}' = s_{2i} \qquad \text{for all $i$.}
\end{equation*}
\end{lemma}

\begin{proof}
By Lemma \ref{l:form}, we can write $t \in N_k$ in the form
\begin{equation*}
(1 - a_1 \tau^1)(1 - a_3 \tau^3) \cdots (1 - a_{2(h-1)-1} \tau^{2(h-1)-1}) \cdot g \qquad \text{for some $g \in H$}.
\end{equation*}
By assumption, the coefficient of $\tau^i$ in $t$ vanishes for odd $i$ with $i \leq k$. Thus we take $a_i = 0$ for $i \leq k$. Furthermore, by Lemma \ref{l:semicentH}, proving that $tst^{-1}$ satisfies Property $\star$ for $k$ is equivalent to proving that $gsg^{-1}$ has Property $\star$ for $k$ and that for any $r > k$ odd and $a \in \F$, $(1-a\tau^r)s(1-a\tau^r)^{-1}$ has Property $\star$ for $k$. But we already know from Lemma \ref{l:semicentH} and \ref{l:semicentr} that this holds. Furthermore, Lemma \ref{l:semicentH} and \ref{l:semicentr} imply $s_{2i}' = s_{2i}$ for all $i$. This completes the proof.
\end{proof}

\begin{lemma}\label{l:centchisharp}
Let $s \in G_{h-2-k} \smallsetminus G_{h-2-k+1}$. Then
\begin{equation*}
\chi_\circ^\sharp(tst^{-1}) = \chi^\sharp(tst^{-1}) = \chi(s) \qquad \text{for any $t \in N_k$}.
\end{equation*}
\end{lemma}

\begin{proof}
Let $t \in N_k$. The assumption $s \in G_{h-2-k} \smallsetminus G_{h-1-k}$ implies that $s$ satisfies Property $\star$ for $k$. Thus by Lemma \ref{l:semicent}, the element $tst^{-1}$ also satisfies Property $\star$ for $k$ and the even-degree coefficients of $tst^{-1}$ agree with the corresponding coefficients of $s$. The desired conclusion follows.
\end{proof}

\begin{lemma}\label{l:conjUodd}
Suppose $s$ satisfies Property $\star$ for $k$ where $k$ is odd. Then
\begin{equation*}
\sum_{a \in \F^\times} \chi_\circ^\sharp((1-a\tau^k)s(1-a\tau^k)^{-1}) = -(q+1) \chi^\sharp(s).
\end{equation*}
\end{lemma}

\begin{proof}
The first half of this proof is very similar to the proof of Lemma \ref{l:semicentr}. Assume $a \in \F^\times$ and write $s' = \sum s_i' \tau^i = (1-a\tau^k)s(1-a\tau^k)^{-1}$. Combining Property $\star$ and Lemma \ref{l:conjU}, we know
\begin{align*}
s_n' &= s_n + \sum_{\substack{k(l+1)+2m=n \\ l \geq 0 \\ m \geq h-1-k}} -a^{q^l+q^{l-1} + \cdots + q + 1}(s_{2m}^q - s_{2m}) \\
&\qquad \qquad \qquad \qquad+ \sum_{\substack{k(l'+1)+2m+1=n \\ l' \geq 0 \\ 2m+1 \geq 2(h-1)-k+1}} -a^{q^{l'-1} + q^{l'-2} + \cdots + q + 1}(a s_{2m+1}^q - a^q s_{2m+1}).
\end{align*}
Using that $n \equiv l + 1 \pmod 2$, we see that the above implies, in particular,
\begin{equation*}
\begin{cases}
s_n' = s_n & \text{if $n$ is odd and $n \leq \min\{k + 2(h-2-k), 2k+2(h-1)-k\}$}, \\
s_n' = s_n & \text{if $n$ is even and $n \leq \min\{2k+2(h-2-k), k+2(h-1)-k\}$}.
\end{cases}
\end{equation*}
In the min expression, the first expression comes from the vanishing of terms in the first sum and the second expression comes from the vanishing of terms in the second sum. Simplifying these expressions gives us 
\begin{equation}\label{e:parta}
\begin{cases}
s_n' = s_n & \text{if $n$ is odd and $n \leq 2(h-1)-k-2$}, \\
s_n' = s_n & \text{if $n$ is even and $n \leq 2(h-1)-2$}.
\end{cases}
\end{equation}
and also tells us that when $n = 2(h-1)-k$ and when $n = 2(h-1)$, the contributions to $s_n' - s_n$ come from terms in the first sum only. More precisely, we have
\begin{equation}\label{e:partb}
\begin{cases}
s_n' = s_n - a(s_{n-k}^q - s_{n-k}) & \text{if $n = 2(h-1)-k$}, \\
s_n' = s_n - a^{q+1}(s_{n-2k}^q - s_{n-2k}) & \text{if $n = 2(h-1)$.}
\end{cases}
\end{equation}
Since $k \leq h-2$, then $2(h-1)-k \geq h$. Thus since $s \in K$, Equation \eqref{e:parta} implies that $s' \in K$. Recalling the definition of $\chi_\circ^\sharp$, Equations \eqref{e:parta} and \eqref{e:partb} imply
\begin{align*}
\chi_\circ^\sharp(s') 
&= \chi^\sharp(s \cdot (1 - a^{q+1}(s_{2(h-1)-k}^q - s_{2(h-1)-k})\tau^{h-1})) \\
&= \chi^\sharp(s) \cdot \psi(-a^{q+1}(s_{2(h-1)-k}^q - s_{2(h-1)-k})).
\end{align*}
Note that since $s$ satisfies Property $\star$ for $k$ and $a \neq 0$, then $-a^{q+1}(s_{2(h-1)-k}^q - s_{2(h-1)-k}) \neq 0$. Furthermore, this element is in $\ker(\Tr_{\F/\FF_q})$. Since $\psi$ has conductor $q^2$, we know that its restriction to $\ker(\Tr_{\F/\FF_q})$ is nontrivial. Notice that ranging $a \in \F^\times$ allows $-a^{q+1}(s_{2(h-1)-k}^q - s_{2(h-1)-k})$ to take each nonzero value of $\ker(\Tr_{\F/\FF_q})$ exactly $q+1$ times. Therefore, for any $s$ satisfying Property $\star$ for $k$,
\begin{align*}
\sum_{a \in \F^\times} \chi_\circ^\sharp((1 - a\tau^k)s(1-a\tau^k)^{-1}) 
&= \chi^\sharp(s) \sum_{a \in \F^\times} \psi(-a^{q+1}(s_{2(h-1-k)}^q - s_{2(h-1-k)}))  \\
&= -(q+1) \cdot \chi^\sharp(s). \qedhere
\end{align*}
\end{proof}

\begin{lemma}\label{l:sumzero}
Let $s \in G_{h-2-k} \smallsetminus G_{h-2-k+1}$. Let $n_1, \ldots, n_r$ be a decreasing sequence of consecutive odd numbers starting from $n_1 = 2(h-1)-1$ and assume that $n_r < k$. Then
\begin{equation*}
\sum_{\substack{a_1, \ldots, a_r \in \F \\ a_r \neq 0}} \chi_\circ^\sharp((1-a_r\tau^{n_r}) \cdots (1-a_1 \tau^{n_1})s(1-a_1 \tau^{n_1})^{-1} \cdots (1-a_r \tau^{n_r})^{-1}) = 0.
\end{equation*}
\end{lemma}

\begin{proof}
Let $g = (1 - a_r \tau^{n_r}) \cdots (1 - a_1 \tau^{n_1})$. If $n_r + 2(h-1-k) \leq h-1$, then we see from Lemma \ref{l:conjU} that the coefficient of $\tau^{n_r + 2(h-1-k)}$ in $gsg^{-1}$ is
\begin{equation*}
-a_r(s_{2(h-1-k)}^q - s_{2(h-1-k)}) \neq 0.
\end{equation*}
Thus $gsg^{-1} \notin K$ and $\chi_\circ^\sharp(gsg^{-1}) = 0$. We may therefore assume that $n_r > 2k - (h-1)$.

Let $f$ be such that $n_r + n_f = 2k$ (so that $n_r + n_f + 2(h-1-k) = 2(h-1)$).Note that such an $f$ exists by the assumption $n_r < k$. To prove the lemma, we will prove the following sum identity. Fix $a_i \in \F$ for $1 \leq i \leq r$, $i \neq f$, and assume that $a_r \neq 0$. Then
\begin{equation}\label{e:simplesum}
\sum_{a_f \in \F} \chi_\circ^\sharp((1-a_r\tau^{n_r}) \cdots (1-a_1 \tau^{n_1})s(1-a_1 \tau^{n_1})^{-1} \cdots (1-a_r \tau^{n_r})^{-1}) = 0.
\end{equation}
It is clear that once this is established, the lemma follows immediately. We thus focus the rest of the proof on proving Equation \eqref{e:simplesum}.

Write $g = (1 - a_r \tau^{n_r}) \cdots (1 - a_1 \tau^{n_1})$. We must study the contribution of $a_f$ in $gsg^{-1}$. By construction $n_f$ is odd and $s \in H$. Thus $a_f$ can only contribute to the coefficient of $\tau^{2l}$ in conjunction with (at least) one of the other $a_i$'s for $1 \leq i \leq r$. 

For convenience, let $gsg^{-1} = 1 + \sum c_i \tau^i$. First observe that since $s \in G_{h-2-k}$, we have $s_{2i}^q - s_{2i} = 0$ for $1 \leq i \leq h-2-k$, and thus the smallest odd $i$ such that $a_f$ has a nonzero contribution to $c_i$ is when $i = 2(h-1-k) + n_f > 2(h-1-k) + 2k - (h-1) = h-1$. Thus we see that $gsg^{-1} \in K$ for $a_f = 0$ if and only if $gsg^{-1} \in K$ for any $a_f \in \F$ (remember that $g$ depends on $a_f$).

If $gsg^{-1} \notin K$, then we are done. Now assume $gsg^{-1} \in K$. By construction, the value of $\chi^\sharp(gsg^{-1})$ depends only on the coefficients $c_{2i}$ for $1 \leq i \leq h-1$. If $a_f$ contributes to some $c_{2i}$, then we must have 
\begin{equation*}
2i \geq n_f + n_r + 2(h-1-k) = 2(h-1).
\end{equation*}
Thus $a_f$ only contributes to $c_{2(h-1)}$. Furthermore, its contribution to $c_{2(h-1)}$ is
\begin{align*}
a_r a_f^q s_{2(h-1-k)} - a_r s_{2(h-1-k)}^q a_f^q - a_f &s_{2(h-1-k)}^q a_r^q + s_{2(h-1-k)} a_f a_r^q \\
&= -(a_r a_f^q + a_r^q a_f)(s_{2(h-1-k)}^q - s_{2(h-1-k)}).
\end{align*}
(One can see this by computing $(1-a_r \tau^{n_r})(1-a_f \tau^{n_f})s(1-a_f \tau^{n_f})^{-1}(1-a_r \tau^{n_r})^{-1}$.) Thus
\begin{equation*}
\chi_\circ^\sharp(gsg^{-1}) = \chi^\sharp(gsg^{-1}) = \chi^\sharp(\gamma) \cdot \psi(-(a_r a_f^q + a_r^q a_f)(s_{2(h-1-k)}^q - s_{2(h-1-k)})),
\end{equation*}
where $\gamma$ does not depend on the choice of $a_f \in \F$. Thus
\begin{align*}
\sum_{a_f \in \F} \chi_\circ^\sharp(gsg^{-1}) 
= \chi^\sharp(\gamma) \cdot \sum_{a_f \in \F} \psi(-(a_r a_f^q + a_r^q a_f)(s_{2(h-1-k)}^q - s_{2(h-1-k)})).
\end{align*}
Note that for any $c \in \FF_q$, any solution $x$ to $a_r x^q + a_r^q x = c$ must satisfy $x^{q^2} = x$. Thus varying $a_f \in \F$, the quantity $a_r a_f^q + a_r^q a_f$ takes the value of each element $\FF_q$ exactly $q$ times. Since $s_{2(h-1-k)} \notin \FF_q$, then $-(a_r a_f^q + a_r^q a_f)(s_{2(h-1-k)}^q - s_{2(h-1-k)})$ attains each value of $\ker \Tr_{\F/\FF_q}$ exactly $q$ times. By the assumption that $\psi$ has conductor $q^2$, the restriction of $\psi$ to the subgroup $\ker \Tr_{\F/\FF_q}$ is nontrivial, and thus
\begin{equation*}
\sum_{a_f \in \F} \psi(-(a_r a_f^q + a_r^q a_f)(s_{2(h-1-k)}^q - s_{2(h-1-k)})) = 0.
\end{equation*}
Equation \eqref{e:simplesum} follows.
\end{proof}

We are now ready to prove Proposition \ref{p:trace}, Theorem \ref{t:decomp}, and Corollary \ref{c:mult}.

\begin{proof}[Proof of Proposition \ref{p:trace}]
It is easy to see that
\begin{equation}\label{e:size}
|N_k| = \begin{cases}
q^{4(h-1)-k} & \text{if $k$ is even,} \\
q^{4(h-1)-(k+1)} & \text{if $k$ is odd.}
\end{cases}
\end{equation}
We will use this at various points in this proof.

If $s \in G_{h-2}$, then $s$ is central in $\UnipF$. Thus for any $t \in \UnipF$, $tst^{-1} = s$ and 
\begin{equation*}
\frac{1}{|K|} \sum_{t \in \UnipF} \chi_\circ^\sharp(tst^{-1}) = \frac{|\UnipF|}{|K|} \chi(s) = \begin{cases} q^{h-1} \cdot \chi(s) & \text{if $h$ is odd} \\ q^h \cdot \chi(s) & \text{if $h$ is even.} \end{cases}
\end{equation*}
Thus by Remark \ref{r:isotypic}, we have
\begin{equation*}
\Tr \rho_\chi(s) = q^{h-1} \cdot \chi(s).
\end{equation*}
This proves (a).

Let $s \in G_{h-2-k} \smallsetminus G_{h-2-k+1}$. We first handle the case when $k$ is even. We have
\begin{align*}
\sum_{t \in \UnipF} \chi_\circ^\sharp(tst^{-1}) 
&= \underbrace{\sum_{t \in N_k} \chi_\circ^\sharp(tst^{-1})}_{(1)} + \underbrace{\sum_{t \notin N_k} \chi_\circ^\sharp(tst^{-1})}_{(2)}.
\end{align*}
By Lemma \ref{l:centchisharp}, we know
\begin{equation}\label{e:1odd}
(1) = |N_k| \cdot \chi(s).
\end{equation}
By Lemma \ref{l:form}, we know that every element $t \in U_h^{q,2}(\F)$ can be written in the form
\begin{equation*}
(1 - a_1 \tau^1)(1 - a_3 \tau^3) \cdots (1 - a_{2(h-1)-1} \tau^{2(h-1)-1}) \cdot g
\end{equation*}
for some $g \in H.$ Since $H$ is abelian, this implies that $gsg^{-1} = s$, and the assumption $t \notin N_k$ implies that there exists $i$ odd with $i < k$ such that $a_i \neq 0$. Thus
\begin{equation*}
(2) = |H| \cdot \sum_{\substack{a_i \in \F \\ \text{$\exists \, i < k$, with $a_i \neq 0$}}} \chi_\circ^\sharp(asa^{-1}) = 0,
\end{equation*}
where $a = (1 - a_1 \tau)(1 - a_3 \tau^3) \cdots (1 - a_{2(h-1)-1} \tau^{2(h-1)-1})$, and the last equality holds by Lemma \ref{l:sumzero}.

Therefore,
\begin{equation*}
\frac{1}{|K|} \sum_{t \in \UnipF} \chi_\circ^\sharp(tst^{-1}) 
= \frac{|N_k|}{|K|} \cdot \chi(s)
= \begin{cases}
q^{h-1-k} \cdot \chi(s) & \text{$h$ odd}, \\
q^{h-k} \cdot \chi(s) & \text{$h$ even}.
\end{cases}
\end{equation*}
Recalling Remark \ref{r:isotypic}, this finishes the proof of the proposition in the case $k$ is even.

Now let $k$ be odd. By Lemma \ref{l:form}, we have
\begin{align*}
\sum_{t \in \UnipF} \chi_\circ^\sharp(tst^{-1}) 
&= \underbrace{\sum_{t \in N_k} \chi_\circ^\sharp(tst^{-1})}_{(1)} + \underbrace{\sum_{\substack{t \in N_k \\ a \in \F^\times}} \chi_\circ^\sharp((1-a\tau^k)tst^{-1}(1-a\tau^k)^{-1})}_{(2)} \\
&\qquad\qquad+ \underbrace{\sum_{\substack{t \notin N_k \\ \text{$\exists \, i < k$ odd s.t.\ $t_i \neq 0$}}} \chi_\circ^\sharp(tst^{-1})}_{(3)}.
\end{align*}
By Lemma \ref{l:centchisharp}, we know
\begin{equation}\label{e:1odd}
(1) = |N_k| \cdot \chi(s).
\end{equation}
By Lemma \ref{l:semicent}, we know that given $s \in G_{h-2-k} \smallsetminus G_{h-2-k+1}$ and $t \in N_k$, we have that $tst^{-1}$ satisfies Property $\star$ for $k$ and that $\chi^\sharp(tst^{-1}) = \chi(s)$. Thus by Lemma \ref{l:conjUodd}, we have
\begin{equation*}
\sum_{a \in \F^\times} \chi_\circ^\sharp((1-a\tau^k)tst^{-1}(1-a\tau^k)^{-1}) = -(q+1)\chi^\sharp(tst^{-1}) = -(q+1) \chi(s).
\end{equation*}
Therefore
\begin{equation}\label{e:2odd}
(2) = - |N_k| (q+1) \cdot \chi(s).
\end{equation}
By the same argument as the case when $k$ is even, it follows from Lemma \ref{l:form} and Lemma \ref{l:sumzero} that
\begin{equation*}
(3) = 0.
\end{equation*}

Therefore,
\begin{align*}
\frac{1}{|K|} \sum_{t \in \UnipF} \chi_\circ^\sharp(tst^{-1}) 
&= \frac{1}{|K|} \cdot (|N_k| \cdot \chi(s) - |N_k| (q + 1) \cdot \chi(s)) \\
&= \begin{cases}
-q^{h-1-k} \cdot \chi(s) & \text{$h$ odd}, \\
-q^{h-k} \cdot \chi(s) & \text{$h$ even}.
\end{cases}
\end{align*}
By Remark \ref{r:isotypic}, this finishes the proof of the proposition when $k$ is odd.
\end{proof}

\begin{proof}[Proof of Theorem \ref{t:decomp}]
Consider the (virtual) $H$-representation 
\begin{equation*}
\rho = (-1)^h\left(q \cdot \chi + \sum_{i=1}^{h-2} (-1)^i (q+1) \Ind_{G_i}^H(\chi)\right).
\end{equation*}
Since $H$ is abelian, its trace is very easy to calculate: using $|H|/|G_i| = q^i$, for any $s \in H$, 
\begin{align*}
\Tr \rho(s) 
&= (-1)^h\left(q \cdot \chi(s) + \sum_{i=1}^{h-2} (-1)^{i}(q+1) \Tr \Ind_{G_i}^H(\chi)(s)\right) \\
&= (-1)^h\left(q \cdot \chi(s) + \sum_{i=1}^{h-2} (-1)^{i}(q+1) q^i \cdot \ONE_{G_i}(s) \cdot \chi(s)\right).
\end{align*}
Therefore:
\begin{enumerate}[label=(\alph*)]
\item
If $s \in G_{h-2}$, then
\begin{equation*}
\Tr \rho(s) = (-1)^h \cdot \chi(s) \cdot \left(q + \sum_{i=1}^{h-2} (-1)^i(q+1) q^i\right) = q^{h-1} \cdot \chi(s).
\end{equation*}

\item
If $s \in G_{h-2-k} \smallsetminus G_{h-2-k+1}$, then
\begin{equation*}
\Tr \rho(s)  = (-1)^h \cdot \chi(s) \cdot \left(q + \sum_{i=1}^{h-2-k} (-1)^i(q+1) q^i \right) = (-1)^k q^{h-1-k} \cdot \chi(s).
\end{equation*}
\end{enumerate}
Comparing this with Proposition \ref{p:trace}, we see that
\begin{equation*}
\rho_\chi(s) = \rho(s) \qquad \qquad \text{for all $s \in H(\F)$}.
\end{equation*}
Therefore $\rho_\chi = \rho$ as elements of the Grothendieck group of $H$.
\end{proof}

\begin{proof}[Proof of Corollary \ref{c:mult}]
Given a character $\theta \from H(\F) \to \overline \QQ_\ell^\times$, we can read off its multiplicity from the result of Theorem \ref{t:decomp}. Indeed, since $H$ is abelian, then if $\theta$ is an $H$-character that agrees with $\chi$ on some subgroup $G_m$ but not on $G_{m-1}$, then it occurs exactly once in $\Ind_{G_i}^H(\chi)$ for every $i \geq m$ and does not occur in $\Ind_{G_i}^H(\chi)$ for $i \leq m - 1$. Therefore:
\begin{enumerate}[label=(\alph*)]
\item
The character $\chi$ occurs in $\rho_\chi$ with multiplicity equal to
\begin{align*}
(-1)^h&\left(q - (q+1) + (q+1) - \cdots + (-1)^{h-2}(q+1)\right) \\
&\qquad\qquad\qquad= \begin{cases}
(-1)(q - (q+1)) = 1 & \text{if $h$ is odd,} \\
q & \text{if $h$ is even.}
\end{cases}
\end{align*}

\item
Let $\theta$ be a character of $H$ such that, for some odd $k$, $\theta$ agrees with $\chi$ on $G_{h-2-k}$ but not on $G_{h-2-k-1}$. Then $\theta$ occurs in $\rho_\chi$ with multiplicity equal to
\begin{align*}
(-1)^h\left((-1)^{h-2-k}(q+1) + \cdots + (-1)^{h-2}(q+1)\right) = 0,
\end{align*}
since this is an alternating sum of $k+1$ terms and $k$ is odd.

\item
Let $\theta$ be a character of $H$ such that, for some even $k$, $\theta$ agrees with $\chi$ on $G_{h-2-k}$ but not on $G_{h-2-k-1}$. Then $\theta$ occurs in $\rho_\chi$ with multiplicity equal to
\begin{align*}
(-1)^h&\left((-1)^{h-2-k}(q+1) + \cdots + (-1)^{h-2}(q+1)\right) \\
&\qquad\qquad\qquad\qquad = (-1)^h(-1)^{h-2}(q+1) = q+1.
\end{align*}

\item
Let $\theta$ be a character of $H$ that does not agree with $\chi$ on $G_{h-2}$. Since $G_{h-2}$ is in the center of $\UnipF$, then the restriction of $\rho_\chi$ to $G_{h-2}$ must be a sum of $\chi|_{G_{h-2}}$. Therefore the multiplicity of $\theta$ in $\rho_\chi$ must be 0.
\end{enumerate}
This completes the proof.
\end{proof}

\section{Morphisms Between $H_c^i(X_h)$ and Representations of $\UnipF$}\label{s:homspace}

Let $H_c^\bullet(X_h) = \bigoplus_{i \in \ZZ} H_c^i(X_h, \overline \QQ_\ell)$. The aim of this section is to compute the space $\Hom_{\UnipF}(V_\psi, H_c^\bullet(X_h))$. Recall that 
\begin{equation*}
V_\psi = \Ind_{H_0'(\F)}^{\UnipF}(\widetilde \psi),
\end{equation*}
where $\widetilde \psi$ is the extension of $\psi$ to $H_0'(\F)$ defined in Section \ref{s:repthy}.

In the following theorem, we prove a clean way to express the equations cutting out the scheme $X_h \subseteq \Unip$. This will be heavily used in this section, as well as in the next section.

\begin{theorem}\label{t:Xhpolys}
The scheme $X_h \subset \Unip$ is defined by the vanishing of the polynomials
\begin{equation*}
f_{2k} \colonequals (a_{2k}^{q^2} - a_{2k}) + \sum_{i = 1}^{2k-1} (-1)^i a_i^q(a_{2k-i}^{q^2} - a_{2k-i})
\end{equation*}
for $1 \leq k \leq h-1$.
\end{theorem}

\begin{proof}
It suffices to verify this claim at the level of $\overline \FF_q$-points. Recall that we have an embedding of $\Unip(\overline \FF_q)$ into the set of matrices over $\overline \FF_q[\pi]/(\pi^h)$ given by
\begin{equation*}
\iota_h \from 1 + \sum_{i=1}^{2(h-1)} a_i \tau^i \mapsto 
\left(\begin{matrix}
1 + a_2 \pi + a_4 \pi^2 + \cdots & a_1 + a_3 \pi + a_5 \pi^2 + \cdots \\
a_1^q \pi + a_3^q \pi^2 + \cdots & 1 + a_2^q \pi + a_4^q \pi^2 + \cdots
\end{matrix}\right).
\end{equation*}
The determinant of $\iota_h(1 + \sum a_i \tau^i)$ is a polynomial in $\pi$ with coefficients in $\overline \FF_q$. Let $c_k$ be the coefficient of $\pi^k$ in this determinant. By definition, $1 + \sum a_i \tau^i$ is in $X_h(\overline \FF_q)$ if and only if $c_k^q = c_k$ for $k = 1, \ldots, h-1$. Note that $c_k$ is a polynomial in $a_1, \ldots, a_{2(h-1)}$. We wish to find a clean way to write down $c_k^q - c_k$ as a polynomial in the $a_i$'s.

The coefficient $c_k$ of $\pi^k$ is equal to the coefficient of $\pi^k$ in $(1 + a_2 \pi + \cdots)(1 + a_2^q \pi + \cdots)$ minus the coefficient of $\pi^k$ in $(a_1 + a_3 \pi + \cdots)(a_1^q \pi + a_3^q \pi^2 + \cdots)$. Thus the coefficient of $\pi^k$ in the determinant is
\begin{equation*}
c_k \colonequals \sum_{j=0}^k a_{2j} a_{2k-2j}^q - a_{2j+1} a_{2k-(2j+1)}^q,
\end{equation*}
where it is understood that $a_0 = 1$ and $a_{-n} = 0$ for $n \in \NN$. We can now focus our attention on the terms in $c_k$ involving $a_i$. We see that if $i = 2j$, then the terms in $c_k$ involving $a_i$ are
\begin{equation*}
a_{2j} a_{2k-2j}^q + a_{2k-2j} a_{2j}^q,
\end{equation*}
and that if $i = 2j + 1$, then the terms in $c_k$ involving $a_i$ are
\begin{equation*}
-a_{2j+1} a_{2k-(2j+1)}^q - a_{2k-(2j+1)}a_{2j+1}^q.
\end{equation*}
Therefore the terms in $c_k^q - c_k$ involving $a_i$ are
\begin{equation*}
(-1)^i[(a_i a_{2k-i}^q + a_{2k-i} a_i^q)^q - (a_i a_{2k-i}^q + a_{2k-i} a_i^q)],
\end{equation*}
which simplifies to
\begin{equation*}
(-1)^i [a_i^q(a_{2k-i}^{q^2} - a_{2k-i}) + a_{2k-i}^q(a_i^{q^2} - a_i)].
\end{equation*}
Setting $f_{2k} \colonequals c_k^q - c_k$ completes the proof.
\end{proof}

The following theorem is the main result of this section.

\begin{theorem}\label{t:homspace}
Let $\psi$ be an additive character of $\F$ with conductor $q^2$. If $h$ is odd, then
\begin{equation*}
\dim\Hom_{\UnipF}(V_\psi, H_c^i(X_h, \overline \QQ_\ell)) = 
\begin{cases}
q^{2(h-2)} & \text{if $i = h-1$}, \\
0 & \text{otherwise.}
\end{cases}
\end{equation*}
If $h$ is even, then
\begin{equation*}
\dim\Hom_{\UnipF}(V_\psi, H_c^i(X_h, \overline \QQ_\ell)) = 
\begin{cases}
q^{2(h-2)+1} & \text{if $i = h-1$,} \\
0 & \text{otherwise.}
\end{cases}
\end{equation*}
Moreover, the Frobenius $\Fr_{q^2}$ acts on $\Hom_{\UnipF}(V_\psi, H_c^{h-1}(X_h, \overline \QQ_\ell))$ via multiplication by the scalar $(-1)^{h-1} q^{h-1}$.
\end{theorem}

This is proven in Section \ref{s:pfhomspace}. As a corollary to Theorem \ref{t:homspace}, we have the following.

\begin{corollary}\label{c:chispacevanish}
Let $\psi$ be an additive character of $\F$ with conductor $q^2$. If $\chi \from U_L^1/U_L^h \to \overline \QQ_\ell^\times$ is a character that restricts to $\psi$ on $U_L^{h-1}/U_L^h$, then $H_c^i(X_h, \overline \QQ_\ell)[\chi] = 0$ for all $i \neq h-1$.
\end{corollary}

\begin{proof}
The left action of $U_L^1/U_L^h$ and the right action of $\UnipF$ on $X_h$ agree on $U_L^{h-1}/U_L^h \cong H_{2(h-1)}(\F)$. Therefore, since $U_L^{h-1}/U_L^h$ acts by $\psi$ on $H_c^i(X_h, \overline \QQ_\ell)[\chi]$, then $H_{2(h-1)}(\F)$ also acts by $\psi$. We know from our analysis of the representations of $\UnipF$ that every irreducible component of $H_c^i(X_h, \overline \QQ_\ell)[\chi]$ appears in $V_\psi$, so this forces $H_c^i(X_h, \overline \QQ_\ell)[\chi] = 0$ if $i \neq h - 1$.
\end{proof}

This will allow us to compute intertwining spaces using Lemma 2.13 of \cite{B12}. This will be exploited in Section \ref{s:cohomreps}.

\subsection{Proof of Theorem \ref{t:homspace}} \label{s:pfhomspace}

The structure of the proof is as follows. We first use Proposition 2.3 of \cite{B12} to reduce the computation of $\Hom_{\UnipF}(V_\psi, H_c^i(X_h, \overline \QQ_\ell))$ to the computation of the cohomology of a certain scheme $S$ with coefficients in a certain constructible $\QQ_\ell$ sheaf $\EFF$. Then, to compute $H_c^i(S, \EFF)$, we apply (a slightly more general version of) Proposition 2.10 of \cite{B12} inductively. This will allow us to reduce the computation of $H_c^i(X, \sF)$ to a computation involving a 0-dimensional scheme in the case that $h$ is odd, and a computation involving a 1-dimensional scheme in the case that $h$ is even. Because the computation is identical until this final step, we treat these to cases simultaneously until the very last step.

We start with a slight generalization of Proposition 2.10 of \cite{B12} that has been tailored for our purposes.

\begin{proposition}\label{p:2.10rev}
Let $q$ be a power of $p$, let $n \in \NN$, and let $\psi \from \FF_{q^n} \to \overline \QQ_\ell^\times$ be a character that has conductor $q	^m$. Let $S_2$ be a scheme of finite type over $\FF_{q^n}$, put $S = S_2 \times \Affine^1$ and suppose that a morphism $P \from S \to \GG_a$ has the form
\begin{equation*}
P(x,y) = f(x)^{q^j} y - f(x)^{q^n} y^{q^{n-j}} + \alpha(x,y)^{q^m} - \alpha(x,y) + P_2(x).
\end{equation*}
Here, $j$ is some integer $j$ not divisible by $m$; $f, P_2 \from S_2 \to \GG_a$ are two morphisms; and $\alpha \from S_2 \times \Affine^1 \to \GG_a$ is a morphism.
Let $S_3 \subset S_2$ be the subscheme defined by $f = 0$ and let $P_3 = P_2|_{S_3} \from S_3 \to \GG_a$. Then for all $i \in \ZZ$, we have
\begin{equation*}
H_c^i(S, P^* \Loc_\psi) \cong H_c^{i-2}(S_3, P_3^* \Loc_\psi)(-1)
\end{equation*}
as vector spaces equipped with an action of $\Fr_{q^n}$, where the Tate twist $(-1)$ means that the action of $\Fr_{q^n}$ on $H_c^{i-2}(S_3, P_3^* \Loc_\psi)$ is multiplied by $q^n$.
\end{proposition}

\begin{proof}[Proof of Proposition \ref{p:2.10rev}] 
Let $P'(x,y) = f(x)^{q^j} y - f(x)^{q^n} y^{q^{n-j}} + P_2(x)$ and $P''(x,y) = \alpha(x,y)^{q^m} - \alpha(x,y).$ We show that the pullbacks $P^*\Loc_\psi$ and $(P')^*\Loc_\psi$ are isomorphic. Since $\psi$ has conductor $q^m$, the pullback of $\Loc_\psi$ by the map $z \mapsto z^{q^m}$ is trivial, and so thus $(P'')^* \Loc_\psi$ is trivial. Since $\Loc_\psi$ is additive, then we have shown that $P^* \Loc_\psi$ and $(P')^* \Loc_\psi$ are isomorphic and thus by Proposition 2.10 of \cite{B12},
\begin{equation*}
H_c^i(S, P^* \Loc_\psi) \cong H_c^i(S, (P')^* \Loc_\psi) \cong H_c^{i-2}(S_3, P_3^* \Loc_\psi)(-1)
\end{equation*}
as vector spaces equipped with an action of $\Fr_{q^2}$.
%
\end{proof}

We now return to the proof of Theorem \ref{t:homspace}.

\subsubsection*{Step 0.}
We first need to establish some notation. I have tried to make this notation reminiscent of that used in the proof of Proposition 6.5 in \cite{BW11}.
\begin{enumerate}[label=\textbullet]
\item
Let $I'$ denote the set of integers $j$ such that $h-1 < j < 2(h-1)$ and $2 \nmid j$. Put $I = I' \cup \{2(h-1)\}$.

\item
Put $J = [2(h-1)] \setminus I$, where $[n] = \{1, \ldots, n\}$.

\item
Put $I_0 \colonequals I'$ and $J_0 \colonequals J$. Then define $I_1 \colonequals I_0 \setminus \{2(h-1) - 1\}$ and $J_1 \colonequals J_0 \setminus \{1\}$. This describes a recursive construction of $I_k$ and $J_k$; namely, one obtains $I_k$ from $I_{k-1}$ by removing the largest odd number and one obtains $J_k$ from $J_{k-1}$ by removing the smallest odd number. This defines indexing sets $I_k$ and $J_k$ for $1 \leq k \leq \lfloor (h-1)/2 \rfloor$.

\item
Note that if $h$ is odd, then
\begin{align*}
I_{\lfloor (h-1)/2 \rfloor} &= I_{(h-1)/2} = \varnothing, \\
J_{\lfloor (h-1)/2 \rfloor} &= J_{(h-1)/2} = \{2, 4, \ldots, 2(h-2)\}.
\end{align*}
If $h$ is even, then
\begin{align*}
I_{\lfloor (h-1)/2 \rfloor} &= I_{(h-2)/2} = \varnothing, \\
J_{\lfloor (h-1)/2 \rfloor} &= J_{(h-2)/2} = \{2, 4, \ldots, 2(h-2)\} \cup \{h-1\}.
\end{align*}
This distinction is exactly why our inductive argument reduces to a $0$-dimensional scheme in the case that $h$ is odd and a $1$-dimensional scheme in the case that $h$ is even.

\item
Note that $H_0' = \{1 + \sum a_i \tau^i : i \in I\}$.

\item
For a finite set $T \subset \NN$, we will write $\Affine[T]$ to denote affine space with coordinates $x_i$ for $i \in T$.
\end{enumerate}

\subsubsection*{Step 1}
We apply Proposition 2.3 of \cite{B12} to the following set-up:
\begin{enumerate}[label=\textbullet]
\item
$G = \Unip$ and $H = H_0'$, both defined over $\F$

\item
the morphism $s \from \Unip/H_0' \to \Unip$ defined by sending the $i$th coordinate to the coefficient of $\tau^i$; that is, identify $\Unip/H_0'$ with affine space with coordinates indexed by $J$, and set $s \from (x_i)_{i \in J} \mapsto 1 + \sum_{i \in J} x_i \tau^i$.

\item
the algebraic group homomorphism $f \from H_0' \to \GG_a$ given by projection to the last coordinate. That is, $f \from 1 + \sum_{i \in I} a_i \tau^i \mapsto a_{2(h-1)}$. (From the definition of $H_0'$, it is easy to see that this map is a homomorphism.)

\item
an additive character $\psi \from \F \to \overline \QQ_\ell^\times$

\item
a locally closed subvariety $Y_h \subset \Unip$ which is chosen so that $X_h = L_{q^2}^{-1}(Y_h)$
\end{enumerate}
Since $X_h$ has a right-multiplication action by $\UnipF$, the cohomology groups $H_c^i(X_h, \overline \QQ_\ell)$ inherit a $\UnipF$-action. For each $i \geq 0$, Proposition 2.3 of \cite{B12} implies that we have a vector space isomorphism
\begin{equation*}
\boxed{\Hom_{\UnipF}(V_\psi, H_c^i(X_h, \overline \QQ_\ell)) \cong H_c^i(\beta^{-1}(Y_h), P^* \Loc_\psi)}
\end{equation*}
compatible with the action of $\Fr_{q^2}$. Here, $\Loc_\psi$ is the Artin-Schreier local system on $\GG_a$ corresponding to $\psi$, the morphism $\beta \from (\Unip/H_0') \times H_0' \to \Unip$ is given by $\beta(x,g) = s(\Fr_{q^2}(x)) \cdot g \cdot s(x)^{-1}$, and the morphism $P \from \beta^{-1}(Y_h) \to \GG_a$ is the composition $\beta^{-1}(Y_h) \hookrightarrow (\Unip/H_0') \times H_0' \stackrel{\pr}{\longrightarrow} H_0' \stackrel{f}{\longrightarrow} \GG_a$.

We now work out an explicit description of $\beta^{-1}(Y_h) \subset \Affine[J] \times H_0'$ (keep in mind that we identified $\Unip/H_0'$ with $\Affine[J]$). For $1 \leq l \leq (h-1)$, recall the polynomial described in Theorem \ref{t:Xhpolys}
\begin{equation*}
f_{2l} \colonequals (a_{2l}^{q^2} - a_{2l}) + \sum_{i = 1}^{2l-1} (-1)^i a_i^q(a_{2l-i}^{q^2} - a_{2l-i}).
\end{equation*}
Write $x = (x_i)_{i \in J} \in \Affine[J]$ and $g = 1 + \sum_{i \in I} x_i \tau^i \in H_0'(\overline \FF_q)$. (Note that $I \cap J = \varnothing$; the $x_i$ in $x$ and $x_i$ in $g$ are independent of each other.) For each $i \in I$, we can write $x_i = y_i^{q^2} - y_i$ for $y_i \in \overline \FF_q$, so that $g = L_{q^2}(y)$, where $y \colonequals 1 + \sum_{i \in I} y_i \tau^i$. Therefore
\begin{equation*}
\beta(x,g) = \Fr_{q^2}(s(x)) \cdot L_{q^2}(y) \cdot s(x)^{-1} = L_{q^2}(s(x) \cdot y).
\end{equation*}
We see that $\beta(x,g) \in Y_h$ if and only if $s(x) \cdot y \in X_h$. Let $s(x) \cdot y = 1 + \sum a_i \tau^i = a$. By Theorem \ref{t:Xhpolys}, we know that $s(x) \cdot y \in X_h$ if and only if $f_{2l}(a) = 0$ for all $l$ with $1 \leq l \leq h-1$.

\subsubsection*{Step 2}

This is a necessary preparation step before we apply Proposition \ref{p:2.10rev}. As in Step 1, let $x = (x_i)_{i \in J} \in \Affine[J]$ and $s(x) = 1 + \sum_{i \in J} x_i \tau^i \in \Unip(\overline \FF_q)$. Let $g = 1 + \sum_{i \in I} x_i \tau^i \in H_0'(\overline \FF_q)$ and let $y_i$ be such that $x_i = y_i^{q^2} - y_i$ for $i \in I$ so that $g = L_{q^2}(y)$, where $y = 1 + \sum_{i \in I} y_i \tau^i$. Recall that we wrote $s(x) \cdot y = 1 + \sum a_i \tau^i = a$.

From direct computation, we can write down an explicit description of each coefficient $a_i$ in terms of $x$ and $y$. For convenience, let $r = 2 \lfloor h/2 \rfloor$. Then
\begin{equation}\label{e:aiform}
a_i = 
\begin{cases}
x_i & \text{if $i \leq r$,} \\
y_i + x_2 y_{i-2}^{q^2} + x_4 y_{i-4}^{q^4} + \cdots + x_{i-(r+1)} y_{r+1}^{q^{i-(r+1)}} & \text{if $r < i$ and $i$ is odd,} \\
x_1 y_{i-1}^q + x_3 y_{i-3}^{q^3} + \cdots + x_{i-(r+1)} y_{r+1}^{q^{i-(r+1)}} + x_i & \text{if $r < i < 2(h-1)$ and $i$ is even,} \\
y_{2(h-1)} + x_1 y_{2(h-1) - 1}^q + x_3 y_{2(h-1) - 3}^{q^3} + \cdots & \\
\qquad\qquad\qquad\qquad+ x_{2(h-1)-(r+1)} y_{r+1}^{q^{2(h-1)-(r+1)}} & \text{if $i = 2(h-1)$.}
\end{cases}
\end{equation}

Fix $1 \leq l \leq h-1$. The polynomial $f_{2l}(a)$ is \textit{a priori} a polynomial in $x_i$ for $i \in J$ and $y_i$ for $i \in I$. In this step, we show that, after setting $x_i = y_i^{q^2} - y_i$ for $i \in I$, the expression $f_{2l}(a)$ is actually a polynomial in $x_i$ for $i \in I \cup J$.

First observe that the monomials occurring in $f_{2l}(s(x) \cdot y)$ can involve $y_i$ for at most one $i$. More precisely, a monomial occurring $f_{2l}(a)$ takes one of the following forms: 
\begin{enumerate}[label=(\roman*)]
\item
It is a product of powers of $x_i$'s.

\item
It involves $y_{2(h-1)}$.

\item
It is of the form $x_i^\alpha x_j^\beta y_k^\gamma$, where $i \geq 0$ is even, $j \leq 2\lfloor h/2 \rfloor$ is odd, and $k \geq 2 \lfloor h/2 \rfloor + 1$ is odd. (As usual, we set $x_0 = 1$.)
\end{enumerate}
We need to show that setting $x_i = y_i^{q^2} - y_i$ for $i \in I$ allows us to write the monomials in (ii) and (iii) as expressions involving only $x_i$'s for $i \in I \cup J$.

The term $y_{2(h-1)}$ only occurs in the polynomial $f_{2l}(a)$ for $l = h-1$. Its contribution to $f_{2(h-1)}(a)$ is
\begin{equation*}
y_{2(h-1)}^{q^2} - y_{2(h-1)} = x_{2(h-1)},
\end{equation*}
so this takes care of (ii).

Now pick $i$, $j$, $k$ with $i + j + k = 2l$ so that $i \geq 0$ is even, $j \leq 2\lfloor h/2 \rfloor$ is odd, and $k > 2\lfloor h/2 \rfloor$ is odd. Then $y_k$, $x_i$, and $x_j$ occur in $f_{2l}(a)$ in the terms
\begin{equation*}
a_i^q(a_{j+k}^{q^2} - a_{j+k}) + a_{j+k}^q(a_i^{q^2} - a_i) - a_j^q(a_{i+k}^{q^2} - a_{i+k}) - a_{i+k}^q(a_j^{q^2} - a_j),
\end{equation*}
and are exactly
\begin{equation*}
x_i^q((x_j y_k^{q^j})^{q^2} - x_j y_k^{q^j}) + (x_j y_k^{q^j})^q(x_i^{q^2} - x_i) - x_j^q((x_i y_k^{q^i})^{q^2} - x_i y_k^{q^i}) - (x_i y_k^{q^i})^q(x_j^{q^2} - x_j).
\end{equation*}
Note that monomials of the form $x_i^\alpha x_j^\beta y_k^\gamma$ do not occur in $a_k^q(a_{i+j}^{q^2} - a_{i+j})$ or $a_{i+j}^q(a_k^q - a_k)$ (see Equation \eqref{e:aiform}). The above simplifies to
\begin{equation*}
x_i^q x_j^{q^2}(y_k^{q^{j+2}} - y_k^{q^{i+1}}) - x_i x_j^q(y_k^{q^{j+1}} - y_k^{q^i}) - x_i^q x_j(y_k^{q^j} - y_k^{q^{i+1}}) + x_i^{q^2} x_j^q(y_k^{q^{j+1}} - y_k^{q^{i+2}}).
\end{equation*}
By assumption, $i$ is even and $j$ is odd, which means that each expression involving $y_k$'s is of the form $y_k^{q^{m+2n}} - y_k^{q^m}$ for some $n$. Since $y_k^{q^2} - y_k = x_k$, we then have
\begin{equation*}
y_k^{q^{m+2n}} - y_k^{q^m} = (y_k^{2n} - y_k)^{q^m} = (x_k^{q^{2n-2}} + x_k^{q^{2n-4}} + \cdots + x_k^{q^2} + x_k)^{q^m}.
\end{equation*}
This takes care of (iii) and thus we have shown that for any $1 \leq l \leq h-1$, $f_{2l}(a)$ is a polynomial in terms of $x_i$ for $i \in I \cup J$. We will write $F_{2l}$ to mean the polynomial $f_{2l}(a)$ viewed as a polynomial in $x_i$ for $i \in I \cup J$.

\subsubsection*{Step 3}

Let $P^{(0)} = x_{2(h-1)} - F_{2(h-1)}$. By Step 2, $P^{(0)}$ is a polynomial in terms of $x_i$ for $i \in I_0 \cup J = I_0 \cup J_0$. Recall from Step 1 that $\beta(x,g) \in Y_h$ if and only if $s(x) \cdot y \in X_h$. If $s(x) \cdot y \in X_h$, then we must have $F_{2(h-1)} = f_{2(h-1)}(s(x) \cdot y) = 0$, so $P^{(0)} = x_{2(h-1)}$. Thus we see that the $2(h-1)$th coordinate of $\beta^{-1}(Y_h) \subset \Affine[I \cup J]$ is uniquely determined by the other coordinates. We can therefore rewrite this scheme as a subscheme $S^{(0)}$ of $\Affine[I_0 \cup J_0]$. Furthermore the morphism $P^{(0)} \from S^{(0)} \to \GG_a$ is exactly the restriction of the morphism $P \from \beta^{-1}(Y_h) \to \GG_a$ introduced in Step 1. Thus
\begin{equation*}
\boxed{H_c^i(\beta^{-1}(Y_h), P^* \Loc_\psi) \cong H_c^i(S^{(0)}, (P^{(0)})^* \Loc_\psi).}
\end{equation*}

\subsubsection*{Step 4}
In the next two steps, we describe an inductive application of Proposition \ref{p:2.10rev}.

We apply Proposition \ref{p:2.10rev} to the following set-up:
\begin{enumerate}[label=\textbullet]
\item
Let $S^{(0)}$ be as in Step 3. Explicitly, it is the subscheme of $\Affine[I_0 \cup J_0]$ defined by the equations $F_{2l} = 0$ for $l < h-1$, where $\Affine[I_0 \cup J_0]$ is the affine space $\Affine^{2(h-1)-1}$ with coordinates labelled by $x_i$ for $i \in I_0 \cup J_0$.

\item
Let $S_2^{(0)}$ denote the subscheme of $\Affine[I_1 \cup J_0]$ defined by the same equations.

\item
Note that $S^{(0)} = S_2^{(0)} \times \Affine[\{2(h-1) - 1\}]$, since $x_{2(h-1)-1}$ has no contribution to $F_{2l}$ for $l < h-1$.

\item
Let $f \from S_2^{(0)} \to \GG_a$ be defined as the projection to $x_1$.

\item
For $v \in S_2^{(0)}$ and $w = x_{2(h-1) - 1}$, we may write
\begin{equation} \label{InductPoly}
P^{(0)}(v,w) = f(v)^q w - f(v)^{q^2} w^q + P_2^{(0)}(v).
\end{equation}
(We justify this later.)

\item
Let $S_3^{(0)} \subset S_2^{(0)} \subset \Affine[I_1 \cup J_0]$ be the subscheme defined by $f = x_1 = 0$ and let $P_3^{(0)} \colonequals P_2^{(0)}|_{S_3^{(0)}} \from S_3^{(0)} \to \GG_a$.
\end{enumerate}
Then by Proposition \ref{p:2.10rev}, for all $i \in \ZZ$,
\begin{equation*}
\boxed{H_c^i(S^{(0)}, (P^{(0)})^* \Loc_\psi) \cong H_c^{i-2}(S_3^{(0)}, (P_3^{(0)})^* \Loc_\psi)(-1)}
\end{equation*}
as vector spaces equipped with an action of $\Fr_{q^2}$, where the Tate twist $(-1)$ means that the action of $\Fr_{q^2}$ on $H_c^{i-2}(S_3^{(0)}, (P_3^{(0)})^* \Loc_\psi)$ is multiplied by $q^2$.
 
Before we proceed, we must show that one can indeed decompose $P^{(0)}$ into the form described in Equation \eqref{InductPoly}. Using Theorem \ref{t:Xhpolys} together with the explicit equations for the coordinates of the product $s(x) \cdot y \equalscolon a$ described in Equation \eqref{e:aiform}, we see that the only terms in $x_{2(h-1)} - f_{2(h-1)}(a)$ involving $y_{2(h-1) - 1}$ occur in the expression
\begin{equation*}
-(a_{2(h-1)}^{q^2} - a_{2(h-1)}) + a_1^q(a_{2(h-1) - 1}^{q^2} - a_{2(h-1) - 1}) + a_{2(h-1) - 1}^q(a_1^{q^2} - a_1)
\end{equation*}
and are exactly
\begin{equation*}
-((x_1 y_{2(h-1) - 1}^q)^{q^2} - (x_1 y_{2(h-1) - 1}^q)) + x_1^q(y_{2(h-1) - 1}^{q^2} - y_{2(h-1) - 1}) + y_{2(h-1) - 1}^q(x_1^{q^2} - x_1).
\end{equation*}
Thus the only terms involving $x_{2(h-1)-1}$ in $P^{(0)}$ are
\begin{equation*}
x_1^q x_{2(h-1) - 1} - x_1^{q^2} x_{2(h-1) - 1}^q.
\end{equation*}
Moreover, the remaining terms in $P^{(0)}$ only involve indices in $I_1 \cup J_0$. This proves that the decomposition in \eqref{InductPoly} exists.

\begin{remark}
Note that since $S_3^{(0)}$ was defined to be the subscheme of $S_2^{(0)} \subset \Affine[I_1 \cup J_0]$ cut out by $x_1$, we can actually view $S_3^{(0)}$ as a subscheme of $\Affine[I_1 \cup J_1]$. Thus what we have done in this step is reduce a computation about a subscheme of $\Affine[I_0 \cup J_0]$ to a computation about a subscheme of $\Affine[I_1 \cup J_1]$.
\end{remark}

\subsubsection*{Step 5}

We now apply Proposition \ref{p:2.10rev} again. We apply it to the following set up.
\begin{enumerate}[label=\textbullet]
\item
Let $S^{(1)} \colonequals S_3^{(0)} \subset \Affine[I_1 \cup J_1]$.

\item
Let $S_2^{(1)}$ be the subscheme of $S^{(1)}$ cut out by $x_{2(h-1) - 3}$ so that we can in fact view $S_2^{(1)}$ as a subscheme of $\Affine[I_2 \cup J_1]$.

\item
Note that $S^{(1)} = S_2^{(1)} \times \Affine[\{2(h-1) - 3\}]$ since $x_1 = x_{2(h-1)-1} = 0$ implies that $x_{2(h-1)-3}$ does not contribute to $F_{2l}$ for $l < h-1$.

\item
Let $f \from S_2^{(1)} \to \GG_a$ be defined as the projection to $x_3$.

\item
For $v \in S_2^{(1)}$ and $w = x_{2(h-1) - 3}$, we may write
\begin{equation}\label{InductPoly2}
P^{(1)}(v,w) \colonequals P_3^{(0)}(v,w) = f(v)^qw - f(v)^{q^2}w^q + (f(v)w^q - (f(v)w^q)^{q^2}) + P_2^{(1)}.
\end{equation}
(We justify this step later.) Note that in the notation of Proposition \ref{p:2.10rev}, we have $\alpha(v,w) = -f(v) w^q$.

\item
Let $S_3^{(1)} \subset S_2^{(1)} \subset \Affine[I_2 \cup J_1]$ be the subscheme defined by $f = x_3 = 0$ and let $P_3^{(1)} \colonequals P_2^{(1)}|_{S_3^{(1)}} \from S_3^{(1)} \to \GG_a$.
\end{enumerate}
Then by Proposition \ref{p:2.10rev}, for all $i \in \ZZ$,
\begin{equation*}
\boxed{H_c^i(S^{(1)}, (P^{(1)})^* \Loc_\psi) \cong H_c^{i-2}(S_3^{(1)}, (P_3^{(1)})^* \Loc_\psi)(-1)}
\end{equation*}
as vector spaces equipped with an action of $\Fr_{q^2}$.

As before, we must verify that one can indeed decompose $P^{(1)}$ into the form described in Equation \eqref{InductPoly2}. This computation will turn out to be very similar to the computation in Step 4. Again using Theorem \ref{t:Xhpolys} together with Equation \eqref{e:aiform}, we see that once we set $x_1 = 0$ and $x_{2(h-1) - 1} = 0$, the only terms in $x_{2(h-1)} - f_{2(h-1)}(s(x) \cdot y)$ involving $y_{2(h-1) - 3}$ occur in the expression
\begin{equation*}
-(a_{2(h-1)}^{q^2} - a_{2(h-1)}) + a_3^q(a_{2(h-1) - 3}^{q^2} - a_{2(h-1) - 3}) + a_{2(h-1) - 3}^q(a_3^{q^2} - a_3)
\end{equation*}
and are
\begin{equation*}
-((x_3 y_{2(h-1) - 3}^{q^3})^{q^2} - (x_3 y_{2(h-1) - 3}^{q^3})) + x_3^q(y_{2(h-1) - 3}^{q^2} - y_{2(h-1) - 3}) + y_{2(h-1) - 3}^q(x_3^{q^2} - x_3).
\end{equation*}
Thus the only terms involving $x_{2(h-1)-3}$ in $P^{(0)}$ are
\begin{equation*}
x_3^q x_{2(h-1)-3} - x_3^{q^2} x_{2(h-1)-3}^q + x_3 x_{2(h-1)-3}^q - x_3^{q^2} x_{2(h-1)-3}^{q^3}.
\end{equation*}
Moreover, the remaining terms in $P^{(1)}$ only involve indices in $I_1 \cup J_0$. This verifies \eqref{InductPoly2}.

\begin{remark}
Each time we iterate Step 5, it will be of the following form. Let $k$ be a positive odd integer $< (h-1)$. We will have $S = S_2 \times \Affine[\{2(h-1)-k\}]$ with $f \from S_2 \to \GG_a$ defined as the projection to $x_k$. For $v \in S_2$ and $w = x_{2(h-1)-k}$, we may write
\begin{equation}\label{InductPolyk}
P(v,w) = f(v)^q w - f(v)^{q^2} w^q + (f(v)g(w) - (f(v)g(w))^{q^2}) + P_2,
\end{equation}
where $g(w) = w^{q^{k-2}} + w^{q^{k-4}} + \cdots + w$. (In the notation of Proposition \ref{p:2.10rev}, $\alpha(v,w) = -f(v)g(w).$) Let $S_3 \subset S_2$ be the subscheme defined by $f = x_k = 0$ and let $P_3 = P_2|_{S_3} \from S_3 \to \GG_a$. Then by Proposition \ref{p:2.10rev},
\begin{equation*}
\boxed{
H_c^i(S, P^*\Loc_\psi) \cong H_c^{i-2}(S_3, P_3^* \Loc_\psi)(-1)
}
\end{equation*}
as vector spaces equipped with an action of $\Fr_{q^2}$. To see \eqref{InductPolyk}, observe that once we set $x_l = x_{2(h-1)-l} = 0$ for $l$ odd and $l < k$, the only terms in $x_{2(h-1)} - f_{2(h-1)}(s(x) \cdot y)$ involving $y_{2(h-1)-k}$ occur in the expression
\begin{equation*}
-(a_{2(h-1)}^{q^2} - a_{2(h-1)}) + a_k^q(a_{2(h-1) - k}^{q^2} - a_{2(h-1) - k}) + a_{2(h-1) - k}^q(a_k^{q^2} - a_k).
\end{equation*}
Thus we see that the only terms involving $y_{2(h-1) - k}$ are
\begin{equation*}
-((x_k y_{2(h-1) - k}^{q^k})^{q^2} - (x_k y_{2(h-1) - k}^{q^k})) + x_k^q(y_{2(h-1) - k}^{q^2} - y_{2(h-1) - k}) + y_{2(h-1) - k}^q(x_k^{q^2} - x_k),
\end{equation*}
which simplifies to
\begin{align*}
-(x_k^{q^2} &y_{2(h-1) - k}^{q^{k+2}} - x_k y_{2(h-1) - k}^{q^k}) + x_k^q x_{2(h-1)-k} + y_{2(h-1) - k}^q(x_k^{q^2} - x_k) \\
&= -x_k^{q^2}(y_{2(h-1) - k}^{q^{k+2}} - y_{2(h-1) - k}^q) + x_k(y_{2(h-1) - k}^{q^k} - y_{2(h-1) - k}) + x_k^q x_{2(h-1)-k} \\
&= x_k^q x_{2(h-1)-k} - x_k^{q^2} x_{2(h-1)-k}^q \\
&\qquad\qquad+ \Big(x_k(x_{2(h-1)-k}^{q^{k-2}} + x_{2(h-1)-k}^{q^{k-4}} + \cdots + x_{2(h-1)-k}) \\
&\qquad\qquad - x_k^{q^2}(x_{2(h-1)-k}^{q^k} + x_{2(h-1)-k}^{q^{k-2}} + \cdots + x_{2(h-1)-k}^{q^2})\Big).
\end{align*}
This verifies \eqref{InductPolyk} and allows us to use Proposition \ref{p:2.10rev} to iterate the induction.
\end{remark}

\subsubsection*{Step 6, Odd Case}

Iterating Step 5, we reduce the computation about the cohomology of $S^{(0)}$ to a computation about the cohomology of $S^{((h-1)/2)} \colonequals S_3^{((h-3)/2)}$, which is the subscheme of $\Affine[I_{(h-1)/2} \cup J_{(h-1)/2}]$ defined by the equations
\begin{equation*}
x_2^{q^2} - x_2 = 0, \quad x_4^{q^2} - x_4 = 0, \quad \ldots, \quad x_{2(h-2)}^{q^2} - x_{2(h-2)} = 0.
\end{equation*}
These equations come from the equations given in Theorem \ref{t:Xhpolys} together with setting $x_i = 0$ for all odd $i$. Recalling that $I_{(h-1)/2} \cup J_{(h-1)/2} = \{2, 4, \ldots, 2(h-2)\}$, we see that $S^{((h-1)/2)}$ is a 0-dimensional scheme with $q^{2(h-2)}$ points and $\Fr_{q^2}$ acts trivially on the cohomology. Therefore
\begin{equation*}
\boxed{\dim H_c^i(S^{((h-1)/2)}, (P^{((h-1)/2)})^* \Loc_\psi) = 
\begin{cases}
q^{2(h-2)} & \text{if $i = 0$,} \\
0 & \text{otherwise.}
\end{cases}}
\end{equation*}

\subsubsection*{Step 6, Even Case}

Iterating Step 5, we reduce the computation about the cohomology of $S^{(0)}$ to a computation about the cohomology of $S^{((h-2)/2)} \colonequals S_3^{((h-4)/2)}$, which is the subscheme of $\Affine[I_{(h-2)/2} \cup J_{(h-2)/2}]$ defined by the equations
\begin{equation*}
x_2^{q^2} - x_2 = 0, \quad x_4^{q^2} - x_4 = 0, \quad \ldots, \quad x_{2(h-2)}^{q^2} - x_{2(h-2)} = 0.
\end{equation*}
Recalling that $I_{(h-2)/2} \cup J_{(h-2)/2} = \{2, 4, \ldots, 2(h-2)\} \cup \{h-1\}$, we see that $S^{((h-2)/2)}$ is a one-dimensional scheme. Moreover $P^{((h-2)/2)}$ is the morphism
\begin{equation*}
P^{((h-2)/2)} \from S^{((h-2)/2)} \to \GG_a, \qquad (x_i)_{i \in I_{(h-2)/2} \cup J_{(h-2)/2}} \mapsto x_{h-1}^q(x_{h-1}^{q^2} - x_{h-1}).
\end{equation*}
The above shows that
\begin{equation*}
H_c^i(S^{((h-2)/2)}, (P^{((h-2)/2)})^* \Loc_\psi) \cong H_c^i(\GG_a, P^* \Loc_\psi)^{\oplus q^{2(h-2)}},
\end{equation*}
where the morphism $P$ is defined as
\begin{equation*}
P \from \GG_a \to \GG_a, \qquad x \mapsto x^q(x^{q^2} - x).
\end{equation*}
We now compute the right-hand-side cohomology groups in the same way as in Sections 6.5 and 6.6 in \cite{BW11}. We may write $P = p_1 \circ p_2$ where $p_1(x) = x^q - x$ and $p_2(x) = x^{q+1}$. Since $p_1$ is a group homomorphism, then $p_1^* \Loc_\psi \cong \Loc_{\psi \circ p_1},$ where $\Loc_{\psi \circ p_1}$ is the multiplicative local system on $\GG_a$ corresponding to the additive character $\psi \circ p_1 \from \F \to \overline \QQ_\ell^\times.$ By assumption, $\psi$ has trivial $\Gal(\F/\FF_q)$-stabilizer, and so $\psi \circ p_1$ is nontrivial. Furthermore, $\psi \circ p_1$ is trivial on $\FF_q$. Thus the character $\psi \circ p_1 \from \F \to \overline \QQ_\ell^\times$ satisfies the hypotheses of Proposition 6.12 in \cite{BW11}, and thus
\begin{equation*}
\dim H_c^i(\GG_a, P^* \Loc_\psi) = \dim H_c^i(\GG_a, p_2^*\Loc_{\psi \circ p_1}) = 
\begin{cases}
q & \text{if $i = 1$,} \\
0 & \text{otherwise.}
\end{cases}
\end{equation*}
Moreover, the Frobenius $\Fr_{q^2}$ acts on $H_c^1(\GG_a, P^* \Loc_\psi)$ via multiplication by $-q$.

Putting this together, we have
\begin{equation*}
\boxed{\dim H_c^i(S^{((h-2)/2)}, (P^{((h-2)/2)})^* \Loc_\psi) = 
\begin{cases}
q^{2(h-2) + 1} & \text{if $i = 1$,} \\
0 & \text{otherwise,}
\end{cases}}
\end{equation*}
and the Frobenius $\Fr_{q^2}$ acts on $H_c^1(\GG_a, P^* \Loc_\psi)$ via multiplication by $-q$.

\subsubsection*{Step 7}

We now put together all of the boxed equations. We have
\begin{align*}
\Hom_{\UnipF}(V_\psi, H_c^i(X_h, \overline \QQ_\ell)) 
&\cong H_c^i(\beta^{-1}(Y_h), P^* \Loc_\psi) \\
&= H_c^i(S^{(0)}, (P^{(0)})^* \Loc_\psi) \\
&\cong H_c^{i-2}(S_3^{(0)}, (P_3^{(0)})^*\Loc_\psi)(-1) \\
&= H_c^{i-2}(S^{(1)}, (P^{(1)})^* \Loc_\psi)(-1) \\
&\cong H_c^{i-2\lfloor (h-1)/2 \rfloor}(S^{(\lfloor (h-1)/2 \rfloor)}, (P^{(\lfloor (h-1)/2\rfloor)})^*\Loc_\psi)(-\lfloor (h-1)/2\rfloor)
\end{align*}
Therefore if $h$ is odd, then
\begin{equation*}
\dim \Hom_{\UnipF}(V_\psi, H_c^i(X_h, \overline \QQ_\ell)) =
\begin{cases}
q^{2(h-2)} & \text{if $i = h-1$,} \\
0 & \text{otherwise.}
\end{cases}
\end{equation*}
If $h$ is even, then
\begin{equation*}
\dim \Hom_{\UnipF}(V_\psi, H_c^i(X_h, \overline \QQ_\ell)) =
\begin{cases}
q^{2(h-2)+1} & \text{if $i = h-1$,} \\
0 & \text{otherwise.}
\end{cases}
\end{equation*}
Moreover, the Frobenius $\Fr_{q^2}$ acts on $\Hom_{\UnipF}(V_\psi, H_c^i(X_h, \overline \QQ_\ell))$ via multiplication by the scalar $(-1)^{h-1} q^{h-1}$.

\section{The Representations \texorpdfstring{$H_c^\bullet(X_h)[\chi]$}{Hc(Xh)[chi]}}
\label{s:cohomreps}

Let $K \colonequals H'(\F)$, where $H'$ is defined as in Section \ref{s:repthy}. Let $\psi \from \F \to \overline \QQ_\ell^\times$ be a character of conductor $q^2$ and let $\chi \in \cA_\psi$. In this section, we will compute the representation $\sigma_\chi \colonequals H_c^{h-1}(X_h, \overline \QQ_\ell)[\chi]$ by computing its restriction to $H \colonequals H(\F)$. It will turn out that $\sigma_\chi$ is irreducible and therefore by Corollary \ref{c:detbyH}, determining $\sigma_\chi$ as a representation of $H$ will be enough to determine $\sigma_\chi$ as a representation of $\UnipF$.

Recall that the left action of $U_L^1/U_L^h$ and right action of $\UnipF$ on $X_h$ induce a $(U_L^1/U_L^h \times \UnipF)$-module structure on $H_c^{h-1}(X_h, \overline \QQ_\ell)$. The primary object of interest in this section is the subspace $H_c^{h-1}(X_h, \overline \QQ_\ell)_{\chi_1, \chi_2} \subset H_c^{h-1}(X_h, \overline \QQ_\ell)$ wherein $U_L^1/U_L^h \times H(\F)$ acts by $\chi_1 \otimes \chi_2$. Here, $\chi_1$ and $\chi_2$ are characters of $U_L^1/U_L^h \cong H(\F)$.

We first present the main theorems of this section.

\begin{theorem}\label{t:irrep}
Let $\psi \from \F \to \overline \QQ_\ell^\times$ be a character of $q^2$ and let $\chi \in \cA_\psi$. Then $H_c^{h-1}(X_h, \overline \QQ_\ell)[\chi]$ is an irreducible representation of $\UnipF$.
\end{theorem}

Theorem \ref{t:irrep} proves Conjecture 5.18 of \cite{B12} (this was restated in Section \ref{s:introduction} of this paper). We prove this in Section \ref{s:pfirrep}. However, it will be important to know exactly which representation $H_c^{h-1}(X_h, \overline \QQ_\ell)[\chi]$ is. Thus we need the following finer statement

\begin{theorem}\label{t:interdim}
Let $\psi \from \F \to \overline \QQ_\ell^\times$ be a character of conductor $q^2$ and let $\chi_1 \in \cA_\psi$. Then for any character $\chi_2 \from U_L^1/U_L^h \to \overline \QQ_\ell^\times$,
\begin{equation*}
\dim H_c^{h-1}(X_h, \overline \QQ_\ell)_{\chi_1, \chi_2} = (-1)^h \left(q \cdot \langle \chi_1, \chi_2 \rangle+ \sum_{i=1}^{h-2} (-1)^i (q + 1) \cdot \langle \chi_1, \chi_2 \rangle_{G_i} \right).
\end{equation*}
\end{theorem}

We prove this in Section \ref{s:pfinterdim}. Note that Theorem \ref{t:irrep} is a consequence of Theorem \ref{t:interdim}. However, because the proof of Theorem \ref{t:interdim} is complicated, we hope that proving Theorem \ref{t:irrep} independently (in Section \ref{s:pfirrep}) will illustrate the flavor of the computation in a simpler situation. 

As a consequence of Theorem \ref{t:interdim}, we have

\begin{theorem}\label{t:cohomdesc}
Let $\psi \from \F \to \overline \QQ_\ell^\times$ be a character of conductor $q^2$ and let $\chi \in \cA_\psi$. Then
\begin{equation*}
H_c^i(X_h, \overline \QQ_\ell)[\chi] = 
\begin{cases} 
\rho_\chi & \text{if $i = h-1$,} \\
0 & \text{otherwise.}
\end{cases}
\end{equation*}
\end{theorem}

\begin{proof}
We know from Theorem \ref{t:homspace} that $H_c^i(X_h, \overline \QQ_\ell)[\chi] = 0$ if $i \neq h-1$. Let $\sigma_\chi = H_c^{h-1}(X_h, \overline \QQ_\ell)[\chi]$. By construction, $\sigma_\chi$ is a representation of $\UnipF$ wherein $H_{2(h-1)}(\F)$ acts by some character $\psi$ with conductor $q^2$.

Theorem \ref{t:interdim} implies that
\begin{equation}\label{e:decomp}
\sigma_\chi = (-1)^h(q \cdot \chi + \sum_{i=1}^{h-2} (-1)^i(q+1) \cdot \Ind_{G_i}^H(\chi)),
\end{equation}
which implies that $\dim \sigma_\chi = q^{h-1}$. By Theorem \ref{t:bij}, we know that if $\rho$ is an irreducible representation of $\UnipF$ such that $H_{2(h-1)}(\F)$ acts by $\psi$, then $\dim \rho = q^{h-1}$. Therefore $\sigma_\chi$ is irreducible. (Note that instead of the reasoning in this paragraph, we could have referenced Theorem \ref{t:irrep}.)

Thus Corollary \ref{c:detbyH} implies that the isomorphism class of $\sigma_\chi$ is determined by $\sigma_\chi$. Finally, Equation \eqref{e:decomp} and Theorem \ref{t:decomp} allow us to conclude that $\sigma_\chi \cong \rho_\chi$ as $\UnipF$-representations.
\end{proof}

\subsection{Proof of Thoerem \ref{t:interdim}}\label{s:pfinterdim}

By Theorem \ref{t:homspace}, we can apply Lemma 2.13 of \cite{B12} to our situation and get
\begin{align}\label{e:DLfixedpts}
\dim H_c^{h-1}(X_h, \overline \QQ_\ell)_{\chi_1, \chi_2} 
&= \frac{1}{q^{5(h-1)}} \sum_{g, \gamma \in H} \chi_1(g)^{-1} \chi_2(\gamma) \cdot N(g, \gamma) \\ \label{e:case1}
&= \frac{1}{q^{5(h-1)}} \sum_{\substack{g, \gamma \in H \\ \text{$g_{2i} = h_{2i}$ for $1 \leq i \leq h-2$}}} \chi_1(g)^{-1} \chi_2(\gamma) \cdot N(g,\gamma) \\ \label{e:case2}
&\qquad+ \frac{1}{q^{5(h-1)}} \sum_{\substack{g, \gamma \in H \\ \text{$\exists \, k < h-2$ s.t.\ $h_{2k} \neq g_{2k}$}}} \chi_1(g)^{-1} \chi_2(\gamma) \cdot N(g, \gamma),
\end{align}
where
\begin{align*}
g &= 1 + \sum g_i \tau^i \\
\gamma &= 1 + \sum h_i \tau^i \\
N(g,\gamma) &= \#\{x \in X_h(\overline \FF_q) : g * \Fr_{q^2}(x) = x \cdot \gamma\}.
\end{align*}
We compute line \eqref{e:case1} in Proposition \ref{p:case1} and line \eqref{e:case2} in Proposition \ref{p:case2}.

In both of these situations, we need to analyze the set of solutions to a large system of equations. These equations are:
\begin{align*}
\label{*}
x_{2k}^{q^2} - x_{2k} &= \sum_{i=1}^{2k-1} (-1)^{(i+1)} x_i^q(x_{2k-i}^{q^2} - x_{2k-i}) & \text{for $1 \leq k \leq h-1$} \tag{$*$} \\
x_{2k}^{q^2} - x_{2k} &= \sum_{i=1}^{k} \big[(h_{2i} - g_{2i})x_{2k-2i} \\
\label{**}
&\qquad\qquad- g_{2i}(x_{2k-2i}^{q^2} - x_{2k-2i})\big] & \text{for $1 \leq k \leq h-1$} \tag{$**$} \\
x_{2k+1}^{q^2} - x_{2k+1} &= \sum_{i=1}^{k} \big[(h_{2i}^q - g_{2i})x_{2k+1-2i} \\
\label{***}
&\qquad\qquad- g_{2i}(x_{2k+1-2i}^{q^2} - x_{2k+1-2i})\big] & \text{for $1 \leq k \leq h-2$} \tag{$***$}
\end{align*}
The equations of Type \eqref{*} are equivalent to the condition that $x \in X_h(\overline \FF_p)$ (this was proved in Theorem \ref{t:Xhpolys}). The equations of Type \eqref{**} and Type \eqref{***} are equivalent to the condition that $g * \Fr_{q^2}(x) = x \cdot h$, where we write $g = 1 + \sum_{i=1}^{h-1} g_{2i} \tau^{2i}$ and similarly for $h$. (These two actions were defined in Section \ref{s:notation}.) We will call the above equations the Type \eqref{*} equation for $2k$, the Type \eqref{**} equation for $2k$, and the Type\eqref{***} for $2k+1$, respectively. Furthermore, when we refer to these equations as polynomials, we view them as multivariate polynomials in the $x_i$'s.

\subsubsection{Computation of Line \eqref{e:case1}}\label{s:case1}

We prove a sequence of lemmas that build up to the following 

\begin{proposition}\label{p:case1}
Let $\psi \from \F \to \overline \QQ_\ell^\times$ be a character with conductor $q^2$ and let $\chi_1 \in \cA_\psi$. Then
\begin{equation*}
\text{Line \eqref{e:case1}} = (-1)^h \left(\langle \chi_1, \chi_2 \rangle \cdot q+ \sum_{i=1}^{h-2} (-1)^i \langle \chi_1, \chi_1 \rangle_{G_i} \cdot (q+1) \right).
\end{equation*}
\end{proposition}

\begin{lemma}\label{l:evenvan}
Assume that $h_{2i} = g_{2i}$ for $1 \leq i \leq h-2$. Then $x_{2k}^{q^2} - x_{2k} = 0$ for $1 \leq k \leq h-2$.
\end{lemma}

\begin{proof}
This is just a simple execution of induction. It is clear that this is true for $k = 1$. Now assume that it is true for $k < h-2$, and we can show that it is true for $k+1$. Indeed, by assumption, $h_{2i} = g_{2i}$ for $1 \leq i \leq h-2$, so the induction hypothesis implies that the Type \eqref{**} equation for $2(k+1)$ simples to
\begin{align*}
x_{2(k+1)}^{q^2} - x_{2(k+1)} 
&= \sum_{i=1}^{k+1} h_{2i}x_{2k-2i} - g_{2i}x_{2k-2i}^{q^2} = \sum_{i=1}^{k+1} (h_{2i} - g_{2i})x_{2k-2i} = 0. \qedhere
\end{align*}
\end{proof}

\begin{importantremark}
The key observation that we will capitalize on in the next few lemmas is the following. The Type \eqref{*} equations ``intertwine'' the equations of Type \eqref{**} and Type \eqref{***}. Using Lemma \ref{l:evenvan} and substituting Type \eqref{**} and \eqref{***} equations into Type \eqref{*} equations, we have, for $1 \leq k \leq h-1$,
\begin{align*}
h_{2k} - g_{2k}
&= \sum_{\substack{\text{$i$ odd} \\ 1 \leq i \leq 2k-3}} x_i^q \left(\sum_{\substack{\text{$j$ odd} \\ 1 \leq j \leq 2k-2-i}} (g_{2k-i-j}^q - g_{2k-i-j}) x_j - g_{2k-i-j}(x_j^{q^2} - x_j) \right) \\
&= \sum_{\substack{\text{$i$ odd} \\ 1 \leq i \leq 2k-3}} x_i^q \left(\sum_{\substack{\text{$j$ odd} \\ 1 \leq j \leq 2k-2-i}} (g_{2k-i-j}^q - g_{2k-i-j}) x_j\right)  - \sum_{2 \leq l \leq 2k-2} g_{2k-l} (x_l^{q^2} - x_l).
\end{align*}
Thus:
\begin{align*}
h_{2k} - g_{2k} &= \sum_{\substack{\text{$i$ odd} \\ 1 \leq i \leq 2k - 3}} x_i^q\left(\sum_{\substack{\text{$j$ odd} \\ 1 \leq j \leq 2k-2-i}} (g_{2k-i-j}^q - g_{2k-i-j}) x_j\right) & \text{for $1 \leq k \leq h-1$} \tag{$\dagger$} \label{dagger}
\end{align*}
Recall that $h_{2k} - g_{2k} = 0$ for $1 \leq k \leq h-2$ by assumption.
\end{importantremark}

\begin{lemma}\label{l:easycond}
Assume that $h_{2i} = g_{2i}$ for $1 \leq i \leq h-2$. If $g_{2k} \in \FF_q$ for $1 \leq k \leq h-2$, then 
\begin{equation*}
N(g,\gamma) = \begin{cases}
q^{4(h-1)} & \text{if $g_{2(h-1)} = h_{2(h-1)}$} \\
0 & \text{otherwise.}
\end{cases}
\end{equation*}
\end{lemma}

\begin{proof}
If $g_{2k} \in \FF_q$ for $1 \leq k \leq h-2$, then all the coefficients in the Type \eqref{dagger} equations vanish, imposing no further conditions on the $x_i$'s but forcing $h_{2(h-1)} - g_{2(h-1)} = 0$. Therefore the number of solutions to the equations of Types \eqref{*} through \eqref{***} are the solutions to $x_1^{q^2} - x_1 = 0$ together with the the solutions to \eqref{**} and \eqref{***}. Thus we have $q^2$ choices for each $x_k$, and so
\begin{equation*}
N(g,\gamma) = \begin{cases} q^{2(2(h-1))} & \text{if $g_{2(h-1)} = h_{2(h-1)}$,} \\
0 & \text{otherwise.}
\end{cases}\qedhere
\end{equation*}
\end{proof}

\begin{lemma}\label{l:oddcond}
Assume that $h_{2i} = g_{2i}$ for $1 \leq i \leq h-2$. Pick $k \geq 1$ and suppose that $g_{2i} \in \FF_q$ for $1 \leq i \leq h - (2k+1)$ and $g_{2(h-2k)} \notin \FF_q$. Then
\begin{equation*}
N(g,\gamma) = \begin{cases}
q^{2(2(h-1) - k)} & \text{if $h_{2(h-1)} = g_{2(h-1)}$,} \\
(q+1) q^{2(2(h-1) - k)} & \text{if $0 \neq h_{2(h-1)} - g_{2(h-1)} \in \ker \Tr_{\F/\FF_q}$,} \\
0 & \text{otherwise.}
\end{cases}
\end{equation*}
\end{lemma}

\begin{proof}
All the coefficients in the Type \eqref{dagger} equations for $2 \leq 2j \leq 2(h-2k)$ vanish so we get the empty conditions $h_{2j} - g_{2j} = 0$, yielding no additional restrictions on any of the $x_i$'s. The first nontrivial restriction comes from the Type \eqref{dagger} equation for $2(h-2k+1)$:
\begin{equation*}
h_{2(h-2k+1)} - g_{2(h-2k+1)} = x_1^q(g_{2(h-2k)}^q - g_{2(h-2k)})x_1.
\end{equation*}
By assumption, the left-hand side vanishes and the coefficient of $x_1$ on the right-hand side is nonvanishing, which forces $x_1 = 0$. This extra condition implies that the subsequent Type \eqref{dagger} equation (i.e.\ the Type \eqref{dagger} equation for $2(h-2k+2)$) simplifies to
\begin{equation*}
h_{2(h-2k+2)} - g_{2(h-2k+2)} = 0,
\end{equation*} 
imposing no additional constraints on the $x_i$'s. The subsequence Type \eqref{dagger} equation simplifies to
\begin{equation*}
h_{2(h-2k+3)} - g_{2(h-2k+3)} = x_3^q(g_{2(h-2k)}^q - g_{2(h-2k)}) x_3,
\end{equation*}
which forces $x_3 = 0$ since the left-hand side vanishes. This continues until the equation
\begin{equation*}
h_{2(h-1)} - g_{2(h-1)} = x_{2k-1}^q(g_{2(h-2k)}^q - g_{2(h-2k)})x_{2k-1}.
\end{equation*}
Thus we see that regardless of whether $h_{2(h-1)}$ and $g_{2(h-1)}$ agree, Equations \eqref{dagger} force $x_1 = 0, x_3 = 0, \ldots, x_{2k-3} = 0$. Note that equation \eqref{***} for $2k-1$ implies that $x_{2(k-1)} \in \F$. Thus we see that if $h_{2(h-1)} = g_{2(h-1)}$, then this last displayed equations implies that we have the additional constraint that $x_{2k-1} = 0$. Furthermore, this gives $q+1$ choices for $x_{2k-1}$ if $h_{2(h-1)} - g_{2(h-1)} \in \ker \Tr_{\F/\FF_q}$ and no choices for $x_{2k-1}$ otherwise. We see from equations of Type \eqref{**} and \eqref{***} that regardless of whether $h_{2(h-1)}$ and $g_{2(h-1)}$ agree, we have $q^2$ choices for the remaining $x_i$'s (i.e.\ for even $i \leq 2k-2$ and all $i > 2k-1$). Therefore
\begin{equation*}
N(g,\gamma) = \begin{cases}
q^{2k-2} \cdot 1 \cdot q^{2(2h-2k-1)} & \text{if $h_{2(h-1)} = g_{2(h-1)}$,} \\
q^{2k-2} \cdot (q+1) \cdot q^{2(2h-2k-1)} & \text{if $0 \neq h_{2(h-1)} - g_{2(h-1)} \in \ker \Tr_{\F/\FF_q}$,} \\
0 & \text{otherwise.}
\end{cases}\qedhere
\end{equation*}
\end{proof}

\begin{lemma}\label{l:evencond}
Assume that $h_{2i} = g_{2i}$ for $1 \leq i \leq h-2$. Pick $k \geq 1$ and suppose that $g_{2i} \in \FF_q$ for $1 \leq i \leq h - (2k+2)$ and $g_{2(h-(2k+1))} \notin \FF_q$. Then
\begin{equation*}
\# =
\begin{cases}
q^{2(2(h-1) - k)} & \text{if $h_{2(h-1)} = g_{2(h-1)}$}, \\
0 & \text{otherwise.}
\end{cases}
\end{equation*}
\end{lemma}

\begin{proof}
We argue as in the proof of Lemma \ref{l:oddcond}. All the coefficients in the Type \eqref{dagger} equations for $2 \leq 2j \leq 2(h-2k-1)$ vanish so we get the empty conditions $h_{2j} - g_{2j} = 0$, yielding no additional restrictions on any of the $x_i$'s. The first nontrivial restriction comes from the Type \eqref{dagger} equation for $2(h-2k)$:
\begin{equation*}
h_{2(h-2k)} - g_{2(h-2k)} = x_1^q(g_{2(h-(2k+1))}^q - g_{2(h-(2k+1))})x_1.
\end{equation*}
Thus, for odd $l$ with $1 \leq l \leq 2k-1$, the Type \eqref{dagger} equation for $2(h-(2k+1)+l)$ reduces to
\begin{equation*}
h_{2(h-(2k+1)+l)} - g_{2(h-(2k+1)+l)} = x_l^q(g_{2(h-(2k+1))}^q - g_{2(h-(2k+1))})x_l,
\end{equation*}
which forces $x_l = 0$. But then this implies that the Type \eqref{dagger} equation for $2(h-1)$ simplifies to
\begin{equation*}
h_{2(h-1)} - g_{2(h-1)} = 0.
\end{equation*}
Thus there are no solutions if $h_{2(h-1)} \neq g_{2(h-1)}$. If $h_{2(h-1)} = g_{2(h-1)}$, then we get $q^2$ solutions for each of the $x_i$'s other than the odd $i$, $1 \leq i \leq 2k-1$. Therefore
\begin{equation*}
N(g,\gamma) = \begin{cases}
q^{2k-2} \cdot q^{2(2h-2k-1)} & \text{if $h_{2(h-1)} = g_{2(h-1)}$,} \\
0 & \text{otherwise.}
\end{cases}\qedhere
\end{equation*}
\end{proof}

\begin{lemma}\label{l:psisum}
Assume that $h_{2i} = g_{2i}$ for $1 \leq i \leq h-2$ and let $\psi$ be the restriction of $\chi_1$ to $U_L^{h-1}/U_L^h$. By assumption, we can write $\gamma = g \cdot (1 + \epsilon \tau^{2(h-1)}).$ Then
\begin{equation*}
\sum_{\epsilon \in \F} \psi(\epsilon) \cdot N(g,\gamma) = 
\begin{cases}
q^{4(h-1)} & \text{if $g \in G_{h-2}$,} \\
-q \cdot q^{2(2(h-1) - k)} & \text{if $g \in G_{2(h-(2k+1))} \smallsetminus G_{2(h-(2k+2))}$, $k \geq 1$,} \\
q^{2(2(h-1) - k)} & \text{if $g \in G_{2(h-(2k+2))} \smallsetminus G_{2(h-(2k+3))},$ $k \geq 1$.}
\end{cases}
\end{equation*}
\end{lemma}

\begin{proof}
By assumption $\psi$ is a nontrivial additive character $\F \to \overline \QQ_\ell^\times$. Recalling the definition of $G_i$ from Section \ref{s:morerepthy}, it is easy to see that the first and third cases of the lemma follow from Lemmas \ref{l:easycond} and \ref{l:evencond} respectively. To see the second case, recall from Lemma \ref{l:oddcond} that we have
\begin{equation*}
N(g,\gamma) = \begin{cases}
q^{2(2(h-1) - k)} & \text{if $\epsilon = 0$,} \\
(q+1)q^{2(2(h-1) - k)} & \text{if $0 \neq \epsilon \in \ker \Tr_{\F/\FF_q}$,} \\
0 & \text{otherwise.}
\end{cases}
\end{equation*}
Thus
\begin{align*}
\sum_{\epsilon \in \F} \psi(\epsilon) \cdot N(g,\gamma) 
&= -q \cdot q^{2(2(h-1) - k)} + \sum_{\epsilon \in \ker\Tr_{\F/\FF_q}} \psi(\epsilon) \cdot (q+1)q^{2(2(h-1) - k)} \\
&= -q \cdot q^{2(2(h-1) - k)}.\qedhere
\end{align*}
\end{proof}

We are now ready to prove Proposition \ref{p:case1}.

\begin{proof}[Proof of Proposition \ref{p:case1}]
First notice that by the assumption $h_{2i} = g_{2i}$ for $1 \leq i \leq h-2$, we can write Line \eqref{e:case1} as
\begin{equation*}
\frac{1}{q^{5(h-1)}} \sum_{\substack{g \in H \\ \epsilon \in \F}} \frac{\chi_2(g)}{\chi_1(g)} \cdot \psi(\epsilon) \cdot N(g,\gamma),
\end{equation*}
where $\gamma = g \cdot (1 + \epsilon \tau^{2(h-1)})$. Also notice that
\begin{equation*}
\sum_{g \in H} = \sum_{g \in G_{h-2}} + \sum_{g \in G_{h-3} \smallsetminus G_{h-2}} + \cdots + \sum_{g \in G_1 \smallsetminus G_2} + \sum_{g \in H \smallsetminus G_1}.
\end{equation*}
We first analyze each of the summands. Pick $k \geq 1$. Then:
\begin{align*}
\sum_{g \in G_{h-2}} &\frac{\chi_2(g)}{\chi_1(g)} \sum_{\epsilon \in \F} \psi(\epsilon) \cdot N(g,\gamma) \\
&= \langle \chi_1, \chi_2 \rangle_{G_{h-2}} \cdot |G_{h-2}| \cdot q^{2(2(h-1))} = q^{5(h-1)} \cdot \langle \chi_1, \chi_2 \rangle_{G_{h-2}} \cdot q \\
\sum_{\substack{g \in G_{h-(2k+1)} \\ g \notin G_{h-2k}}} &\frac{\chi_2(g)}{\chi_1(g)} \sum_{\epsilon \in \F} \psi(\epsilon) \cdot N(g,\gamma) \\
&= \Big(\langle \chi_1, \chi_2 \rangle_{G_{h-(2k+1)}} \cdot |G_{h-(2k+1)}| - \langle \chi_1, \chi_2 \rangle_{G_{h-2k}} \cdot |G_{h-2k}|\Big) \cdot -q \cdot q^{2(2(h-1)-k)} 
\\
\sum_{\substack{g \in G_{h-(2k+2)} \\ g \notin G_{h-(2k+1)}}} & \frac{\chi_2(g)}{\chi_1(g)} \sum_{\epsilon \in \F} \psi(\epsilon) \cdot N(g,\gamma) \\
&= \Big(\langle \chi_1, \chi_2 \rangle_{G_{h-(2k+2)}} \cdot |G_{h-(2k+2)}| - \langle \chi_1, \chi_2 \rangle_{G_{h-(2k+1)}} \cdot |G_{h-(2k+1)}|\Big) \cdot q^{2(2(h-1)-k)} 
\end{align*}
Now we put this together. From the definition, it is easy to see that
\begin{equation*}
|G_{h-n}| = q^{h-2+n}.
\end{equation*}
We now analyze the coefficient of $\langle \chi_1, \chi_2 \rangle_{G_{h-n}}$.
\begin{enumerate}[label=\textbullet]
\item
We first handle the border cases. Recall that $H = G_0$. If $h$ is odd, then from the above we see that the coefficient of $\langle \chi_1, \chi_2 \rangle_{G_0}$ is 
\begin{equation*}
|G_0| \cdot -q \cdot q^{4(h-1) - (h-1)} = q^{5(h-1)} \cdot -q.
\end{equation*}
If $h$ is even, then from the above we see that the coefficient is 
\begin{equation*}
|G_0| \cdot q^{4(h-1) - (h-2)} = q^{5(h-1)} \cdot q.
\end{equation*}

\item
Now for the middle cases. Let $k \geq 1$. The coefficient of $\langle \chi_1, \chi_2 \rangle_{G_{h-(2k+1)}}$ is 
\begin{equation*}
|G_{h-(2k+1)}| \cdot (-q \cdot q^{2(2(h-1) - k)} - q^{2(2(h-1) - k)}) = q^{5(h-1)} \cdot (-q-1).
\end{equation*}
The coefficient of $\langle \chi_1, \chi_2 \rangle_{G_{h-(2k+2)}}$ is
\begin{equation*}
|G_{h-(2k+2)}| \cdot (q^{2(2(h-1)-k)} + q \cdot q^{2(2(h-1) - (k+1))}) = q^{5(h-1)} \cdot (q+1).
\end{equation*}
\end{enumerate}
The desired result follows.
\end{proof}

\subsubsection{Computation of Line \eqref{e:case2}}\label{s:case2}

In this subsection, we prove the following

\begin{proposition}\label{p:case2}
Let $\psi \from \F \to \overline \QQ_\ell^\times$ be a character with conductor $q^2$ and let $\chi_1 \in \cA_\psi$. Then
\begin{equation*}
\text{Line \eqref{e:case2}} = 0.
\end{equation*}
\end{proposition}

Here is the idea of the proof. First let 
\begin{equation*}
A_{g,\gamma} \colonequals \{x \in X_h(\overline \FF_p) : g * \Fr_{q^2}(x) = x \cdot \gamma \}.
\end{equation*}
We will show that a partial solution $(x_1, \ldots, x_{2(h-2)})$ extends to a full solution $(x_1, \ldots, x_{2(h-1)})$ if and only if the partial solution satisfies an equation of the form $ax^q - a^qx + a_0 = 0$ for some nonzero $a \in \F$. The $x$ in this equation will be one of the $x_k$'s. The main work is in giving a nonvanishing condition for coefficients for certain $x_k$'s which will allow us to find such an $a$. This will give us a bijection between $A_{g,h}$ and $A_{g,h+\delta \tau^{2(h-1)}}$. Once we have established this, we will be able to prove Proposition \ref{p:case2}.

\begin{lemma}\label{l:testcond}
Let $(x_1, \ldots, x_{2(h-1)})$ be a solution to the equations of type \eqref{**} and \eqref{***}. Then $(x_1, \ldots, x_{2(h-1)})$ also satisfies the equations of type \eqref{*} if and only if, for every $k$, the tuple satisfies the equation
\begin{align*}
h_{2k} - g_{2k} 
&= \sum_{\substack{1 \leq i \leq 2k-1 \\ \text{$i$ odd}}} x_i^q \left[\sum_{j=1}^{(2k-i-1)/2} (h_{2j}^q - g_{2j}) x_{2k-i-2j}\right] \\
&\qquad - \sum_{\substack{1 \leq i \leq 2k-1 \\ \text{$i$ even}}} x_i^q \left[\sum_{j=1}^{(2k-i)/2} (h_{2j} - g_{2j}) x_{2k-i-2j}\right] - \sum_{i=1}^{k-1}(h_{2i} - g_{2i})x_{2k-2i}.\label{dagger2}\tag{$\dagger\dagger$}
\end{align*}
\end{lemma}

\begin{proof}
First note that a tuple $(x_1, \ldots, x_{2(h-1)})$ satisfying \eqref{**} and \eqref{***} can be constructed as follows: pick any $x_1, x_2 \in \F$ and then notice that the equations of type \eqref{**} and \eqref{***} allow us to choose $x_k$ given $x_1, \ldots, x_{k-1}$.

Now we substitute \eqref{**} and \eqref{***} into \eqref{*}.
\begin{align*}
x_{2k}^{q^2} - x_{2k}
&= \sum_{i=1}^{2k-1} (-1)^{i+1} x_i^q(x_{2k-i}^{q^2} - x_{2k-i}) \\
&= \sum_{\text{$i$ odd}} x_i^q \Big[\sum_{j=1}^{(2k-i-1)/2} \big((h_{2j}^q - g_{2j})x_{2k-i-2j} - g_{2j}(x_{2k-i-2j}^{q^2} - x_{2k-i-2j})\big)\Big] \\
&\qquad\qquad- \sum_{\text{$i$ even}} x_i^q\Big[\sum_{j=1}^{(2k-i)/2} \big((h_{2j} - g_{2j}) x_{2k-i-2j} - g_{2j}(x_{2k-i-2j}^{q^2} - x_{2k-i-2j})\big)\Big] \\
&= \sum_{\text{$i$ odd}} x_i^q \left[\sum_j (h_{2j}^q - g_{2j})x_{2k-i-2j}\right] - \sum_{\text{$i$ even}} x_i^q \left[\sum_j (h_{2j} - g_{2j})x_{2k-i-2j}\right] \\
&\qquad\qquad - \sum_{j=1}^{k-1} g_{2j}(x_{2k-2j}^{q^2} - x_{2k-2j}).
\end{align*}
On the other hand,
\begin{align*}
x_{2k}^{q^2} - x_{2k} 
&= \sum_{i=1}^k\big((h_{2i} - g_{2i})x_{2k-2i} - g_{2i}(x_{2k-2i}^{q^2} - x_{2k-2i})\big) \\
&= (h_{2k} - g_{2k}) + \sum_{i=1}^{k-1} (h_{2i} - g_{2i})x_{2k-2i} - \sum_{i=1}^{k-1} g_{2i}(x_{2k-2i}^{q^2} - x_{2k-2i}).
\end{align*}
Therefore
\begin{align*}
h_{2k} - g_{2k} 
&= \sum_{\substack{1 \leq i \leq 2k-1 \\ \text{$i$ odd}}} x_i^q \left[\sum_{j=1}^{(2k-i-1)/2} (h_{2j}^q - g_{2j}) x_{2k-i-2j}\right] \\
&\qquad\qquad - \sum_{\substack{1 \leq i \leq 2k-1 \\ \text{$i$ even}}} x_i^q \left[\sum_{j=1}^{(2k-i)/2} (h_{2j} - g_{2j}) x_{2k-i-2j}\right] - \sum_{i=1}^{k-1}(h_{2i} - g_{2i})x_{2k-2i}.
\end{align*}
This shows that the above collection of equations imposes the same conditions as the equations of type \eqref{*}.
\end{proof}

\begin{lemma}\label{l:nonvanish}
Let $k$ be the smallest $k$ such that $h_{2k} \neq g_{2k}$ and assume that $k \leq h-2$. If $(x_1, \ldots, x_{2(h-1)})$ is a solution to the equations \eqref{*} through \eqref{***}, then:
\begin{enumerate}[label=(\alph*)]
\item
If $g_{2i} \in \FF_q$ for $1 \leq i \leq k - 3$, then $g_{2(k-2)} \in \FF_q$ and $x_1(g_{2k-2}^q - g_{2k-2}) \neq 0$.

\item
If $g_{2i} \in \FF_q$ for $1 \leq i < k-3$ and $g_{2(k-3)} \notin \FF_q$, then $x_3(g_{2(k-3)}^q - g_{2(k-3)}) \neq 0$.

\item
If $n$ is odd and $g_{2i} \in \FF_q$ for $1 \leq i < k-n$ and $g_{2(k-n)} \notin \FF_q$, then $x_n(g_{2(k-n)}^q - g_{2(k-n)}) \neq 0$.

\item
If $n > 2$ is even and $g_{2i} \in \FF_q$ for $1 \leq i < k-n$ and $g_{2(k-n)} \notin \FF_q$, then $\# = 0$.
\end{enumerate}
\end{lemma}

\begin{proof}[Proof of (a)]
If $g_{2i} \in \FF_q$ for $1 \leq i \leq k-3$, then by Lemma \ref{l:testcond}, the tuple $(x_1, \ldots, x_{2(h-1)})$ must satisfy
\begin{align*}
0 = h_{2(k-1)} - g_{2(k-1)} &= x_1^q\big((g_{2(k-2)}^q - g_{2(k-2)})x_1\big) \\
0 \neq h_{2k} - g_{2k} &= x_1^q\big((g_{2(k-1)}^q - g_{2(k-1)})x_1 + (g_{2(k-2)}^q - g_{2(k-2)})x_3\big) \\
&\qquad\qquad + x_3^q\big((g_{2(k-2)}^q - g_{2(k-2)})x_1\big). \end{align*}
If $x_1 = 0$, then this automatically implies that $h_{2k} - g_{2k} = 0$, which contradicts the assumption that $h_{2k} \neq g_{2k}$. Therefore $x_1 \neq 0$. The first equation above then forces $g_{2(k-2)} \in \FF_q$, and so the second equation simplifies to
\begin{equation*}
0 \neq h_{2k} - g_{2k} = x_1^q(g_{2(k-1)}^q - g_{2(k-1)})x_1. \qedhere
\end{equation*}
\end{proof}

\begin{proof}[Proof of (b)]
If $g_{2i} \in \FF_q$ for $1 \leq i < k-3$ and $g_{2(k-3)} \notin \FF_q$, then by Lemma \ref{l:testcond}, we necessarily have
\begin{align*}
0 = h_{2(k-2)} - g_{2(k-2)} &= x_1^q\big((g_{2(k-3)}^q - g_{2(k-3)})x_1\big) \\
0 = h_{2(k-1)} - g_{2(k-1)} &= x_1^q\big((g_{2(k-2)}^q - g_{2(k-2)})x_1 + (g_{2(k-3)}^q - g_{2(k-3)})x_3\big) \\
&\qquad\qquad+ x_3^q\big((g_{2(k-3)}^q - g_{2(k-3)})x_1\big) \\
0 \neq h_{2k} - g_{2k} &= x_1^q\big((g_{2(k-1)}^q - g_{2(k-1)})x_1 + (g_{2(k-2)}^q - g_{2(k-2)})x_3 + (g_{2(k-3)}^q - g_{2(k-3)})x_5\big) \\
&\qquad\qquad + x_3^q\big((g_{2(k-2)}^q - g_{2(k-2)})x_1 + (g_{2(k-3)}^q - g_{2(k-3)})x_3\big) \\
&\qquad\qquad + x_5^q\big((g_{2(k-3)}^q - g_{2(k-3)})x_1\big).
\end{align*}
By assumption $g_{2(k-3)} \notin \FF_q$, so the first equation forces $x_1 = 0$. Then the second equation simplifies to $0 = 0$ and the third equation simplifies to
\begin{equation*}
0 \neq h_{2k} - g_{2k} = x_3^q(g_{2(k-2)}^q - g_{2(k-2)})x_3. \qedhere
\end{equation*}
\end{proof}

\begin{proof}[Proof of (c) and (d)]
Now suppose that there is some $n > 3$ such that $g_{2i} \in \FF_q$ for $1 \leq i < k-n$ and $g_{2(k-n)} \notin \FF_q$. Then since $h_{2i} = g_{2i}$ for $1 \leq i < k$, the equations \eqref{dagger2} simplify to
\begin{align*}
h_{2m} - g_{2m} &= \sum_{\substack{1 \leq i \leq 2m-1 \\ \text{$i$ odd}}} x_i^q\left[\sum_{j=k-n}^{(2m-i-1)/2} (g_{2j}^q - g_{2j}) x_{2m-i-2j}\right] && \text{for $k - n + 1 \leq m \leq k$.} \\
\end{align*}
So Equation \eqref{dagger2} for $2(k-n+1)$ is
\begin{equation*}
0 = x_1^q(g_{2(k-n)}^q - g_{2(k-n)})x_1,
\end{equation*}
which forces $x_1 = 0$ since by assumption $g_{2(k-n)} \notin \FF_q$. This implies that Equation \eqref{dagger2} for $2(k-n+2)$ gives the empty condition $0 = 0$. Setting $x_1 = 0$, Equation \eqref{dagger2} for $2(k-n+3)$ simplifies to
\begin{equation*}
0 = x_3^q(g_{2(k-n)}^q - g_{2(k-n)})x_3,
\end{equation*}
which forces $x_3 = 0$. Continuing this, we see that:
\begin{enumerate}[label=\textbullet]
\item
For $2l+1 \leq n$, Equation \eqref{dagger2} for $2(k - n + 2l + 1)$ yields
\begin{equation*}
h_{2(k-n+2l+1)} - g_{2(k-n+2l+1)} = x_{2l+1}^q(g_{2(k-n)}^q - g_{2(k-n)})x_{2l+1},
\end{equation*}
which forces $x_{2l + 1} = 0$ if $2l + 1 < n$, and $x_n^q(g_{2(k-n)}^q - g_{2(k-n)})x_n \neq 0$ when $2l + 1 = n$. This proves (c).

\item
For $2l \leq n$, the test equation for $m = k - n + 2l$ gives the condition equation $h_{2(k-n+2l)} - g_{2(k-n+2l)} = 0$. In particular, if $2l = n$, then we have $h_{2k} - g_{2k} = 0$, which is a contradiction. This proves (d). \hfill \qedhere
\end{enumerate}
\end{proof}

\begin{lemma}\label{l:polyform}
Let $k$ be as in Lemma \ref{l:nonvanish}. Then
\begin{enumerate}[label=(\alph*)]
\item
If $g_{2i} \in \FF_q$ for $1 \leq i \leq k - 3$, then Equation \eqref{dagger2} for $2(h-1)$ is of the form $ax_{2(h-1)-2(k-1)-1}^q - a^qx_{2(h-1)-2(k-1)-1} + a_0 = 0$, where $a = x_1(g_{2k-2}^q - g_{2k-2})$. Moreover, $x_{2(h-1)-2(k-1)-1}$ has no contribution to $a$ or $a_0$.

\item
If $n \geq 3$ is odd and $g_{2i} \in \FF_q$ for $1 \leq i < k-n$ and $g_{2(k-n)} \notin \FF_q$, then Equation \eqref{dagger2} for $2(h-1)$ is of the form $ax_{2(h-1)-2(k-n)-n}^q - a^qx_{2(h-1)-2(k-n)-n} + a_0 = 0$, where $a = x_n(g_{2(k-n)}^q - g_{2(k-n)}).$ Moreover, $x_{2(h-1)-2(k-n)-n}$ has no contribution to $a$ or $a_0$.
\end{enumerate}
\end{lemma}

\begin{proof}
First note that since $k \leq h-2$, then necessarily $2(h-1)-2(k-n)-n \neq n$, which automatically implies that $x_{2(h-1)-2(k-n)-n}$ has no contribution to $a$.

Recall Equation \eqref{dagger2} for $2(h-1)$:
\begin{align}\nonumber
h_{2(h-1)} - g_{2(h-1)} 
&= \sum_{\substack{1 \leq i \leq 2(h-1)-1 \\ \text{$i$ odd}}} x_i^q \left[\sum_{j=1}^{(2(h-1)-i-1)/2} (h_{2j}^q - g_{2j}) x_{2(h-1)-i-2j}\right] \\ \nonumber
&\qquad\qquad - \sum_{\substack{1 \leq i \leq 2(h-1)-1 \\ \text{$i$ even}}} x_i^q \left[\sum_{j=1}^{(2(h-1)-i)/2} (h_{2j} - g_{2j}) x_{2(h-1)-i-2j}\right] \\
&\qquad\qquad - \sum_{i=1}^{(h-1)-1}(h_{2i} - g_{2i})x_{2(h-1)-2i}. \label{e:polyform}
\end{align}

We prove (a). We need only show that the only terms in Equation \eqref{e:polyform} involving $x_{2(h-1)-2(k-1)-1}$ are exactly the terms
\begin{equation*}
ax_{2(h-1)-2(k-1)-1}^q - a^qx_{2(h-1)-2(k-1)-1}, \qquad \text{where $a = x_1(g_{2k-2}^q - g_{2k-2})$.}
\end{equation*}
Clearly any term involving $x_{2(h-1)-2(k-1)-1}$ must come from the first sum in the equation. These terms are
\begin{equation*}
\sum x_i^q(h_{2j}^q - g_{2j})x_{2(h-1)-2(k-1)-1} +  x_{2(h-1)-2(k-1)-1}^q(h_{2j}^q-g_{2j})x_i,
\end{equation*}
where the sum ranges over $i$ and $j$ such that $i + 2j + 2(h-1)-2(k-1)-1 = 2(h-1)$. In particular, if $i \geq 3$, then $2j \leq 2(k-2)$. We know by assumption that $h_{2j} = g_{2j} \in \FF_q$ for $2j \leq 2(k-3)$ and Lemma \ref{l:nonvanish}(a) implies $g_{2(k-2)} \in \FF_q$, so the coefficient $h_{2j}^q - g_{2j}$ vanishes for $2j \leq 2(k-2)$. Therefore the sum above simplifies to 
\begin{equation*}
x_1^q(g_{2(k-1)}^q - g_{2(k-1)})x_{2(h-1)-2(k-1)-1} + x_{2(h-1)-2(k-1)-1}^q(g_{2(k-1)}^q - g_{2(k-1)})x_1. 
\end{equation*}
Set $a = x_1(g_{2(k-1)}^q - g_{2(k-1)})$ and notice that since $x_1 \in \F$, the above expression simplifies to
\begin{equation*}
-a^qx_{2(h-1)-2(k-1)-1} + ax_{2(h-1)-2(k-1)-1}^q,
\end{equation*}
which is exactly what we wanted to show in (a).

We now prove (b). We need to establish the following statements:
\begin{enumerate}[label=(\roman*)]
\item
$x_n \in \F$

\item
The only term in the equation for $2(h-1)$ in Lemma \ref{l:testcond} that contains $x_{2(h-1)-2(k-n)-n}$ are the terms $ax_{2(h-1)-2(k-n)-n}^q$ and $a^q x_{2(h-1)-2(k-n)-n}$.
\end{enumerate}

In the proof of Lemma \ref{l:nonvanish}(c), we showed that for odd $m$ with $m < n$, we have $x_m = 0$. Then by the equation for $n$ of type \eqref{***}, we see that we must have
\begin{equation*}
x_n^{q^2} - x_n = 0,
\end{equation*}
so this shows (i).

To see (ii), we proceed as in the proof of part (a) of this lemma. Clearly any term involving $x_{2(h-1)-2(k-n)-n}$ must come from the first sum in Equation \eqref{e:polyform}. In this sum, the terms involving $x_{2(h-1)-2(k-n)-n}$ are
\begin{equation*}
\sum x_i^q(h_{2j}^q - g_{2j})x_{2(h-1)-2(k-n)-n} + x_{2(h-1)-2(k-n)-n}^q(h_{2j}^q - g_{2j})x_i,
\end{equation*}
where the sum ranges over $i$ and $j$ such that $i + 2j + 2(h-1)-2(k-n)-n = 2(h-1)$. Equivalently, $i + 2j = 2(k-n)+n.$ Note that this forces $i$ to be odd since $n$ is odd by assumption. If $i  < n$, then $x_i = 0$, and thus any terms involving $(h_{2j}^q - g_{2j})$ for $j > 2(k-n)$ vanish. If $j < 2(k-n)$, then by assumption $h_{2j} = g_{2j} \in \FF_q$, and so $h_{2j}^q - g_{2j} = 0$ when $j < 2(k-n)$. Therefore the above sum simplifies to
\begin{equation*}
x_n^q(g_{2(k-n)}^q - g_{2(k-n)})x_{2(h-1)-2(k-n)-n} + x_{2(h-1)-2(k-n)-n}^q(g_{2(k-n)}^q - g_{2(k-1)})x_n.
\end{equation*}
Set $a = x_n(g_{2(k-n)}^q - g_{2(k-n)})$. By (i), we know that $x_n \in \F$, and thus the above expression simplifies to
\begin{equation*}
-a^q x_{2(h-1)-2(k-n)-n} + a x_{2(h-1)-2(k-n)-n}^q,
\end{equation*}
which is exactly what we wanted to show in (b). This completes the proof.
\end{proof}

\begin{definition}\label{d:element}
Let $\delta \in \ker \Tr_{\F/\FF_q}$. Given a tuple $(x_1, \ldots, x_{2(h-1)}) \in \overline \FF_q^{2(h-1)}$ together with $g, \gamma \in H(\F)$ satisfying the conditions of Line \eqref{e:case2}, define a tuple $(x_1', \ldots, x_{2(h-1)}')$ in the following way:
\begin{enumerate}[label=\textbullet]
\item
Pick $z$ so that $z^{q^2} - z = \delta$

\item
Pick $y$ such that $ay^q - a^q y + \delta = 0,$ where $a$ is as in Lemma \ref{l:polyform}.

\item
Set $y_{2(h-1)-2(k-n)-n} \colonequals y$ and $y_i = 0$ for odd such that $i < 2(h-1)-n$ and $i \neq 2(h-1)-2(k-n)-n$. Here, $k$ is as in Lemma \ref{l:polyform}.

\item
For each odd $i$ with $i > 2(h-1)-n$, pick $y_i$ so that
\begin{equation*}
y_i^{q^2} - y_i = \sum_{2m \leq i} (h_{2m}^q - g_{2m}) y_{i-2m}
\end{equation*}
\end{enumerate}
Finally, define
\begin{equation*}
x_i' = \begin{cases}
x_i + y_i & \text{if $i$ is odd,} \\
x_i + z & \text{if $i = 2(h-1)$,} \\
x_i & \text{otherwise.}
\end{cases}
\end{equation*}
\end{definition}

\begin{lemma}\label{l:element}
Let $k$ be the smallest integer such that $h_{2k} \neq g_{2k}$ and assume that $k \leq h-2$. Let $(x_1, \ldots, x_{2(h-1)}) \in A_{g,\gamma}$ and $\delta \in \ker \Tr_{\F/\FF_q}$. Then the tuple $(x_1', \ldots, x_{2(h-1)}')$ defined in Definition \ref{d:element} is an element of $A_{g,\gamma + \delta \tau^{2(h-1)}}$.
\end{lemma}

\begin{proof}
It is easy to see that $x'$ satisfies each Type \eqref{**} equation. To see that $x'$ satisfies the Type \eqref{***} equations for the all odd $i$, amounts to checking that
\begin{equation*}
(x_i')^{q^2} - (x_{i}') = x_{i}^{q^2} - x_{i} + \sum_{2m \leq i} (h_{2m}^q - g_{2m}) y_{i-2m},
\end{equation*}
which certainly holds by the construction of $y_i$. 
By Lemma \ref{l:testcond}, it remains only to show that $x'$ satisfies each Type \eqref{dagger2} equation. This amounts to showing that for each $i$, 
\begin{equation*}
\sum_{\substack{1 \leq j \leq 2i-1 \\ \text{$j$ odd}}} x_j^q \left[\sum_{m=1}^{(2i-j-1)/2} (h_{2m}^q - g_{2m}) x_{2i-j-2m} \right] = \sum_{\substack{1 \leq j \leq 2i-1 \\ \text{$j$ odd}}} (x_j')^q \left[\sum_{m=1}^{(2i-j-1)/2} (h_{2m}^q - g_{2m}) x_{2i-j-2m}' \right]
\end{equation*}
Note that by construction, if $j$ is odd and $j < n$, then $x_j' = x_j = 0$ (using the proof of Lemma \ref{l:nonvanish}(c) here). Furthermore, since we have $h_{2m}^q - g_{2m} = 0$ if $m < k-n$, then the only potentially nonzero terms on the right-hand side are of the form $(x_j')^q (h_{2m}^q - g_{2m}) x_{2i-j-2m}'$ where $j \geq n$, $m \geq k-n$, and $2(h-1)-j-2m \geq n$. First assume $i < h-1$. Then it follows that $j, 2i-j-2m < 2(h-1)-2(k-n)-n$, and thus $x_j' = x_j$ and $x_{2i-j-2m}' = x_{2i-j-2m}$. Therefore equality holds when $i < h-1$.

Finally, let $i = h-1$. Then by the above analysis together with Lemma \ref{l:polyform}, showing $x'$ satisfies the Type \eqref{dagger2} equation for $2(h-1)$ is equivalent to showing the equality
\begin{equation*}
a(x_{2(h-1)-2(k-n)-n}')^q - a^q(x_{2(h-1)-2(k-n)-n)}') + a_0 + \delta = 0.
\end{equation*}
But since $(x_1, \ldots, x_{2(h-1)})$ satisfies Equation \eqref{dagger2} for $2(h-1)$, and since $x_{2(h-1)-2(k-n)-n}' = x_{2(h-1)-2(k-n)-n} + y$ where $ay^q - a^q y + \delta = 0$, then the above equality holds. This finishes the proof that $(x_1', \ldots, x_{2(h-1)}') \in A_{g, \gamma + \delta \tau^{2(h-1)}}$.
\end{proof}

\begin{lemma}\label{l:solbij}
There is a bijection between $A_{g,\gamma}$ and $A_{g,\gamma+\delta\tau^{2(h-1)}}$, where $\delta \in \ker \Tr_{\F/\FF_q}$.
\end{lemma}

\begin{proof}
Pick $\delta \in \ker \Tr_{\F/\FF_q}$. First notice is that if $g$ and $h$ satisfy the hypotheses of Lemma \ref{l:nonvanish}(d), then $g$ and $\gamma + \delta \tau^{2(h-1)}$ also satisfy the hypotheses of Lemma \ref{l:nonvanish}(d). Thus, by Lemma \ref{l:nonvanish}(d), $A_{g,\gamma}$ and $A_{g,\gamma+\delta}$ are both empty.

From now on, we assume that $g$ and $h$ either satisfy the hypotheses in Lemma \ref{l:polyform}(a) or \ref{l:polyform}(b).  By Lemma \ref{l:polyform}, the Type \eqref{dagger2} equation for $2(h-1)$ is of the form $ax_{2(h-1)-2(k-n)-n}^q - a^qx_{2(h-1)-2(k-n)-n} + a_0 = 0$ where $a = x_n(g_{2(k-n)}^q - g_{2(k-n)})$. 

Let $(x_1, \ldots, x_{2(h-1)}) \in A_{g,\gamma}$ and let $(x_1', \ldots, x_{2(h-1)}')$ be the element of $A_{g,\gamma + \delta \tau^{2(h-1)}}$ constructed in Definition \ref{d:element}. Then we have a map
\begin{equation*}
\varphi_\delta \from A_{g,\gamma} \to A_{g,\gamma+\delta\tau^{2(h-1)}}, \qquad x \mapsto x'.
\end{equation*}

Now we check that $\varphi_\delta$ is invertible. Using the same notation as in Definition \ref{d:element}, it is easy to see that by Lemma \ref{l:element}, setting
\begin{equation*}
x_i'' = \begin{cases}
x_i - y_i & \text{if $i$ is odd} \\
x_i - z & \text{if $i = 2(h-1)$} \\
x_i & \text{otherwise}
\end{cases}
\end{equation*}
defines a map
\begin{equation*}
\varphi_{-\delta} \from A_{g,\gamma+\delta\tau^{2(h-1)}} \to A_{g,\gamma}, \qquad x \mapsto x''
\end{equation*}
wherein
\begin{align*}
\varphi_{-\delta} \circ \varphi_{\delta} &= \id_{A_{g,\gamma}} \\
\varphi_{\delta} \circ \varphi_{-\delta} &= \id_{A_{g,\gamma+\delta\tau^{2(h-1)}}}.
\end{align*}
Therefore $\varphi_\delta$ must be a bijection.
\end{proof}

We are now ready to prove Proposition \ref{p:case2}.

\begin{proof}[Proof of Proposition \ref{p:case2}]
From Lemma \ref{l:solbij}, $|A_{g,\gamma}| = |A_{g,\gamma+\delta\tau^{2(h-1)}}|$, so
\begin{equation*}
\sum \chi_2(g^{-1}\gamma) \cdot \# = {\sum}' \chi_2(g^{-1}\gamma) \sum_{\delta \in \ker \Tr} \psi(\delta) \cdot N_{g,\gamma+\delta\tau^{2(h-1)}} = 0,
\end{equation*}
where $\sum$ and $\sum'$ range over $g \in H$ and $\gamma \in H$ satisfying the condition that there exists $k \leq h-2$ such that $h_{2k} \neq g_{2k}$, and $\sum'$ has the additional restriction that $h_{2(h-1)} \in \FF_q$. This completes the proof of Proposition \ref{p:case2}.
\end{proof}

\begin{proof}[Proof of Theorem \ref{t:interdim}]
This follows directly from Proposition \ref{p:case1} and \ref{p:case2}.
\end{proof}

\subsection{Proof of Theorem \ref{t:irrep}}\label{s:pfirrep}

Let $\psi$ and $\chi \in \cA_\psi$ be as in the statement of the theorem. Let $\theta$ be an arbitrary character of $G_{h-2} \subset H$. To prove Theorem \ref{t:irrep}, we will compute
\begin{equation*}
\dim H_c^{h-1}(X_h, \overline \QQ_\ell)_{\chi, \theta} = \frac{1}{q^{h-1} \cdot q^{2(h-1)} \cdot q^h} \sum_{\substack{g \in H \\ \gamma \in G_2}} \chi(g)^{-1} \theta(\gamma) \cdot N(g, \gamma),
\end{equation*}
where the above equation follows by Lemma 2.13 of \cite{B12} and $N(g,\gamma) = \#\{x \in X_h(\overline \FF_q) : g * \Fr_{q^2}(x) = x \cdot \gamma\} = \# A_{g,\gamma}$, as in Section \ref{s:pfinterdim}. Since $G_{h-2}$ is a subgroup of $H$, then in fact $x \in A_{g,\gamma}$ if and only if $x = (x_1, \ldots, x_{2(h-1)})$ satisfies Equations \eqref{*} through \eqref{***}, where as before, we write $g = 1 + \sum g_i \tau^i$ and $\gamma = 1 + \sum h_i \tau^i$.

Now comes the simplification. It is not difficult to see inductively that since $\gamma \in G_{h-2}$, Equations \eqref{*} through \eqref{***} are equivalent to the following:
\begin{enumerate}[label=(\roman*)]
\item
For $1 \leq n \leq 2(h-1)$, we have $x_n^{q^2} - x_n = 0$.

\item
For $1 \leq k \leq h-1$, we have $h_{2k} = g_{2k}$.
\end{enumerate}
Thus,
\begin{equation*}
N(g,\gamma) = 
\begin{cases}
q^{4(h-1)} & \text{if $g = \gamma$,} \\
0 & \text{otherwise.}
\end{cases}
\end{equation*}
Therefore,
\begin{align*}
\dim H_c^{h-1}(X_h, \overline \QQ_\ell)_{\chi, \theta} 
&= \frac{1}{q^{h-1} \cdot q^{2(h-1)} \cdot q^h} \sum_{\substack{g \in H \\ \gamma \in G_2}} \chi(g)^{-1} \theta(\gamma) \cdot N(g, \gamma) \\
&= \frac{1}{q^{h-1} \cdot q^{2(h-1)} \cdot q^h} \sum_{\gamma \in G_2} \chi(\gamma)^{-1} \theta(\gamma) \cdot q^{4(h-1)} \\
&= \frac{q^{4(h-1)} \cdot |G_2|}{q^{h-1} \cdot q^{2(h-1)} \cdot q^h} \cdot \langle \chi, \theta \rangle_{G_{h-2}} = q^{h-1} \langle \chi, \theta \rangle_{G_{h-2}}.
\end{align*}
It follows that $H_c^{h-1}(X_h, \overline \QQ_\ell)[\chi]$ has dimension $q^{h-1}$. Now, $H_c^{h-1}(X_h, \overline \QQ_\ell)[\chi]$ is a representation of $\UnipF$ wherein $H_{2(h-1)}(\F)$ acts by $\psi$. Therefore, by Theorem \ref{t:bij}, $H_c^{h-1}(X_h, \overline \QQ_\ell)[\chi]$ is irreducible. This completes the proof.

\section{An Example: Level 3}\label{s:examples} 

In \cite{B12} (see Theorem 5.20), Boyarchenko computes the representations $H_c^\bullet(X_3)[\chi]$ for characters $\chi$ whose restriction to $U_L^2/U_L^3$ has conductor $q^2$. The computational method presented in this paper generalizes the result Theorem 5.20 in \cite{B12} but differs from its proof. Specifically, for characters $\chi_1, \chi_2 \from U_L^1/U_L^3 \to \overline \QQ_\ell^\times$, Boyarchenko computes the subspace $H_c^\bullet(X_3)_{\chi_1, \chi_2^\sharp} \subset H_c^\bullet(X_3)$ wherein $U_L^1/U_L^3$ acts by $\chi_1$ and $H'(\F)$ acts by $\chi_2^\sharp$. 

In this paper (see Section \ref{s:cohomreps}), we compute the subspace $H_c^\bullet(X_h)_{\chi_1, \chi_2} \subset H_c^\bullet(X_h)$ wherein $U_L^1/U_L^h$ acts by $\chi_1$ and $H(\F) \subset \UnipF$ acts by $\chi_2$. Here, the action of $U_L^1/U_L^h$ is the one induced by the left action on $X_h$, and the action of any subgroup of $\UnipF$ is the one induced by the right-multiplication action on $X_h$. We then use the character formula established in Section \ref{s:morerepthy} to determine the representation $H_c^\bullet(X_h)[\chi]$.

In this section, we apply the arguments of this paper to the special case $h = 3$, thereby obtaining a different proof of Theorem 5.20 of \cite{B12}. These examples allow us to illustrate the structure and flavor of the general computations in a simpler setting. The boxed equations indicate the milestone steps.

\subsection{Restrictions of Irreducible Representations of $U_3^{2,q}(\F)$}

In this subsection, we describe the computations of Section \ref{s:morerepthy} in our special case $h = 3$. For a character $\chi \from U_L^1/U_L^3 \to \overline \QQ_\ell^\times$ whose restriction to $U_L^2/U_L^3$ has conductor $q^2$, let $\rho_\chi$ be the irreducible representation of $U_3^{2,q}(\F)$ associated to $\chi$ under the bijection described in Proposition \ref{p:bijodd}.

For convenience, we remind the reader of the notation established in Section \ref{s:morerepthy}. Let
\begin{align*}
H &= \{1 + a_2 \tau^2 + a_4 \tau^4 : a_i \in \F\}, \\
K &= \{1 + a_2 \tau^2 + a_3 \tau^3 + a_4 \tau^4 : a_i \in \F\}, \\
G_1 &= \{1 + a_2 \tau^2 + a_4 \tau^4 : a_2 \in \FF_q, \, a_4 \in \F\}, \\
\cA(\chi) &= \{\text{characters $\theta \from H \to \overline \QQ_\ell^\times$ s.t.\ $\chi = \theta$ on $G_1$ but not on $H$}\}.
\end{align*}

We would like to show that as elements of the Grothendieck group of $H$,
\begin{equation}\label{e:char3}
\boxed{\rho_\chi = (-1)(q \cdot \chi + (-1)(q + 1) \cdot \Ind_{G_1}^H(\chi)),}
\end{equation}
and therefore, as a representation of $H$, the representation $\rho_\chi$ comprises
\begin{equation}\label{e:decomp3}
\boxed{\begin{cases}
\text{$1$ copy of $\chi$, and} \\
\text{$q+1$ copies of $\theta$, for $\theta \in \cA(\chi)$.}
\end{cases}}
\end{equation}

First note that $G_1$ is the center of $U_3^{2,q}(\F)$, so if $s \in G_1$, then $\Tr \rho_\chi(s) = q^2 \cdot \chi(s)$.

Now suppose $s \in H \smallsetminus G_1$. (Note that if we write $1 = h - 1 - k$, we have $k = 1.$) By a straightforward computation, one can see that every element $t \in U_3^{2,q}(\F)$ can be written in the form $t = (1 - a_1 \tau)(1 - a_3 \tau^3) \cdot g$ for some $g \in H$. Furthermore, $(1 - a_3 \tau^3)s(1 - a_3 \tau^3)^{-1} = s$. Thus if $t \in K$, then $tst^{-1} = s$.

Now take $a \in \F^\times$. Then
\begin{equation*}
(1 - a \tau)(1 + s_2 \tau^2 + s_4 \tau^4)(1 - a \tau)^{-1} = (1 + s_2 \tau^2 + (-a(s_2^q - s_2)) \tau^3 + s_4 \tau^4)(1 + (-a^{q+1}(s_2^q - s_2)) \tau^4),
\end{equation*}
and therefore, remembering that $\chi^\sharp(1 + a_2 \tau^2 + a_3 \tau^3 + a_4 \tau^4) = \chi(1 + a_2 \tau^2 + a_4 \tau^4)$ by definition, we have
\begin{equation*}
\chi^\sharp((1 - a\tau)s(1 - a\tau)^{-1}) = \chi(s) \cdot \psi(-a^{q+1}(s_2^q - s_2)).
\end{equation*}
Since $\psi$ has conductor $q^2$, its restriction to the subgroup $\ker \Tr_{\F/\FF_q} \subset \F$ is nontrivial. Note that for any $a \in \F^\times$, we have $a^{q+1}(s_2^q - s_2) \in \ker \Tr_{\F/\FF_q} \smallsetminus \{0\}$ since $s_2 \notin \FF_q$ by assumption. Therefore if $s \in H \smallsetminus G_1$, we have
\begin{align*}
\rho_\chi(s) 
&= \frac{1}{|K|} \sum_{t \in U_3^{2,q}(\F)} \chi_\circ^\sharp(tst^{-1}) \\
&= \frac{1}{|K|} \Big( \sum_{t \in K} \chi_\circ^\sharp(tst^{-1}) + \sum_{t \notin K} \chi_\circ^\sharp(tst^{-1}) \Big) \\
&= \chi(s) + \sum_{a \in \F^\times} \chi_\circ^\sharp((1 - a\tau)s(1 - a\tau)^{-1}) \\
&= \chi(s) + (-1)(q+1) \cdot \chi(s) = -q \cdot \chi(s).
\end{align*}

Consider the $H$-representation
\begin{equation*}
\rho = (-1)(q \cdot \chi + (-1)(q + 1) \cdot \Ind_{G_1}^H(\chi)).
\end{equation*}
Then since $H$ is abelian,
\begin{align*}
\Tr \rho(s) 
&= (-1)(q \cdot \chi(s) + (-1)(q + 1) \frac{|H|}{|G_1|} \cdot \ONE_{G_1}(s) \cdot \chi(s)) \\
&= \begin{cases}
q^2 \cdot \chi(s) & \text{if $s \in G_1$,} \\
-q \cdot \chi(s) & \text{if $s \in H \smallsetminus G_1.$}
\end{cases}
\end{align*}
Thus we can conclude that as (virtual) representations of $H$, $\rho_\chi = \rho$, and Equation \eqref{e:char3} follows.

Let $\theta \from H \to \overline \QQ_\ell^\times$ be any character. If it agrees with $\chi$ on $G_1$, then it occurs exactly once in $\Ind_{G_1}^H(\chi)$. Moreover, if $\theta$ is a constituent of $\rho_\chi$, then it must agree with $\chi$ on $G_1$. Thus by Equation \eqref{e:char3}, we see that if $\theta = \chi$, then $\theta$ occurs in $\rho_\chi$ exactly once ($-q + (q+1) = 1$), and if $\theta = \chi$ on $G_1$ but not on $H$, then $\theta$ occurs in $\rho_\chi$ exactly $q+1$ times. This proves Equation \eqref{e:decomp3}.

\subsection{Morphisms Between $H_c^i(X_3)$ and Representations of $U_3^{2,q}(\F)$}

Let $\psi \from \F \to \overline \QQ_\ell^\times$ have conductor $q^2$. Recall from Section \ref{s:repthy} that every irreducible representation of $U_3^{2,q}(\F)$ that restricts to a sum of $\psi$ occurs in $V_\psi = \Ind_{H_0'(\F)}^{U_3^{2,q}(\F)}(\widetilde \psi)$.

Let $H_c^\bullet(X_3) = \bigoplus_{i \in \ZZ} H_c^i(X_3, \overline \QQ_\ell).$ The action of $U_3^{2,q}(\F)$ on $X_h$ induces a $U_3^{2,q}(\F)$-module structure on $H_c^\bullet(X_h)$. We wish to compute the space of morphisms from $V_\psi$ to $H_c^\bullet(X_3)$. We can show that $\Fr_{q^2}$ acts on $\Hom_{U_3^{2,q}(\F)}(V_\psi, H_c^i(X_3))$ via multiplication by $q^2$ and that
\begin{equation}\label{e:dimhom3}
\boxed{\dim \Hom_{U_3^{2,q}(\F)}(V_\psi, H_c^i(X_3)) = \begin{cases}
q^4 & \text{if $i = 2$,} \\
0 & \text{otherwise.}
\end{cases}}
\end{equation}

If we specialize the proof of Theorem \ref{t:homspace} to the case $h = 3$, we recover the proof of Lemma 6.18 of \cite{B12}. We omit this part of the example and refer the reader to \cite{B12} for this computation.

\subsection{Intertwining Spaces of $H_c^\bullet(X_3)$}

For characters $\chi_1, \chi_2 \from U_L^1/U_L^3 \to \overline \QQ_\ell^\times$, consider the subspace $H_c^i(X_3, \overline \QQ_\ell)_{\chi_1, \chi_2} \subset H_c^i(X_3, \overline \QQ_\ell)$ wherein $U_L^1/U_L^3 \times H \subset U_L^1/U_L^3 \times U_3^{2,q}(\F)$ acts by $\chi_1 \otimes \chi_2$. (Recall that the left action of $U_L^1/U_L^3$ and the right action of $U_3^{2,q}(\F)$ on $X_3$ described in Section \ref{s:notation} induce a $(U_L^1/U_L^3 \times U_3^{2,q}(\F))$-module structure on the cohomology of $X_3$. Recall also that $U_L^1/U_L^3 \cong H$.) Assume that the restriction of $\chi_1$ to $U_L^2/U_L^3$ is $\psi$ and that $\psi$ has conductor $q^2$. Then by Equation \eqref{e:dimhom3}, we know that $H_c^i(X_3, \overline \QQ_\ell)_{\chi_1, \chi_2}$ vanishes for $i \neq 2$. We will show that
\begin{equation}\label{e:interdim3}
\boxed{\dim H_c^2(X_3, \overline \QQ_\ell)_{\chi_1, \chi_2} = (-1)\Big(\langle \chi_1, \chi_2 \rangle \cdot q + (-1)(q+1) \cdot \chi_1, \chi_2 \rangle_{G_1}\Big)}
\end{equation}

Equation \eqref{e:dimhom3} implies that we can apply Lemma 2.13 of \cite{B12}, which implies
\begin{equation*}
\dim H_c^2(X_3, \overline \QQ_\ell)_{\chi_1, \chi_2} = \frac{1}{q^2 \cdot q^4 \cdot q^4} \sum_{g,h \in H} \chi_1(g)^{-1} \chi_2(h) \cdot \#\{x \in X_3(\overline \FF_q) : g * \Fr_{q^2}(x) = x \cdot h\}.
\end{equation*}
Tracing through the definitions in Section \ref{s:notation}, we have
\begin{align*}
g * \Fr_{q^2}(x) 
&= (1 + g_2 \tau^2 + g_4 \tau^4) * (1 + x_1^{q^2} \tau + \cdots + x_4^{q^2} \tau^4) \\
&= 1 + x_1^{q^2} \tau + (x_2^{q^2} + g_2) \tau^2 + (x_3^{q^2} + g_2 x_1^{q^2}) \tau^3 + (x_4^{q^2} + g_2 x_2^{q^2} + g_4) \tau^4, \\
x \cdot h
&= (1 + x_1 \tau + \cdots + x_4 \tau^4) \cdot (1 + h_2 \tau^2 + h_4 \tau^4) \\
&= 1 + x_1 \tau + (x_2 + h_2) \tau^2 + (x_3 + x_1 h_2^q) \tau^3 + (x_4 + x_2 h_2 + h_4) \tau^4.
\end{align*}
Equating coefficients of $\tau$ and combining these equations with the defining equations of $X_3$ (see Theorem \ref{t:Xhpolys}) implies that $x \in X_3(\overline \FF_q)$ if and only if $x$ satisfies
\begin{align}
\label{*1}
x_2^{q^2} - x_2 &= x_1^q(x_1^{q^2} - x_1) \\
\label{*2}
x_4^{q^2} - x_4 &= x_1^q(x_3^{q^2} - x_3) - x_2^q(x_2^{q^2} - x_2) + x_3^q(x_1^{q^2} - x_1) \\ 
\label{***1}
x_1^{q^2} - x_1 &= 0 \\
\label{**1}
x_2^{q^2} - x_2 &= h_2 - g_2 \\
\label{***2}
x_3^{q^2} - x_3 &= h_2^q x_1 - g_2 x_1^{q^2} \\ 
\label{**2}
x_4^{q^2} - x_4 &= h_4 - g_4 + h_2 x_2 - g_2 x_2^{q^2}
\end{align}
which reduce to the conditions $x_1, x_2 \in \F$, $h_2 = g_2$, and
\begin{align}
\label{*3}
x_4^{q^2} - x_4 &= x_1^q(x_3^{q^2} - x_3) \\ 
\label{***3}
x_3^{q^2} - x_3 &= (g_2^q - g_2) x_1 \\ 
\label{**3}
x_4^{q^2} - x_4 &= h_4 - g_4.
\end{align}
Note that Equations \eqref{*1} and \eqref{*2} are of Type \eqref{*}, Equations \eqref{**1} and \eqref{**2} are of Type \eqref{**}, and Equations \eqref{***1} and \eqref{***2} are of Type \eqref{***}.

First observe that if $g_2 \in \FF_q$, then Equation \eqref{***3} implies that $x_3 \in \F$, which forces $x_4 \in \F$ by Equation \eqref{*3}. Thus by Equation \eqref{**3}, we know that
\begin{equation*}
\#\{x \in X_3(\overline \FF_q) : g * \Fr_{q^2}(x) = x \cdot h\} = \begin{cases}
q^8 & \text{if $g_4 = h_4$,} \\
0 & \text{otherwise.}
\end{cases}
\end{equation*}
(This is Lemma \ref{l:easycond}.)

If $g_2 \notin \FF_q$, then combining Equations \eqref{*3}, \eqref{***3}, and \eqref{**3}, we see that 
\begin{align*}
\#\{x \in X_3(\overline \FF_q) : g * \Fr_{q^2}(x) = x \cdot h\} 
&= q^6 \cdot \#\{x_1 \in \overline \FF_q : h_4 - g_4 = x_1^q(g_2^q - g_2)x_1\} \\
&= \begin{cases}
q^6 & \text{if $g_4 = h_4$,} \\
q^6(q+1) & \text{if $g_4 \neq h_4$ and $g_4 - h_4 \in \ker \Tr_{\F/\FF_q}$,} \\
0 & \text{otherwise.}
\end{cases}
\end{align*}
(This is Lemma \ref{l:oddcond} for $k = 1$.) Thus if $g_2 \notin \FF_q$, then
\begin{equation*}
\sum_{\epsilon \in \F} \psi(\epsilon) \cdot \#\{x \in X_3(\overline \FF_q) : g * \Fr_{q^2}(x) = x \cdot h\} = q^6 - (q+1)q^6 = -q^7.
\end{equation*}
Putting this together, we have
\begin{align*}
\dim H_c^2(X_3, \overline \QQ_\ell)_{\chi_1, \chi_2}
&= \frac{1}{q^{10}} \sum_{\substack{g \in H \\ \epsilon \in \F}} \frac{\chi_2(g)}{\chi_1(g)} \cdot \psi(\epsilon) \cdot \#\{x \in X_3(\overline \FF_q) : g * \Fr_{q^2}(x) = x \cdot h\} \\
&= \frac{1}{q^{10}} \Bigg(\sum_{\substack{g \in G_1 \\ \epsilon \in \F}} \frac{\chi_2(g)}{\chi_1(g)} \cdot \psi(\epsilon) \cdot \#\{x \in X_3(\overline \FF_q) : g * \Fr_{q^2}(x) = x \cdot h\} \\
&\qquad\qquad + \sum_{\substack{g \in H \smallsetminus G_1 \\ \epsilon \in \F}} \frac{\chi_2(g)}{\chi_1(g)} \cdot \psi(\epsilon) \cdot \#\{x \in X_3(\overline \FF_q) : g * \Fr_{q^2}(x) = x \cdot h\}\Bigg) \\
&= \frac{1}{q^{10}}\Bigg(|G_1| \cdot \langle \chi_1, \chi_2 \rangle_{G_1} \cdot \psi(1) \cdot q^8 \\
&\qquad\qquad + |H| \cdot \langle \chi_1, \chi_2 \rangle_H \cdot -q^7 - |G_1| \langle \chi_1, \chi_2 \rangle_{G_1} \cdot -q^7\Bigg) \\
&= (q+1) \cdot \langle \chi_1, \chi_2 \rangle_{G_1} - q \cdot \langle \chi_1, \chi_2 \rangle_H.
\end{align*}
This completes the proof of Equation \eqref{e:interdim3}.

\begin{remark}
Note that for $h = 3$, the arguments in Section \ref{s:case1} are enough to compute the intertwining spaces $H_c^i(X_3)_{\chi_1, \chi_2}$. The arguments in Section \ref{s:case2} are needed to compute the intertwining spaces $H_c^i(X_h)_{\chi_1, \chi_2}$ for $h \geq 4$.
\end{remark}

\subsection{The Representations $H_c^\bullet(X_3)[\chi]$}

By Equation \eqref{e:interdim3}, the dimension of the $U_3^{2,q}(\F)$-representation $H_c^2(X_3, \overline \QQ_\ell)_{\chi_1, \chi_2}$ is equal to $q^2$, which implies by Section \ref{s:repthy} (see Lemma \ref{l:irrepodd} and Proposition \ref{p:bijodd}) that it is irreducible. Thus by Corollary \ref{c:detbyH}, it is uniquely determined by its restriction to $H(\F)$. Comparing Equation \eqref{e:interdim3} to Equation \eqref{e:decomp3} allows us to conclude that if $\chi \from U_L^1/U_L^3 \to \overline \QQ_\ell^\times$ restricts to a character of conductor $q^2$ on $U_L^2/U_L^3$, then 
\begin{equation*}
\boxed{H_c^i(X_3, \overline \QQ_\ell)[\chi] = 
\begin{cases}
\rho_\chi & \text{if $i = 2$,} \\
0 & \text{otherwise.}
\end{cases}}
\end{equation*}
This proves Theorem 5.20 of \cite{B12} and completes our example.

\section{Representations of Division Algebras}\label{s:divalg}

Throughout this section, $\theta \from L^\times \to \overline \QQ_\ell^\times$ will be a primitive character of level $h$. Recall that $\theta$ is primitive of level $h$ if for each $\gamma \in \Gal(L/K)$, both $\theta$ and $\theta/\theta^\gamma$ have level $h$. This induces a character $\chi \from U_L^1/U_L^h \to \overline \QQ_\ell^\times$ whose restriction to $U_L^{h-1}/U_L^h \cong \F$ has conductor $q^2$ and will be denoted by $\psi$.

In this section, we use Theorem \ref{t:cohomdesc} in order to describe the representations of the division algebra $D^\times \colonequals D_{1/2}^\times$ arising from Lusztig's conjectural $p$-adic Deligne-Lusztig variety $X$ (see \cite{L79} and \cite{B12}). We can write $D = L \langle \Pi \rangle/(\Pi^2 - \pi)$, where $L \langle \Pi \rangle$ is the twisted polynomial ring defined by the commutation relation $\Pi \cdot a = \varphi(a) \cdot \Pi$ ($\varphi$ is the nontrivial element of $\Gal(L/K)$), and $\pi$ is the uniformizer of $L$. Write $\cO_D = \OH_L \langle \Pi \rangle/(\Pi^2 - \pi)$ for the ring of integers of $D$. Define $P_D^r = \Pi^r \OH_D$ and $U_D^r = 1 + P_D^r$.

There exists a connected reductive group $\GG$ over $K$ such that $\GG(K)$ is isomorphic to $D^\times$, and a $K$-rational maximal torus $\TT \subset \GG$ such that $\TT(K)$ is isomorphic to $L^\times$. We describe $\GG$ more explicitly here. Let $\Khat$ be the completion of the maximal unramified extension of $K$ and let $\varphi$ denote the Frobenius automorphism of $\Khat$ (inducing $x \mapsto x^q$ on the residue field). Letting $\varpi = \left(\begin{smallmatrix} 0 & 1 \\ \pi & 0 \end{smallmatrix}\right)$, the homomorphism $F \from \GL_2(\Khat) \to \GL_2(\Khat)$ given by $F(A) = \varpi^{-1} A^\varphi \varpi$ is a Frobenius relative to a $K$-rational structure whose corresponding algebraic group over $K$ is $\GG$.

Let $\widetilde G \colonequals \GG(\Khat) = \GL_2(\Khat)$ and $\widetilde T \colonequals \TT(\Khat)$. Let $\BB \subset \GG \otimes_K \Khat$ be the Borel subgroup consisting of upper triangular matrices and let $\UU$ be its unipotent radical. Note that $\widetilde T$ consists of all diagonal matrices and $\widetilde U \colonequals \UU(\Khat)$ consists of unipotent upper triangular matrices. Let $\widetilde U^- \subset \GL_2(\Khat)$ denote the subgroup consisting of unipotent lower triangular matrices.

The $p$-adic Deligne-Lusztig construction $X$ for $D^\times$ described in \cite{L79} is the quotient
\begin{equation*}
X \colonequals (\widetilde U \cap F^{-1}(\widetilde U)) \backslash \{A \in \GL_2(\Khat) : F(A)A^{-1} \in \widetilde U\}.
\end{equation*}
In \cite{B12} (see Section 4.2 of \textit{op.\ cit.}), Boyarchenko proves that $X$ can be identified\footnote{Since we are in the situation $n=2$, the subgroup $\widetilde U \cap F^{-1}(\widetilde U)$ is actually trivial. For arbitrary $n$, the analogous subgroup is not trivial, but then there is more substance to the identification of $X$ with $\widetilde X$.} with the set
\begin{equation*}
\widetilde X \colonequals \{A \in \GL_2(\Khat) : F(A)A^{-1} \in \widetilde U \cap F(\widetilde U^-)\}
\end{equation*}
and describes how to define the homology groups $H_i(\widetilde X, \overline \QQ_\ell)$ (see Section 4.4 of \textit{op.\ cit.}). For each $i \geq 0$, $H_i(\widetilde X, \overline \QQ_\ell)$ inherits commuting smooth actions of $\GG(K) \cong D^\times$ and $\TT(K) \cong L^\times$. Given a smooth character $\theta \from L^\times \to \overline \QQ_\ell^\times$, we may consider the subspace $H_i(\widetilde X, \overline \QQ_\ell)[\theta] \subset H_i(\widetilde X, \overline \QQ_\ell)$ wherein $L^\times$ acts by $\theta$.

Using Proposition 5.19 of \textit{op.\ cit.}, we can now describe the cohomology groups $H_i(\widetilde X, \overline \QQ_\ell)[\theta]$ as representations of the division algebra $D^\times \colonequals D_{1/2}^\times$. For convenience, we restate the description given in this proposition.
\begin{enumerate}[label=\textbullet]
\item
Let $\rho_\chi$ denote the representation $H_c^{h-1}(X_h, \overline \QQ_\ell)[\chi]$. (Note that by Theorem \ref{t:cohomdesc}, this notation is consistent with the representation $\rho_\chi$ introduced in Section \ref{s:repthy}.) This is a representation of $\UnipF \cong U_D^1/U_D^{2(h-1)+1}$.

\item
This extends to a representation $\eta_\theta^\circ$ of $\OH_D^\times/U_D^{2(h-1)+1}$ with the property that $\Tr(\eta_\theta^\circ(\zeta)) = (-1)^{h-1} \theta(\zeta).$

\item
This inflates to a representation $\widetilde \eta_\theta^\circ$ of $\OH_D^\times$.

\item
This extends to a representation $\eta_\theta'$ of $\pi^\ZZ \cdot \OH_D^\times$ via setting $\eta_\theta'(\pi) \colonequals \theta(\pi).$

\item
Set $\eta_\theta \colonequals \Ind_{\pi^\ZZ \cdot \OH_D^\times}^{D^\times}(\eta_\theta')$ and Proposition 5.19 of \cite{B12} asserts that 
\begin{equation*}
H_i(\widetilde X, \overline \QQ_\ell)[\theta] \cong \eta_\theta \qquad \qquad \text{for $i = h-1$}.
\end{equation*}
\end{enumerate}

Via the local Langlands and Jacquet-Langlands correspondences, there is a bijection between smooth characters of $L^\times$ and irreducible representations of $D^\times$. For a character $\theta \from L^\times \to \overline \QQ_\ell^\times$, let $\rho_\theta$ denote the corresponding $D^\times$-representation. Theorem 2.6 of \cite{BW11} gives an explicit construction of $\rho_\theta$ in the case that $\theta$ is primitive using a geometric ingredient given by the representation $H_c^1(X_2, \overline \QQ_\ell)[\psi]$ of $U_2^{2,q}(\F)$. Note that in \cite{BW11}, $X_2$ is denoted by $X$ and $U_2^{2,q}(\F)$ is denoted by $U^{2,q}(\F)$.

Our work describes a correspondence between $L^\times$-representations and $D^\times$-representations arising in Lusztig's conjectural construction of a local analogue of Deligne-Lusztig theory. A natural question to ask is whether the map
\begin{center}
\begin{tikzpicture}[xscale=8]
\draw (0,0) node(a){\{primitive characters of $L^\times$\}} (1,0) node(b){\{irreducible representations of $D^\times$\}};
\draw[->] (a.east) to (b.west);
\draw (0,-1) node(a'){$\theta$} (1,-1) node(b'){$H_\bullet(\widetilde X, \overline \QQ_\ell)[\theta]$};
\draw[|->] (a'.east) to (b'.west);
\end{tikzpicture}
\end{center}
matches the correspondence given by the local Langlands and Jacquet-Langlands correspondences. It in fact does!

\begin{theorem}\label{t:JLCcomp}
Let $\theta \from L^\times \to \overline \QQ_\ell^\times$ be a primitive character of level $h$ and let $\rho_\theta$ be the $D^\times$-representation corresponding to $\theta$ under the local Langlands and Jacquet-Langlands correspondences. Then $H_i(\widetilde X, \overline \QQ_\ell)[\theta] = 0$ if $i \neq h-1$ and
\begin{equation*}
\rho_\theta \cong H_{h-1}(\widetilde X, \overline \QQ_\ell)[\theta].
\end{equation*}
\end{theorem}

\begin{proof}
The first assertion is clear from Theorem \ref{t:cohomdesc} and Proposition 5.19 of \cite{B12}. As the description in Theorem 2.6 of \cite{BW11} depends on the parity of $h$, we will handle the even-$h$ and odd-$h$ cases separately. In the case that $h$ is odd, the heart of the proof is really in the observation that the image of $L^\times \cdot U_D^h \cap U_D^1$ in $\UnipF$ under the surjection $U_L^1 \to \UnipF$ is exactly the group $H'(\F)$. The case when $h$ is even requires a bit more work as we must unravel the connection between the $U_2^{2,q}(\F)$-representations $H_c^1(X_2, \overline \QQ_\ell)[\psi]$ and the $\UnipF$-representations $H_c^{h-1}(X_h, \overline \QQ_\ell)[\chi]$.

Let $h$ be odd. By Theorem 2.6 of \cite{BW11}, there is a unique character $\widetilde \theta$ of $L^\times \cdot U_D^h$ that restricts to $\theta$ on $L^\times$ and is trivial on $1 + (C' \cap P_D^h)$. Here, $C' = L \cdot \Pi \subset D$. Then $\rho_\theta = \Ind_{L^\times \cdot U_D^h}^{D^\times}(\widetilde \theta)$. We would like to compare $\rho_\theta$ to the representation $\eta_\theta =H_{h-1}(\widetilde X, \overline \QQ_\ell)[\theta]$. Notice that $\pi^\ZZ \cdot \OH_D^\times = L^\times \cdot U_D^1$ so that $\eta_\theta = \Ind_{L^\times \cdot U_D^1}^{D^\times}(\eta_\theta')$.

The image of $(L^\times \cdot U_D^h) \cap U_D^1$  under the surjection
\begin{align*}
\varphi \from U_D^1 &\to \UnipF \\
1 + \sum_{i \geq 1} a_i \Pi^i &\mapsto 1 + \sum_{i=1}^{2(h-1)} a_i \tau^i
\end{align*}
is exactly equal to $H'(\F)$ and the pullback of $\chi^\sharp \from H'(\F) \to \overline \QQ_\ell^\times$ along $\varphi$ is exactly equal to $\widetilde \theta \from (L^\times \cdot U_D^h) \cap U_D^1 \to \overline \QQ_\ell^\times$. Therefore
\begin{equation*}
\Ind_{(L^\times \cdot U_D^h) \cap U_D^1}^{U_D^1}(\widetilde \theta) \cong \Ind_{(L^\times \cdot U_D^h) \cap U_D^1}^{U_D^1}(\chi^\sharp \circ \varphi),
\end{equation*}
and if follows that, viewing $H_c^{h-1}(X_h, \overline \QQ_\ell)[\chi]$ as a representation of $U_D^1$ by pulling back along $\varphi$, we have
\begin{equation*}
\Ind_{(L^\times \cdot U_D^h) \cap U_D^1}^{U_D^1}(\widetilde \theta) \cong H_c^{h-1}(X_h, \overline \QQ_\ell)[\chi].
\end{equation*}

We may identify $\OH_D^\times/U_D^{2(h-1)+1}$ with the semidirect product $\langle \zeta \rangle \ltimes \UnipF$, where $\zeta$ can be viewed as a generator of $\F^\times$. By Proposition 5.19 of \cite{B12}, we know that $H_c^{h-1}(X_h, \overline \QQ_\ell)[\chi]$ extends to a representation $\eta_\theta^\circ$ of $\OH_D^\times/U_D^{2(h-1)+1}$ which is characterized by $\Tr(\eta_\theta^\circ(\zeta)) = (-1)^{h-1} \theta(\zeta)$, where $\zeta$ is a chosen generator of $\F^\times$. We now check that
\begin{equation*}
\Ind_{(L^\times \cdot U_D^h) \cap \OH_D^\times}^{\OH_D^\times}(\widetilde \theta) \cong \eta_\theta^\circ.
\end{equation*}
It is sufficient to show that the traces agree on $\zeta$. But this is easy: The representation $\Ind_{(L^\times \cdot U_D^h) \cap \OH_D^\times}^{\OH_D^\times}(\widetilde \theta)$ is the pullback of the representation $\Ind_{\langle \zeta \rangle \ltimes H'(\F)}^{\langle \zeta \rangle \ltimes \UnipF}(\widetilde \theta)$, whose trace on $\zeta$ is exactly $\theta(\zeta)$ since any element $g \in \langle \zeta \rangle \ltimes \UnipF$ conjugates $\zeta$ out of $\langle \zeta \rangle \ltimes H'(\F)$.

Pulling back these representations to $\OH_D^\times$, we see that
\begin{equation*}
\Ind_{(L^\times \cdot U_D^h) \cap \OH_D^\times}^{\OH_D^\times}(\widetilde \theta) \cong \widetilde \eta_\theta^\circ.
\end{equation*}
It is clear that
\begin{equation*}
\Ind_{L^\times \cdot U_D^h}^{L^\times \cdot \OH_D^\times}(\widetilde \theta) \cong \eta_\theta'.
\end{equation*}
Noting that $L^\times \cdot \OH_D^\times = \pi^\ZZ \cdot \OH_D^\times$, we may now conclude that
\begin{equation*}
\rho_\theta = \Ind_{L^\times \cdot U_D^h}^{D^\times}(\widetilde \theta) = \Ind_{\pi^\ZZ \cdot \OH_D^\times}^{D^\times} \Ind_{L^\times \cdot U_D^h}^{\pi^\ZZ \cdot \OH_D^\times} (\widetilde \theta) \cong \Ind_{\pi^\ZZ \cdot \OH_D^\times}^{D^\times}(\eta_\theta') \cong H_{h-1}(\widetilde X, \overline \QQ_\ell)[\theta].
\end{equation*}

Now let $h$ be even. By Theorem 2.6 of \cite{BW11}, there is an irreducible representation $\sigma$ of $L^\times \cdot U_D^{h-1}$ such that $\Tr \sigma(x) = (-1) \cdot \theta(x)$ for each very regular element $x \in \OH_L^\times$ and the restriction of $\sigma$ to $K^\times \cdot U_L^1 \cdot U_D^h$ is a direct sum of copies of a character that equals $\theta$ on $K^\times \cdot U_L^1$ and is trivial on $1 + (C' \cap P_D^h)$. Then $\rho_\theta = \Ind_{L^\times \cdot U_D^{h-1}}^{D^\times}(\sigma)$. Just as in the odd-$h$ case, we would like to compare $\rho_\theta$ to the representation $\eta_\theta = H_{h-1}(\widetilde X, \overline \QQ_\ell)[\theta]$. 

The image of $(L^\times \cdot U_D^{h-1}) \cap U_D^1$ under the surjection $\varphi \from U_D^1 \to \UnipF$ is equal to $H''(\F)$, where
\begin{equation*}
H'' \colonequals \{1 +\sum a_i \tau^i : \text{$i$ is even; or $i \geq h-1$}\} \subset \Unip.
\end{equation*}
Note that $H''(\F)$ contains $H'(\F)$ as a degree-$q$ subgroup. 

By the proof of Theorem 2.6 of \cite{BW11}, $\sigma$ is constructed as follows. Consider the group $J = 1 + P_L^{h-1} + (C' \cap P_D^{h-1})$ and $J_+ = 1 + P_L^h + (C' \cap P_D^{h+1})$. Then we have an isomorphism $J/J_+ \cong U_2^{2,q}(\F)$ coming from the natural surjection $L^\times \ltimes J \to \langle \zeta \rangle \ltimes U_2^{2,q}(\F)$. Consider pullback of $H_c^1(X_2, \overline \QQ_\ell)[\psi]$ to $L^\times \ltimes J$ and tensor this representation with $\theta$ to obtain a representation that descends to a representation $\sigma$ of $L^\times \cdot U_D^{h-1}$. The representation $H_c^1(X_2, \overline \QQ_\ell)[\psi]$ is constructed as follows. Let $\widetilde \psi$ be any extension of $\psi$ to $\{1 + a\tau + b \tau^2 : a \in \FF_q\} \subset U_2^{2,q}(\F)$. Then $H_c^1(X_2, \overline \QQ_\ell)[\psi] \cong \Ind^{U_2^{2,q}(\F)}(\widetilde \psi)$ as representations of $U_2^{2,q}(\F)$. By Theorem 2.9 of \cite{BW11}, $\Tr(\sigma(\zeta)) = -\theta(\zeta)$.

We can realize $U_2^{2,q}(\F)$ as a subgroup of $\UnipF$ via the inclusion
\begin{align*}
U_2^{2,q}(\F) &\to \UnipF \\
1 + a_{h-1} \tau + a_{2(h-1)} \tau^2 &\mapsto 1 + a_{h-1} \tau^{h-1} + a_{2(h-1)} \tau^{2(h-1)}.
\end{align*}
Thus $H_c^1(X_2, \overline \QQ_\ell)[\psi] \cong \Ind_{\{1 + a \tau^{h-1} + b \tau^{2(h-1)} : a \in \FF_q\}}^{\{1 + a\tau^{h-1} + b \tau^{2(h-1)}\}}(\widetilde \psi)$ and as representations of $(L^\times \cdot U_D^{h-1}) \cap U_D^1$, $\sigma \cong \Ind_{H'(\F)}^{H''(\F)}(\widetilde \chi)$. Therefore,
\begin{align*}
\Ind_{(L^\times \cdot U_D^{h-1}) \cap U_D^1}^{U_D^1}(\sigma) 
&\cong \Ind_{\varphi^{-1}(H''(\F))}^{U_D^1}\Ind_{\varphi^{-1}(H'(\F))}^{\varphi^{-1}(H''(\F))}(\widetilde \chi) \\
&= \Ind_{\varphi^{-1}(H'(\F))}^{U_D^1}(\widetilde \chi) \\
&\cong H_c^{h-1}(X_h, \overline \QQ_\ell)[\chi].
\end{align*}
By Proposition 5.19 of \cite{B12}, there exists a unique extension of $H_c^{h-1}(X_h, \overline \QQ_\ell)[\chi]$ to a representation of $\OH_D^\times$ characterized by $\Tr(\zeta ; H_c^{h-1}(X_h, \overline \QQ_\ell)[\chi]) = (-1)^{h-1}\theta(\zeta) = -\theta(\zeta).$ This therefore implies that as representations of $\OH_D^\times,$
\begin{equation*}
\Ind_{(L^\times \cdot U_D^{h-1}) \cap \OH_D^\times}^{\OH_D^\times}(\sigma) \cong \widetilde \eta_\theta^\circ.
\end{equation*}
The final conclusion is exactly the same as the argument in the $h$-odd case, and this completes the proof of Theorem \ref{t:JLCcomp}.
\end{proof}


\begin{thebibliography}{WW99}
\bibitem[B12]{B12}
Boyarchenko, M. \textit{Deligne-Lusztig Constructions for Unipotent and $p$-Adic Groups.} Preprint, arXiv:1207.5876, 2012.

\bibitem[BW11]{BW11} Boyarchenko, M. and Weinstein, J. \textit{Maximal Varieties and the Local Langlands Correspondence for $\GL_n$.} Preprint, arXiv:1109.3522, version 3, 2013.

\bibitem[BW13]{BW13} Boyarchenko, M. and Weinstein, J. \textit{Geometric Realization of Special Cases of Local Langlands and Jacquet-Langlands Correspondences.} Preprint, arXiv:1303.5795, 2013.

\bibitem[DL76]{DL76} Deligne, P. and Lusztig, G. \textit{Representations of reductive groups over finite fields.} {Ann.\ of Math.\ (2),} 103(1):103-161, 1976.

\bibitem[I13]{I13} Ivanov, A. \textit{Cohomology of affine Deligne-Lusztig varieties for $\GL_2$}. {J.\ of Algebra 383}, 42-62, 2013.

\bibitem[L79]{L79} Lusztig, G. \textit{Some remarks on the supercuspidal representations of $p$-adic semisimple groups.} In {Automorphic forms, representations and L-functions (Proc.\ Sympos.\ Pure Math., Oregon State Univ., Corvallis, Ore., 1977), Part 1}, Proc.\ Sympos.\ Pure Math., XXXIII, pages 171-175. Amer.\ Math.\ Soc., Providence, R.I., 1979.
\end{thebibliography}
\end{document}